\documentclass{amsart}
\usepackage[english]{babel}
\usepackage{amsfonts}
\usepackage{amsmath,amsthm,amssymb,latexsym,mathabx}
\usepackage{color}
\usepackage{authblk}
\usepackage{graphicx}
\usepackage{tikz}
\usepackage{stmaryrd}
\usepackage{mathrsfs,upgreek}
\usepackage{eucal}
\usepackage{enumerate}
\usepackage{changes}
\usepackage{amsaddr}
\usepackage[hmargin = 3cm,vmargin = 3cm]{geometry}

\newtheorem{theorem}{Theorem}[section]
\newtheorem{lemma}[theorem]{Lemma}
\newtheorem{proposition}[theorem]{Proposition}
\newtheorem{corollary}[theorem]{Corollary}

\theoremstyle{definition}
\newtheorem{definition}[theorem]{Definition}
\newtheorem{example}[theorem]{Example}

\newtheorem{notation}{Notation}
\newtheorem{remark}[theorem]{Remark}

\numberwithin{equation}{section}
\usepackage[english]{babel}
\usepackage{amsfonts}
\usepackage{amsmath,amsthm,amssymb}
\usepackage{color}
\usepackage{graphicx}
\usepackage{tikz}
\usepackage{stmaryrd}
\usepackage{mathrsfs,upgreek}
\usepackage{eucal}
\usepackage{enumerate}
\usepackage{changes}

\newcommand{\op}[1]{\textrm{\upshape #1}}

\newcommand{\join}{\vee}

\newcommand{\meet}{\wedge}

\newcommand{\la}{\langle}
\newcommand{\ra}{\rangle}
\newcommand{\alg}[1]{{\textbf{\upshape #1}}}  
\newcommand{\vv}[1]{\mathsf {#1}}

\renewcommand{\a}{\alpha}
\renewcommand{\b}{\beta}
\renewcommand{\d}{\delta}
\newcommand{\f}{\varphi}
\newcommand{\g}{\gamma}
\newcommand{\e}{\varepsilon}
\renewcommand{\th}{\theta}
\renewcommand{\o}{\omega}

\newcommand{\cg}{\vartheta}

\newcommand{\sse}{\subseteq}

\newcommand{\app}{\approx}
\newcommand{\HH}{{\mathbf H}}  
\newcommand{\II}{{\mathbf I}} 
\newcommand{\SU}{{\mathbf S}} 
\newcommand{\PP}{{\mathbf P}}   
\newcommand{\VV}{{\mathbf V}}   
\newcommand{\QQ}{{\mathbf Q}}

\newcommand{\MTL}{\vv M\vv T\vv L}

\newcommand{\ib}{\item[$\bullet$]}

\newcommand{\Con}[1]{\operatorname{Con}(\alg #1)}

\newcommand{\vuc}[2]{#1_1,\dots,#1_{#2}}

\newcommand{\imp}{\rightarrow}

\newcommand{\lr}{ {\slash}}
\newcommand{\rr}{ {\backslash}}
\newcommand{\cc}[1]{\mathcal{#1}}

\begin{document}
\title[Structural and universal completeness in algebra and logic]{Structural and universal completeness\\  in algebra and logic}

\author{Paolo Aglian\`{o}}
\address{DIISM, Universit\`a di Siena, Siena, Italy}
\email{agliano@live.com}

\author{Sara Ugolini}
\address{Artificial Intelligence Research Institute (IIIA), CSIC, Barcelona, Spain}
\email{sara@iiia.csic.es}

\begin{abstract}
	In this work we study the notions of structural and universal completeness both from the algebraic and logical point of view. In particular, we provide new algebraic characterizations of quasivarieties that are actively and passively universally complete, and passively structurally complete. We apply these general results to varieties of bounded lattices and to quasivarieties related to substructural logics. In particular we show that a substructural logic satisfying weakening is passively structurally complete if and only if every classical contradiction is explosive in it. Moreover, we fully characterize the passive structurally complete varieties of $\mathsf{MTL}$-algebras, i.e., bounded commutative integral residuated lattices generated by chains.
\end{abstract}
\maketitle

\section{Introduction}
The main aim of this paper is to explore some connections between algebra and logic; mainly, we try to produce some {\em bridge theorems}. A bridge theorem is a statement connecting logical (and mostly synctactical) features of deductive systems and properties of classes of algebras; this connection is usually performed using the tools of general algebra and the rich theory that is behind it. The main reason behind this kind of exploration is in the further understanding one can gain by connecting two apparently distant fields. In this way, we can explore logical properties in purely algebraic terms; at the same time statements can be imported from logic  that have an important and often new algebraic meaning.

The set of logical problems we want to explore is connected with the concept of {\em structural completeness} of a deductive system, in the different ways it can be declined. For a deductive system, being structurally complete means that each of its proper extensions admits new theorems. This notion can be formalized in a more rigorous way, using the concept of {\em admissible rule}. A rule is admissible in a logic if, whenever there is a substitution making its premises a theorem, such substitution also makes the conclusion a theorem. A logic is then structurally complete if all its admissible rules are derivable in the system. It is well-known that classical logic is structurally complete; intuitionistic logic is not but it satisfies a weaker although still interesting notion: it is {\em passively} structurally complete.  We will see that this is not just a feature of intuitionism but it can be explained in a much more general framework, and it is connected to the way the contradictions of classical logic are treated. In more details, passive structural completeness means that all rules that do not apply to theorems are derivable. Naturally, the dual notion of {\em active} structural completeness also arises, which instead isolates the derivability of those rules for which there exists a substitution making their premises a theorem. The latter notion has been explored in generality in \cite{DzikStronkowski2016}. Structural completeness and its hereditary version have been deeply studied in the literature: e.g., in general algebraic terms in \cite{Bergman1991}, in substructural logics in \cite{OlsonRafteryVanAlten2008}, in fuzzy logics in \cite{CintulaMetcalfe2009}, in intermediate logics in \cite{Citkin1978}. 

A natural extension of this kind of problems is to consider {\em clauses} instead of rules. A clause is a formal pair $\Sigma \Rightarrow \Delta$, where both $\Sigma$ and $\Delta$ are finite sets of formulas over a suitable language. A clause is then admissible if a substitution making all the formulas in $\Sigma$ into theorems makes at least one of the formulas in $\Delta$ a theorem. Likewise, a clause is derivable if at least one of the formulas in $\Delta$ is derivable from $\Sigma$. A logic is {\em universally complete} if every admissible clause is derivable in it. It is then also possible to investigate the situation in which admissible clauses are active or passive in a deductive system, and thus the corresponding notions of universal completeness. Universal completeness in connection to admissible clauses has been studied in \cite{CabrerMetcalfe2015a}.

The way in which our bridge theorems will be created exploits the machinery of the so-called Blok-Pigozzi connection \cite{BlokPigozzi1989}. Without going into details, this machinery allows us to express purely logical concepts in an algebraic language. The advantage of doing so is evident: on one hand we can use the entire wealth of results about classes of algebras and various algebraic operators. On the other hand, very often by mean of this translation one ends up with algebraic results that are interesting in their own nature, irregardless of their logical origin.

While not every logical system admits this translation, many interesting and/or classical systems do: classical and intuitionistic logic, relevance logics, substructural logics in general, many-valued logics, many modal logics and so on. In this framework, one can translate the previously described notions of structural and universal completeness into properties of the quasiequational or universal theory of a quasivariety of algebras. In this setting, we will rephrase the notions of interest not in terms of formulas, but in terms of equations in a suitable language.

In this manuscript our aim is twofold; on one side we will try to describe in a complete and organic way (as much as it is possible) the phenomena mentioned above and the relations among them. In particular, we will recall the existing results trying to put them in a coherent perspective, which we believe is currently lacking, and we will provide many examples. On the other side, we will provide new results and novel characterizations of those notions that are missing an effective algebraic description. More specifically, we will first show how the characterization of active structural completeness in \cite{DzikStronkowski2016} can be extended to describe active universal completeness. Moreover, we will give algebraic descriptions of the notions of passive universal and structural completeness and the latter will result in an effective characterization. As a particularly interesting consequence, we show that a substructural logic satisfying the weakening rule is passively structurally complete if and only if every contradiction of classical logic is explosive in it. This generalizes and explains the passive structural completeness of intuitionistic logic. Moreover, it entails that all substructural logics (with weakening) with the {\em Glivenko property} with respect to classical logic are passively structurally complete. Further specializing the general result, we build on it to provide a clear characterization (and an axiomatization) of the minimal passive structurally complete logic that is an axiomatic extension of the t-norm based logic $\mathcal{MTL}$. From the algebraic side, this means that we characterize the passive structurally complete quasivarieties of bounded commutative integral residuated lattices generated by chains.

The techniques we will employ in our study are the ones proper of general algebra. In particular, we will use the understanding of algebraic objects such as projective and exact algebras. The same objects are known to be relevant for the algebraic study of unification problems in algebraizable logics \cite{Ghilardi1997}. In fact, we will show how the notion of unifiability of a set of formulas (or, equivalently, a set of equations) plays a major role in our results.

The structure of this manuscript is as follows. In the next section we will discuss the needed preliminary notions. In particular, the Blok-Pigozzi connection, projective and exact algebras, algebraic unification, and finally, we define the notions of structural and universal completeness. Section \ref{univquasi} is devoted to universal completeness, and Section \ref{structprim} to structural completeness, both in their various declinations. The last section is devoted to a deeper understanding of some relevant examples from the realms of algebra and (algebraic) logic respectively. In particular, in Subsection \ref{sec:lattices} we apply our results to the variety of (bounded) lattices; finally, in Subsection \ref{sec:FL}, we prove the aforementioned results and more about substructural logics.

\section{Preliminaries}\label{sec: start}
\subsection{Universal algebra and the Blok-Pigozzi connection}\label{subsec:universal} Let $\vv K$ be a class of algebras; we denote by $\II,\HH,\PP,\SU,\PP_u$ the class operators sending $\vv K$ in the class of all isomorphic copies, homomorphic images, direct products, subalgebras and ultraproducts of members of $\vv K$. The operators can be composed in the obvious way; for instance $\SU\PP(\vv K)$ denotes all algebras that are embeddable in a direct product of members of $\vv K$; moreover there are relations among the classes resulting from applying operators in a specific orders, for instance $\PP\SU(\vv K) \sse \SU\PP(\vv K)$ and $\HH\SU\PP(\vv K)$ is the largest class we can obtain composing the operators. We will use all the known relations without further notice, but the reader can consult \cite{Pigozzi1972} or \cite{BurrisSanka} for a textbook treatment.

 If $\rho$ is a type of algebras, an {\em equation} is a pair $p,q$ of $\rho$-terms (i.e. elements of the absolutely free algebra $\alg T_\rho(\o)$) that we write suggestively as $p \app q$; a {\em universal sentence} or {\em clause} in $\rho$ is a formal pair $(\Sigma, \Gamma)$ that we write as  $\Sigma \Rightarrow \Gamma$, where $\Sigma,\Gamma$ are finite sets of equations; a universal sentence is  a {\em quasiequation} if $|\Gamma| = 1$ and it is is {\em negative} if $\Delta= \emptyset$. Clearly an  equation is a quasiequation in which $\Sigma = \emptyset$.
 
Given any set of variables $X$, an assignment of $X$ into an algebra $A$ of type $\rho$ is a function $h$ mapping each variable $x \in X$ to an element of $\alg A$, that extends (uniquely) to a homomorphism (that we shall also call $h$) from the term algebra $\alg T_\rho(\o)$ to $\alg A$. 
An algebra $\alg A$ satisfies an equation $p \app q$ with an assignment $h$ (and we write $\alg A, h \models p \approx q$) if $h(p) = h(q)$ in $\alg A$.  An equation $p \app q$ is {\em valid} in $\alg A$ (and we write $\alg A \vDash p \approx q$) if for all assignments $h$ in $\alg A$, $\alg A, h \models p \approx q$; if $\Sigma$ is a set of equations then
 $\alg A \vDash \Sigma$ if $\alg A \vDash \sigma$ for all $\sigma \in \Sigma$.
 A universal sentence is {\em valid} in $\alg A$ (and we write $\alg A \vDash \Sigma \Rightarrow \Delta$) if for all assignments $h$ to $\alg A$, $h(p) = h(q)$ for all $p \approx q \in \Sigma$ implies that there is an identity $s \app t \in \Delta$ with
 $h(s) = h(t)$; in other words a universal sentence can be understood as the formula $\forall \mathbf x(\bigwedge \Sigma \imp \bigvee \Delta)$.  An equation or a universal sentence  is {\em valid} in a class $\vv K$ if it is valid in all  algebras in $\vv K$.

A class of algebras is a variety if it is closed under $\HH, \SU$ and $\PP$,  a quasivariety if it is closed under $\II$,$\SU$,$\PP$ and $\PP_u$ and a universal class if it is closed under $\II\SU\PP_u$.
 The following facts were essentially discovered by A. Tarski , J. \L\`os and A. Lyndon in the pioneering phase of model theory; for proof of this and similar statements the reader can consult \cite{ChangKeisler}.

 \begin{lemma}\label{lemma:ISP}
 Let $\vv K$ be any class of algebras. Then:
 \begin{enumerate}
 \item  $\vv K$ is a universal class if and only if $\II\SU\PP_u(\vv K) = \vv K$ if and only if it is the class of algebras in which a set  of universal sentences is valid;
\item  $\vv K$ is a quasivariety if and only if $\II\SU\PP\PP_u(\vv K) = \vv K$ if and only if it is the class of algebras in which a set  of quasiequations  is valid;
\item $\vv K$ is a variety if and only if $\HH\SU\PP(\vv K) = \vv K$ if and only if it is the class of algebras in which a set  of equations  is valid.
\end{enumerate}
  \end{lemma}
\begin{notation}
	We will often write $\VV$ for $\HH\SU\PP$ and $\QQ$ for $\II\SU\PP\PP_u$.
\end{notation}
For the definition of free algebras in a class $\vv K$ on a set $X$ of generators, in symbols $\alg F_\vv K(X)$, we refer to \cite{BurrisSanka}; we merely observe that every free algebra on a class $\vv K$ belongs to
$\II\SU\PP(\vv K)$. It follows that every free algebra in $\vv K$ is free in $\II\SU\PP(\vv K)$ and therefore for any quasivariety $\vv Q$, $\alg F_\vv Q(X) = \alg F_{\VV(\vv Q)}(X)$.

There are two fundamental results that we will be using many times and deserve a spotlight.
Let $\alg B, (\alg A_i)_{i \in I}$ be algebras in the same signature; we say that $\alg B$ {\em embeds} in $\prod_{i \in I} \alg A_i$  if $\alg B \in \II\SU(\prod_{i\in I} \alg A_i)$.
Let $p_i$ be the $i$-th projection, or better, the composition of the embedding and the $i$-th projection, from $\alg B$ to $\alg A_i$; the embedding is
{\em subdirect} if for all $i \in I$, $p_i(\alg B) = \alg A_i$ and in this case we will write
$$
\alg B \le_{sd} \prod_{i \in I} \alg A_i.
$$
An algebra $\alg B$ is {\em subdirectly irreducible} if it is nontrivial and for any subdirect embedding
$$
\alg B \le_{sd} \prod_{i \in I} \alg A_i
$$
there is an $i \in I$ such that $\alg B$ and $\alg A_i$ are isomorphic.  It can be shown that $\alg A$ is {\em subdirectly irreducible} if and only if the congruence lattice $\Con A$ of  $\alg A$ has a unique minimal element different from the trivial congruence. If $\vv V$ is a variety we denote by $\vv V_{si}$ the class of subdirectly irreducible algebras in $\vv V$.

\begin{theorem}\label{birkhoff} \begin{enumerate}
 \item (Birkhoff \cite{Birkhoff1944}) Every algebra can be subdirectly embedded in a product of subdirectly irreducible algebras. So if $\alg A \in \vv V$, then $\alg A$ can be subdirectly embedded in a product of members of $\vv V_{si}$.
\item (J\'onsson's Lemma \cite{Jonsson1967}) Suppose that $\vv K$ is a class of algebras such that $\vv V(\vv K)$ is congruence distributive;
then $\vv V_{si} \sse \HH\SU\PP_u(\vv K)$.
\end{enumerate}
\end{theorem}

If $\vv Q$ is a quasivariety and $\alg A \in \vv Q$, a {\em relative congruence} of $\alg A$ is a congruence $\th$ such that $\alg A/\th \in \vv Q$; relative congruences
form an algebraic lattice $\op{Con}_\vv Q(\alg A)$. 
Moreover, for an algebra $\alg A$ and a set $H \sse A\times A$ there exists the least  relative congruence $\theta_{\vv Q}(H)$ on $\alg A$ containing $H$. When $H = \{(a,b)\}$, we just write $\theta_{\vv Q}(a,b)$. When $\vv Q$ is a variety we simplify the notation by dropping the subscript $\vv Q$.

For any congruence lattice property $P$ we say that $\alg A \in \vv Q$ is {\em relative $P$} if $\op{Con}_\vv Q(\alg A)$ satisfies
$P$. So for instance $\alg A$ is {\em relative subdirectly irreducible} if $\op{Con}_\vv Q(\alg A)$ has a unique minimal element; since clearly $\op{Con}_\vv Q(\alg A)$ is a
meet subsemilattice of $\Con A$, any subdirectly irreducible algebra is relative subdirectly irreducible for any quasivariety to which it belongs. For a quasivariety $\vv Q$ we denote by
$\vv Q_{rsi}$ the class of relative subdirectly irreducible algebras in $\vv Q$.
We have the equivalent of Birkhoff's and J\'onsson's results for quasivarieties:

\begin{theorem}\label{quasivariety} Let $\vv Q$ be any quasivariety.
\begin{enumerate}
\item (Mal'cev \cite{Malcev1956}) Every $\alg A \in \vv Q$ is subdirectly embeddable in a product of algebras  in $\vv Q_{rsi}$.
\item (Czelakowski-Dziobiak \cite{CzelakowskiDziobiak1990}) If $\vv Q = \QQ(\vv K)$, then $\vv Q_{rsi} \sse \II\SU\PP_u(\vv K)$.
\end{enumerate}
\end{theorem}

The following fact will be used in the sequel.

\begin{lemma} \label{lemma: Q(A) variety} Let $\alg A$ be an  algebra, such that $\VV(\alg A)$ is congruence distributive. Then
$\QQ(\alg A) = \VV(\alg A)$ if and only if every subdirectly irreducible algebra in $\HH\SU\PP_u(\alg A)$ is  in $\II\SU\PP_u\alg A$.
\end{lemma}
\begin{proof}

Suppose first that $\QQ(\alg A) = \VV(\alg A)$, and let $\alg A$ be a subdirectly irreducible algebra in $\HH\SU\PP_u(\alg A)$. Thus $\alg A$ is subdirectly irreducible in $\VV(\alg A) = \QQ(\alg A)$, and by Theorem \ref{quasivariety} $\alg A \in \II\SU\PP_u(\alg A)$.

Conversely assume that every subdirectly irreducible algebra in $\HH\SU\PP_u(\alg A)$ is in $\II\SU\PP_u\alg A$. Since $\VV(\alg A)$ is congruence distributive, by Theorem \ref{birkhoff}(2) every subdirectly irreducible algebra in $\VV(\alg A)$ is in $\HH\SU\PP_u(\alg A)$, thus in $\II\SU\PP_u\alg A$. Now every algebra in $\VV(\alg A)$ is subdirectly embeddable in a product of subdirectly irreducible algebras in $\VV(\alg A)$ (Theorem \ref{birkhoff}(1)). Therefore, $\VV(\alg A) \sse \II\SU\PP\II\SU\PP_u(\alg A) \sse \II \SU\PP\PP_u(\alg A) = \QQ(\alg A)$ and thus equality holds.
\end{proof}

In this work we are particularly interested in quasivarieties that are the equivalent algebraic semantics of a logic in the sense of Blok-Pigozzi  \cite{BlokPigozzi1989}.
We will spend some time illustrating the machinery of {\em Abstract Algebraic Logic} that establishes a Galois connection between {\em algebraizable logics} and {\em quasivarieties of logic}, since it is relevant to understand our results. For the omitted details we refer the reader to \cite{BlokPigozzi1989,Font2016}.

By a \emph{logic} $\cc L$ in what follows we mean a substitution invariant consequence relation $\vdash$ on the set of terms $\alg T_{\rho}(\omega)$ (also called \emph{algebra of formulas}) of some algebraic language $\rho$. 
In loose terms, to establish the algebraizability of a logic $\cc L$ with respect to a quasivariety of algebras $\vv Q_{\cc L}$ over the same language $\rho$, one needs a finite set of one-variable equations $$\tau(x) =\{\d_i(x) \app \e_i(x): i = 1,\dots,n\}$$ over terms of type $\rho$ and a finite set of formulas of $\cc L$
in two variables $$\Delta(x,y)=\{\f_1(x,y),\dots,\f_m(x,y)\}$$ that allow to transform  equations, quasiequations and universal sentences in $\vv Q_{\cc L}$ into formulas, rules and clauses of $\cc L$; moreover this transformation must respect both the consequence relation of the logic and the semantical consequence of the quasivariety. More precisely, for all sets of formulas $\Gamma$ of $\cc L$ and formulas $\f \in \alg T_{\rho}(\omega)$
$$
\Gamma \vdash_{\cc L} \f\quad\text{iff}\quad \tau(\Gamma) \vDash_{\vv Q_{\cc L}} \tau(\f)
$$
where $\tau(\Gamma)$ is a shorthand for $\{\tau(\gamma): \gamma \in \Gamma\}$, and also
$$
 (x \app y) \Dashv \vDash_{\vv Q_{\cc L}}\tau(\Delta(x,y)).
$$
A quasivariety $\vv Q$ is a {\em quasivariety of logic} if it is the equivalent algebraic semantics for some logic $\cc L_\vv Q$; the Galois connection between algebraizable logics and quasivarieties of logic is given by
$$
\cc L_{\vv Q_{\cc L}} = \cc L \qquad\qquad \vv Q_{\cc L_\vv Q} = \vv Q.
$$

Not every quasivariety is a quasivariety of logic; for instance no {\em idempotent quasivariety}, such as any quasivariety of lattices, can be a quasivariety of logics. Nonetheless quasivarieties of logic are plentiful. In fact any ideal determined variety is such, as well as any quasivariety coming from a congruential variety with normal ideals (see \cite{OSV3} for details). Moreover, every quasivariety is {\em categorically equivalent} to a quasivariety of logic \cite{MoraschiniRaftery2019}. This means that if an algebraic concept is expressible through notions that are invariant under categorical equivalence, and it holds for a quasivariety $\vv Q$, then it holds for its categorically equivalent quasivariety of logic $\vv Q'$; and hence in can be transformed into a logical concept in $\cc L_{\vv Q'}$ using the Blok-Pigozzi connection.

\begin{definition}
	If $\vv Q$ is any quasivariety,
	with an abuse of notation, we will denote by $\cc L_\vv Q$ a logic whose equivalent algebraic semantics is categorically equivalent to $\vv Q$.
\end{definition}
The following result hints at what kind of properties can be transferred by categorical equivalence.
\begin{theorem}[\cite{BankstonFox1983}]\label{thm:banks}
	Let $\vv K$ be a class closed under subalgebras and direct products; If $\vv K$ is categorically equivalent to a quasivariety $\vv Q$, then $\vv K$ is a quasivariety.
\end{theorem}
Suppose now that $\vv Q$ and $\vv R$ are quasivarieties and suppose that $F:\vv Q \longrightarrow \vv R$ is a functor between the two algebraic categories witnessing the categorical equivalence.
Now, $F$ preserves all the so-called {\em categorical properties}, i.e., those notions that can be expressed as properties of morphisms. In particular, embeddings are mapped to embeddings (since in algebraic categories they are exactly the categorical monomorphisms), surjective homomorphisms are mapped to surjective homomorphisms (since they correspond to {\em regular} epimorphisms in the categories). Moreover, we observe that direct products are preserved as well, since they can be expressed via families of surjective homomorphisms (see e.g. \cite{BurrisSanka}).
Therefore, if $\vv Q'$ is a subquasivariety of $\vv Q$, then the restriction of $F$ to $\vv Q'$ witnesses a categorical equivalence between $\vv Q'$ and
$$
\vv R'=\{\alg B \in \vv R: \alg B = F(\alg A)\ \text{for some $\alg A \in \vv Q'$}\}.
$$
It follows from Theorem \ref{thm:banks} that $\vv R'$ is a subquasivariety of $\vv R$,  and that $\vv R'$ is a variety whenever $\vv Q'$ is such. Given a quasivariety $\vv Q$, we denote by $\Lambda_q(\vv Q)$ the lattice of subquasivarieties of $\vv Q$.
Hence the correspondence sending
$\vv Q' \longmapsto \vv R'$ is a lattice isomorphism between $\Lambda_q(\vv Q)$ in $\Lambda_q(\vv R)$ that preserves all the categorical properties.
Moreover, we observe that, since ultraproducts in an algebraic category admit a categorical definition which turns out to be equivalent to the algebraic one (see for instance \cite{Eklof1977}), the functor $F$ also map universal subclasses to universal subclasses; more precisely, $\vv U \sse \vv Q$ is a universal class if and only if $F(\vv U) \sse \vv R$ is a universal class.

Let us show an example of how we can use these correspondences, that is also a preview of what we will see in the coming sections; if $\vv Q$ is a quasivariety, a subquasivariety $\vv Q'$ is {\em equational} in $\vv Q$ if $\vv Q'= \HH(\vv Q') \cap \vv Q$. A quasivariety is {\em primitive} if every subquasivariety of $\vv Q$ is equational in $\vv Q$. It is clear from the discussion above that this concept is preserved by categorical equivalence and that the lattice isomorphism described above sends primitive subquasivarieties in primitive subquasivarieties.

\subsection{Projectivity, weak projectivity and exactness}
We now introduce the algebraic notions that will be the key tools for our investigation: projective, weakly projective, exact, and finitely presented algebras.
\begin{definition}
	Given a class $\vv K$ of algebras, an algebra $\alg A \in \vv K$ is {\em projective} in $\vv K$ if for all $\alg B,\alg C \in \vv K$, any homomorphism $h:\alg A \longmapsto \alg C$, and any surjective homomorphism $g: \alg B\longmapsto \alg C$, there is a homomorphism $f: \alg A \longmapsto \alg B$ such that $h=gf$.
	
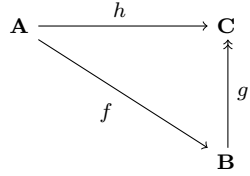
\begin{center}
\begin{tikzpicture}[scale=.9]
\node[left] at (0.5,2) {\footnotesize  $\alg A$};
\node[right] at (3,2) {\footnotesize  $\alg C$};
\node[right] at (3,0) {{\footnotesize  $\alg B$} };
\draw[->] (0.5,2)-- (3,2);
\draw[->] (0.5,1.8) -- (3,0.2);
\draw[->>](3.3,.2) -- (3.3,1.8);
\node[above] at (1.7,2) {\footnotesize $h$};
\node[below] at (1.5,1) {\footnotesize $f$};
\node[right] at (3.3,1) { \footnotesize  $g$};
\end{tikzpicture}
\end{center}
\end{definition}
Determining the projective algebras in a class is usually a challenging problem, especially in a general setting. If however $\vv K$ contains all the free algebras on $\vv K$ (in particular, if $\vv K$ is a quasivariety), projectivity admits a simpler formulation. We call an algebra $\alg B$ a {\em retract} of an algebra $\alg A$ if there is a homomorphism $g: \alg A \longmapsto \alg B$ and a homomorphism $f:\alg B \longmapsto \alg A$ with $gf= \op{id}_\alg B$ (and thus, necessarily, $f$ is injective and $g$ is surjective).
The following theorem was proved first by Whitman for lattices \cite{Whitman1941} but it is well-known to hold for any class of algebras.

\begin{theorem} Let $\vv Q$ be a quasivariety.  Then the following are equivalent:
\begin{enumerate}
\item $\alg A$ is projective in $\vv Q$;
\item $\alg A$ is a retract of  a free algebra in $\vv Q$.
\item $\alg A$ is a retract of a projective algebra in $\vv Q$.
\end{enumerate}
In particular every free algebra in $\vv Q$ is projective in $\vv Q$.
\end{theorem}

\begin{definition}
	Given a quasivariety $\vv Q$ we say that an algebra is {\em finitely presented in $\vv Q$} if there exists a finite set $X$ and a finite set $H$ of pairs of terms over $X$ such that $\alg A \cong \alg F_\vv Q(X)/\th_{\vv Q}(H)$. 
\end{definition}

The proof of the following theorem is
standard (but see \cite{Ghilardi1997}).

\begin{theorem} For a finitely presented algebra $\alg A\in\vv Q$ the following are equivalent:
\begin{enumerate}
\item  $\alg A$ is projective in $\vv Q$;
\item  $\alg A$ is projective in the class of all finitely presented algebras in $\vv Q$;
\item  $\alg A$ is a retract of a finitely generated free algebra in $\vv Q$.
\end{enumerate}
\end{theorem}

As a consequence we stress that if $\vv Q$ is a quasivariety and $\vv V = \VV(\vv Q)$ then all the algebras that are projective in $\vv Q$ are also projective in $\vv V$ (and vice versa). Moreover, all the finitely generated
projective algebras in $\vv Q$ lie inside $\QQ(\alg F_\vv Q(\o))$.
\begin{definition}
An algebra $\alg A$ is {\em weakly projective in an algebra $\alg B$} if $\alg A \in \HH(\alg B)$ implies $\alg A \in \SU(\alg B)$; an algebra is \emph{weakly projective in a class $\vv K$} if it is weakly projective in any algebra $\alg B \in \vv K$.
\end{definition}
\begin{definition}
If $\vv Q$ is a quasivariety of algebras and  $\alg A \in \vv Q$, let $G_\alg A$ be the set of generators of $\alg A$;  $\alg A$ is {\em exact} in $\vv Q$ if it is weakly projective in some $\alg F_\vv Q(X)$ with $|X| \ge |G_\alg A|$.
\end{definition}
 Clearly any projective algebra in $\vv Q$  is  weakly projective in $\vv Q$  and any weakly projective algebra in $\vv Q$  is exact in $\vv Q$.

Observe also the following consequence of the definition.
\begin{lemma} Let $\vv Q$ be a quasivariety and  let $\alg A$ be a finitely generated algebra in $\vv Q$; then the following are equivalent:
\begin{enumerate}
\item $\alg A$ is exact in $\vv Q$;
\item $\alg A \in \SU(\alg F_\vv Q(\omega))$.
\end{enumerate}
\end{lemma}
Therefore for finitely generated algebras our definition of exactness coincides with the one in \cite{CabrerMetcalfe2015}.
We close this subsection with a couple of results connecting projectivity and weak projectivity.
\begin{proposition}\label{lemma: wp implies p} Let $\alg A$ be a finite subdirectly irreducible algebra; if $\alg A$ is weakly projective in $\QQ(\alg A)$, then it is projective in $\QQ(\alg A)$.
\end{proposition}
\begin{proof} Let $\vv Q= \QQ(\alg A)$; since $\alg A$ is finite, $\vv Q$ is locally finite. Let $\alg F$ be a finitely generated (hence finite) free algebra in $\vv Q$ such that $\alg A \in \HH(\alg F)$; since $\alg A$ is weakly projective, $\alg A$ is embeddable in $\alg F$ and without loss of generality we may assume that $\alg A \le \alg F$.
Consider the set
$$
 V =\{\a \in \op{Con}_\vv Q(\alg F):  \a \cap A^2 = 0_\alg A\},
$$
where we denote by $0_{\alg A}$ the minimal congruence of $\alg A$.
It is easy to see that $V$ is an inductive poset so we may apply Zorn's Lemma to find a maximal congruence $\th \in V$.  Clearly $a \longmapsto a/\th$ is an embedding of $\alg A$ into $\alg F/\th$. We claim that $\alg F/\th$ is relative subdirectly irreducible and to prove so, since everything is finite, it is enough to show that $\th$ is meet irreducible in $\op{Con}_\vv Q(\alg F)$;
so let $\a,\b \in \op{Con}_\vv Q(\alg A)$ such that $\a \meet \beta = \th$.  Then
$$
0_ \alg A = \th \cap A^2 = (\a \meet \b) \cap A^2 = (\a \cap A^2) \meet (\b \cap A^2);
$$
But $\alg A$ is subdirectly irreducible, so $0_A$ is meet irreducible in $\op{Con}(\alg A)$; hence either $\a \cap A^2 = 0_\alg A$ or $\b \cap A^2 = 0_\alg A$, so either $\a \in V$ or $\b \in V$. Since $\th$ is maximal in $V$,
either $\a = \th$ or $\b = \th$, which proves that $\alg F/\th$ is relative subdirectly irreducible. Therefore, by Theorem \ref{quasivariety}(2), $\alg F/\th \in \II\SU(\alg A)$; since $\alg F/\th$ and $\alg A$ are both finite and
each one is embeddable in the other, they are in fact isomorphic. Thus $\alg A \le \alg F$, and there is a homomorphism from $\alg F$ onto $\alg A$ that maps each $a \in A$ to itself. This shows that $\alg A$ is a retract of $\alg F$, and therefore $\alg A$ is projective in $\QQ(\alg A)$.
\end{proof}

For varieties we have to add the hypothesis of congruence distributivity, since the use of Theorem \ref{birkhoff}(2) is paramount; for the very similar proof see \cite[Theorem 9]{JipsenNation2022}.

\begin{proposition}\label{lemma: wp implies p in varieties} Let $\alg A$ be a finite subdirectly irreducible algebra such that $\VV(\alg A)$ is congruence distributive; if $\alg A$ is weakly projective in $\VV(\alg A)$, then it is projective in $\VV(\alg A)$.
\end{proposition}

We observe that in algebraic categories projectivity is a property  preserved by categorical
equivalence and the same holds for weak projectivity and exactness. Finally by \cite{GabrielUllmer1971} being finitely presented and being finitely generated are also categorical properties preserved by equivalences.

\subsection{Algebraic unification}
The main objects of our study, i.e., the notions of universal and structural completeness, are closely related to unification problems.
The classical syntactic unification problem given two term $s,t$ finds a {\em unifier} for them; that is, a uniform replacement of the
variables occurring in $s$ and $t$ by other terms that makes $s$ and $t$ identical. When the
latter syntactical identity is replaced by equality modulo a given equational theory
$E$, one speaks of {\em $E$-unification}. 
S. Ghilardi \cite{Ghilardi1997} proved that there is a completely algebraic way of studying ($E$-)unification problems in varieties of logic, which makes use of
finitely presented and projective algebras and thus is invariant under categorical equivalence.

Let us discuss Ghilardi's idea in some detail showing how it can be applied to quasivarieties.  If $\vv Q$ is a quasivariety and $\Sigma$ is a finite set of equations in the variables $X=\{\vuc xn\}$ by a {\em substitution} $\sigma$ we mean an assignment from
$X$ to $\alg F_\vv Q(\o)$, extending to a homomorphism from $\alg F_\vv Q(X)$ to $\alg F_\vv Q(\o)$. 
\begin{definition}
A {\em unification problem} for a quasivariety $\vv Q$ is a finite set of identities $\Sigma$ in the language of $\vv Q$;
	$\Sigma$ is {\em unifiable} in $\vv Q$ if there is a substitution $\sigma$ such that $\vv Q \vDash \sigma(\Sigma)$, i.e.
$$
\vv Q\vDash p(\sigma(x_1),\dots,\sigma (x_n)) \app q(\sigma(x_1),\dots,\sigma (x_n))
$$
for all $p \app q \in \Sigma$. The substitution $\sigma$ is called a {\em unifier} for $\Sigma$.
\end{definition}
Observe that $\Sigma$ is {\em unifiable} in $\vv Q$ if and only if it is unifiable in $\VV(\vv Q)$. Let us now present the algebraic approach, where a unification problem can be represented by a finitely presented algebra in $\vv Q$.
\begin{definition}
	If $\alg A$ is in $\vv Q$, a {\em unifier} for $\alg A$  is  a homomorphism $u: \alg A \longrightarrow \alg P$ where $\alg P$ is a projective algebra in $\vv Q$; we say that an algebra is {\em unifiable in $\vv Q$} if at least one such homomorphism exists. A quasivariety $\vv Q$ is {\em unifiable} if every finitely presented algebra in $\vv Q$ is unifiable.
\end{definition}

\begin{notation}
	 When we write $\alg F_{\vv Q}(X)/\th_{\vv Q}(\Sigma)$, $\th_{\vv Q}(\Sigma)$ is the relative congruence generated in $\alg F_\vv Q(X)$ by the set $\{(p,q): p \app q \in \Sigma\}$.
\end{notation}

The following summarizes the needed results of \cite{Ghilardi1997} applied to quasivarieties.
\begin{theorem}\label{Ghilardi} Let $\vv Q$ be a quasivariety, and let $\Sigma$ be a finite set of equations in the language of $\vv Q$ with variables in a (finite) set $X$; then:
\begin{enumerate}
\item if $\Sigma$ is unifiable via $\sigma: \alg F_\vv Q(X) \to \alg F_\vv Q(Y)$ then $u_\sigma: \alg F_{\vv Q}(X)/\th_{\vv Q}(\Sigma) \to \alg F_\vv Q(Y)$ defined by
$$
u_\sigma(t/\theta_{\vv Q}(\Sigma)) = \sigma(t)
$$
is a unifier for $\alg F_{\vv Q}(X)/\th_{\vv Q}(\Sigma)$;
\item conversely let  $\alg A = \alg F_\vv Q(X)/\theta_{\vv Q}(\Sigma)$. If there is a unifier  $u:\alg A \to \alg P$, where $\alg P$ is projective and a retract of $\alg F_\vv Q(Y)$ witnessed by an embedding $i: \alg P \to \alg F_\vv Q(Y)$,
the substitution  $$\sigma_u: x \longmapsto i(u (x/\theta_{\vv Q}(\Sigma)))$$ is a unifier for $\Sigma$ in $\vv Q$.
 \end{enumerate}
\end{theorem}
\begin{proof}
	For the first claim, consider $\sigma: \alg F_\vv Q(X) \to \alg F_\vv Q(Y)$ and the natural epimorphism $\pi_\Sigma: \alg F_\vv Q(X) \to \alg F_\vv Q(X)/\theta_{\vv Q}(\Sigma)$. Since $\theta_{\vv Q}(\Sigma)$ is the least congruence of $\alg F_\vv Q(X)$ containing the set of pairs $S = \{(p,q): p \app q \in \Sigma\}$, and given that $S \sse \ker(\sigma)$, by the Second Homomorphism Theorem we can close the following diagram with exactly the homomorphism $u_\sigma$: 
	\begin{center}
\begin{tikzpicture}[scale=.9]
\node at (2,3) { $\alg F_{\vv Q}(X)$};
\node[left] at (6.7,3) {$\alg F_{\vv Q}(Y)$};
\node at (2,1) { $\alg F_{\vv Q}(X)/\theta_{\vv Q}(\Sigma)$};
\node[above] at (3.8,3) {\footnotesize $\sigma$};
\node at (4,1.9) {\footnotesize $u_\sigma$};
\node[below] at (1.6,2.45) {\footnotesize $\pi_{\Sigma}$};
\draw[->] (2.7,3) -- (5.2,3);
\draw[->] (2.2,1.5) -- (5.5,2.7);
\draw[->>] (2,2.7) -- (2,1.5);
\end{tikzpicture}
\end{center}
	
	The second claim is easily seen, since $\sigma_u$ is defined by a composition of homomorphism and as above the set of pairs $S = \{(p,q): p \app q \in \Sigma\}$ is contained in its kernel, which yields that $\sigma_u$ is a unifier for $\Sigma$ in $\vv Q$.
\end{proof}
\begin{corollary}\label{cor:unifiablesigma}
	A finite set of identities $\Sigma$ is unifiable in $\vv Q$ if and only if the finitely presented algebra $\alg F_{\vv Q}(X)/\th_{\vv Q}(\Sigma)$ is unifiable in $\vv Q$.
\end{corollary}
The following observation shows how to characterize unifiability in quasivarieties.

\begin{definition}
	For a quasivariety $\vv Q$, we let $\alg F_\vv Q$ be the {\em smallest free algebra}, i.e. $\alg F_\vv Q(\emptyset)$ (if there are constant operations) or else $\alg F_\vv Q(x)$.
\end{definition}
 We have the following :

\begin{lemma}\label{free} Let $\vv Q$ be a quasivariety and let $\alg A \in \vv Q$.  Then the following are equivalent:
\begin{enumerate}
\item $\alg A$ is unifiable in $\vv Q$;
\item there is a  homomorphism from $\alg A$ to $\alg F_\vv Q$.
\end{enumerate}
\end{lemma}
\begin{proof}Note that (2) trivially implies (1), since $\alg F_\vv Q$ is projective. Vice versa, if $\alg A$ is unifiable, there is a homomorphism from $\alg A$ to some projective algebra $\alg P$. Since $\alg P$ is a retract of some free algebra in $\vv Q$, and $\alg F_{\vv Q}$ is a homomorphic image of every free algebra in $\vv Q$, the claim follows.
\end{proof}

The above lemma implies for instance that  if $\alg F_\vv Q$ is trivial, then $\vv Q$ is unifiable since every algebra admits a homomorphism onto a trivial algebra. Hence, examples of unifiable algebras include lattices, groups, lattice-ordered abelian groups, residuated lattices. On the other hand, both bounded lattices and bounded residuated lattices (explored in Subsection \ref{sec:lattices} and \ref{sec:FL} respectively) are unifiable if and only if they admit a homomorphism onto the algebra (over the appropriate signature) with two elements $0$ and $1$.

We observe in passing that if $\alg A \cong \alg F_\vv Q(X)/\th_\vv Q$ is finitely presented  unifiable algebra in $\vv Q$, witnessed by a unifier $u: \alg A \longrightarrow \alg P$, then $u$ can be split into a homomorphism onto its image $u(\alg A)$, and an embedding from $u(\alg A)$ to $\alg P$.
By the Third Homomorphism Theorem there is a $\th' \in \op{Con}(\alg F_\vv Q(X))$ corresponding to the kernel of the onto homomorphism $u: \alg A \longrightarrow u(\alg A)$, $\th' \ge \th$, such that $\alg F_\vv V(X)/\th'$ embeds in $\alg P$; note that $\th'\in \op{Con}_\vv Q(\alg F_\vv Q(X))$, since $\alg P \in \vv Q$.  The diagram in Figure \ref{exact} shows that indeed  $\alg F_\vv V(X)/\th'$ is exact.

\begin{figure}[htbp]
\begin{center}
\begin{tikzpicture}[scale=.9]
\node at (2.1,2) {\footnotesize  $\alg F_\vv V(X)/\th$};
\node at (4.5,2) {\footnotesize  $\alg F_\vv V(X)/\th'$};
\node at (6.4,2) {\footnotesize $\alg P$};
\node at (8.4,2) {\footnotesize  $\alg F_\vv V(\o)$};
\draw[->] (2.9,2)-- (3.7,2);
\draw[->] (2.9,2)-- (3.6,2);
\draw[>->] (5.3,2)-- (6.2,2);
\draw[>->] (6.6,2)-- (7.8,2);
\end{tikzpicture}
\caption{}\label{exact}
\end{center}
\end{figure}

Let us now introduce the usual notion of order among unifiers.
Given two unifiers $u_1,u_2$ for $\alg A$ we say that $u_1$ is {\em less general then} $u_2$ (and we write $u_1 \preceq u_2$), if there is a homomorphism $h$ that makes the following diagram commute.

\begin{center}
\begin{tikzpicture}[scale=.9]
\node[left] at (0,2) {\footnotesize  $\alg A$};
\node[right] at (3,2) {\footnotesize  $\alg P_2$};
\node[right] at (3,0) {{\footnotesize  $\alg P_1$}};
\draw[->] (0,2)-- (3,2);
\draw[->] (0,1.8) -- (3,0.2);
\draw[->](3.3,1.8) -- (3.3,.2);
\node[above] at (1.5,2) {\footnotesize $u_2$};
\node[below] at (1.3,1) {\footnotesize $u_1$};
\node[right] at (3.3,1) { \footnotesize  $h$};
\end{tikzpicture}
\end{center}

Clearly $\preceq$ is a preordering and so the equivalence classes of the associated equivalence relation (i.e. the unifiers that are {\em equally general}) form a poset $U_\alg A$; using the maximal sets of that poset
it is possible to define a hierarchy of unification types  (see \cite{Ghilardi1997}). In particular, the unification type is \emph{unitary} if there is one maximal element, that is called {\em the most general unifier} or {\em mgu}.

\begin{definition}
	We say that a quasivariety $\vv Q$ has {\em  projective unifiers} if every finitely presented unifiable algebra in $\vv Q$ is projective, and that it has {\em exact unifiers} if every finitely presented unifiable algebra in $\vv Q$
is exact.
\end{definition}
 If $\vv Q$ has projective unifiers, then (from the algebraic perspective) the identity map is a unifier, and it is also the most general unifier.
Next we have a lemma whose proof is straightforward (modulo Lemma \ref{free}).

\begin{lemma} \label{projuni} Let $\vv Q$ be a quasivariety; then the following are equivalent:
\begin{enumerate}
\item $\vv  Q$ has projective (exact) unifiers;
\item  for any finitely presented $\alg A\in\vv Q$, $\alg A$ has $\alg F_\vv Q$ as a homomorphic image if
and only if $\alg A$ is projective (exact).
\end{enumerate}
\end{lemma}

If $\vv Q$ is locally finite, then we have a necessary and sufficient condition.

\begin{lemma}\label{tame} Let $\vv Q$ be a locally finite quasivariety of finite type, then the following are equivalent:
\begin{enumerate}
\item $\vv Q$ has projective unifiers;
\item  every finite unifiable algebra in $\vv Q$ is projective in the class of finite algebras in $\vv Q$.
\end{enumerate}
\end{lemma}
\begin{proof} (1) implies (2) is obvious. Assume (2), let $\alg A$ be unifiable and  finite and let $\alg B \in \vv Q$ such that
$f: \alg B \longrightarrow \alg A$ is a onto homomorphism. Let $\vuc an$ be the generators of $\alg A$ and let $\vuc bn \in B$ with
$f(b_i) = a_i$ for $i = 1 \ldots n$; if $\alg B'$ is the subalgebra generated by $\vuc bn$ then $f$ restricted to $\alg B'$ is onto. Clearly $\alg B'$ is finite. Hence by hypothesis there exists a $g: \alg A \longrightarrow \alg B$ such that $fg$ is the identity on $\alg A$.   This shows that $\alg A$ is projective in $\alg B$ and hence in $\vv Q$.  Thus (1) holds.
\end{proof}

Having exact unifiers is weaker than having projective unifiers:
\begin{example}\label{ex: distlattice}
	The variety $\vv D$ of distributive lattices is unifiable since it has no constants and it is idempotent; hence its least free algebra is trivial. But $\vv D$ does note have projective unifiers: a distributive lattice is projective if and only if the meet of join irreducible elements is again join irreducible \cite{Balbes1967}, so there are finite non projective distributive lattices.  However every finitely presented (i.e. finite) distributive lattice is exact \cite{CabrerMetcalfe2015a}.
\end{example}

  \begin{example}
  	A different  example is  the variety of $\mathsf{ST}$ of Stone algebras; a Stone algebra is a pseudocomplemented bounded distributive lattice in the signature $(\land, \lor, *, 0, 1)$ such that $x^* \lor x^{**} \approx 1$ holds. A Stone algebra is unifiable if and only if is has a homomorphism into the two element Boolean algebra if and only if it is nontrivial.
While there are nontrivial Stone algebras that are not projective, any nontrivial Stone algebra is exact (\cite[Lemma 17]{CabrerMetcalfe2015a}). Hence $\mathsf{ST}$ has exact unifiers.
  \end{example}

Moreover, there are examples of varieties having a most general unifier that do not have projective unifiers.

\begin{example}
	From the results in \cite{Ghilardi1999}, the variety $\mathsf{SH}$ of Stonean Heyting algebras (that is, Heyting algebras such that $\neg x \lor \neg\neg x \approx 1$ holds) is such that every unifiable algebra $\alg A \in \mathsf{SH}$ has a most general unifier. However, $\mathsf{SH}$  does not have projective unifiers. The algebra $\alg F_\mathsf{SH}(x,y,z)/\th$, where $\th$ is the congruence generated by the pair $( \neg x \imp (y \join z),1)$,  is unifiable but not projective. We observe that Ghilardi's argument relies heavily on some properties of Heyting algebras and uses Kripke models, making it difficult to generalize.
\end{example}

Trivial examples show that having projective or exact unifiers is not inherited in general by subvarieties (see for instance \cite[Example 7.2]{DzikStronkowski2016}). The following lemma (that we extract from \cite[Lemma 5.4]{DzikStronkowski2016}) gives a sufficient condition for having projective
unifiers. We write a detailed proof for the reader's convenience.
 \begin{lemma}[\cite{DzikStronkowski2016}]\label{lemma:DZlemma}
 Let $\vv Q$ be a quasivariety and let $\vv Q'$ be a subquasivariety of $\vv Q$ such that if $\alg B =\alg F_{\vv Q'}(X)/\th_{\vv Q'}(\Sigma)$ is finitely presented and unifiable in $\vv Q'$, then $\alg A = \alg F_{\vv Q}(X)/\th_\vv Q(\Sigma)$ is unifiable in $\vv Q$. If $\vv Q$ has projective unifiers then $\vv Q'$ has projective unifiers.
 \end{lemma}
 \begin{proof}  It is an easy exercise in general algebra to show that if $\Theta = \bigcap\{\th \in \op{Con} (\alg F_\vv Q(X)): \alg F_\vv Q(X)/\th \in \vv Q'\}$ then
$$
\alg F_{\vv Q'}(X)/\th_{\vv Q'}(\Sigma) \cong  \alg F_\vv Q(X)/(\th_\vv Q(\Sigma) \join \Theta).
$$
It follows that $\alg B$ is a homomorphic image of $\alg A$ via the natural surjection
$$
p:a/\th_\vv Q(\Sigma) \longmapsto a/(\th_\vv Q(\Sigma) \join \Theta)
$$
composed with the isomorphism.
Moreover if $f: \alg A \longrightarrow \alg C$ is a homomorphism
and $\alg C \in \vv Q'$, then $\op{ker}(p) \le \op{ker}(f)$ and by the Second Homomorphism Theorem there is a $f': \alg B \longrightarrow \alg C$ with $f'p = f$.

Now let $\alg B = \alg F_{\vv Q'}/\th_{\vv Q'}(\Sigma)$  be finitely presented and unifiable and let $\alg A = \alg F_{\vv Q}(X)/\th_\vv Q(\Sigma)$; then $\alg A$ is finitely presented and unifiable as well, so, since $\vv Q$ has projective unifiers, $\alg A$ is projective in $\vv Q$. We now show that $\alg B$ is projective. Suppose there are algebras $\alg C, \alg D \in \vv Q' \sse \vv Q$ and homomorphisms $h: \alg B \to \alg D, g: \alg C \to \alg D$ with $g$ surjective. Then, there is a homomorphism $h  p: \alg A \to \alg D$, and since $\alg A$ is projective by the definition of projectivity there is a homomorphism $f: \alg A \to \alg C$ such that $gf = hp$. Factoring $f$ as above, there is $f'$ such that $f'p = f$. Therefore since $g f'p = gf  = hp$ and $p$ is surjective, we get that $gf' = h$ which means that $\alg B$ is projective in $\vv Q'$.
\end{proof}

We will see later in Section \ref{subsec: active universal} (Example \ref{ex: De Morgan}) that Lemma \ref{lemma:DZlemma} does not hold with ``projective unifiers'' replaced by ``exact unifiers''.
We can build on the previous lemma and obtain the following.
\begin{lemma}\label{lemma:dzik2} Suppose that $\vv Q$ is a quasivariety such that $\alg F_\vv Q = \alg F_{\vv Q'}$  for all $\vv Q'\sse \vv Q$.  If $\vv Q$ has projective unifiers, then every subquasivariety $\vv Q'$ has projective unifiers.
\end{lemma}
\begin{proof}  Let $\vv Q'$ be a subquasivariety of $\vv Q$, let  $\alg B = \alg F_{\vv Q'}(X)/\th(\Sigma)$ be finitely presented and unifiable in $\vv Q'$ and let $\alg A =\alg F_\vv Q(X)/\th(\Sigma)$.
Then $\alg B$ is a homomorphic image of $\alg A$ and, since $\alg B$ is unifiable there is a homomorphism from $\alg B$ to $\alg F_{\vv Q'} = \alg F_\vv Q$. Hence $\alg A$ is unifiable as well; hence the hypothesis
of Lemma \ref{lemma:DZlemma} are satisfied, and so $\vv Q'$ has projective unifiers.
\end{proof}

We close this subsection with a corollary appearing also in \cite{DzikStronkowski2016} that is useful to some examples we will explore in what follows. We reproduce the easy proof for the reader's convenience.

\begin{corollary}\label{cor: VQ pu Q pu} Let $\vv Q$ be a quasivariety and let  $\VV(\vv Q)= \vv V$;  if $\vv V$ has exact (projective) unifiers, then so does $\vv Q$.
\end{corollary}
\begin{proof} First recall that $\vv Q$ and $\vv V$ have the same free algebras. Let $\alg A = \alg F_\vv Q(X)/\th_\vv V(\Sigma)$ and $\alg B = \alg F_\vv Q(X)/\th_\vv Q(\Sigma)$; if $\alg B$ is unifiable then, as $\alg B$ is a homomorphic image of $\alg A$ via the epimorphism $p$ described in the proof of Lemma \ref{lemma:DZlemma}, $\alg A$ is unifiable as well hence it is exact. Therefore there is an embedding $u: \alg A \longrightarrow \alg F_\vv Q(\o)$;
then by (the proof of) Lemma \ref{lemma:DZlemma} there is a $g:\alg B \longrightarrow \alg F_\vv Q(\o)$ with $gp = u$. Since $u$ is injective, so is $p$ and hence $\alg A$ and $\alg B$ are isomorphic. This proves the thesis.
\end{proof}

\subsection{Structural and universal completeness}\label{subsec:structuraluniversal}
We now introduce the main notions of interest of this work, that is, structural and universal completeness.

Let $\cc L$ be a logic with consequence relation $\vdash$.
We call {\em clause} of $\cc L$ an ordered pair $(\Sigma,\Gamma)$ where $\Sigma,\Gamma$ are finite sets of formulas. We usually write a clause as $\Sigma \Rightarrow \Gamma$. A {\em rule} is a clause $\Sigma \Rightarrow \Gamma$ where $\Gamma = \{\g\}$.
A rule is {\em admissible} in a logic if, when added to its calculus, it does not produce new theorems. More precisely:
\begin{definition}
	A clause $\Sigma \Rightarrow \Delta$ is {\em admissible} in a logic $\cc L$ if every substitution making all the formulas in $\Sigma$ a theorem, also makes at least one of the formulas in $\Delta$ a theorem.
\end{definition}
Moreover, we say that a clause $\Sigma \Rightarrow \Delta$ is {\em derivable} in a logic $\cc L$ if $\Sigma \vdash \delta$ for some $\delta \in \Delta$.
An admissible clause is not necessarily derivable; a popular example is Harrop's rule for intuitionistic logic
 $$
 \{ \neg p \imp (q \join r)\} \ \Rightarrow \  \{(\neg p \imp q) \join (\neg p \imp r)\}
 $$
 which is admissible but not derivable.
\begin{definition}
Let $\cc L$ be a logic.
	A clause $\Sigma \Rightarrow \Gamma$ is {\em passive} in $\vv Q$ if there is no substitution making the premises $\Sigma$ a theorem of $\cc L$; a clause is {\em active} otherwise. Finally, a clause $\Sigma \Rightarrow \Delta$
is  {\em negative} if $\Delta = \emptyset$.
\end{definition}
 We observe that every admissible negative clause is necessarily passive.
\begin{definition}
	A logic is said to be
 \begin{enumerate}
 \ib {\em universally complete} if every admissible clause is derivable;
 \ib  {\em structurally complete}  if every admissible rule is derivable;
 \ib {\em actively universally complete}  if every active admissible clause is derivable;
 \ib {\em actively structurally complete}  if every active admissible rule is derivable\footnote{Logics with this property have been more often called {\em almost structurally complete} but here we follow A. Citkin's advice (see \cite[footnote 2, page 8]{DzikStronkowski2016}).}
 \ib {\em passively universally complete} if every passive admissible clause is derivable;
 \ib  {\em passively structurally complete}  if every passive admissible rule is derivable;
  \ib {\em non negatively universally complete}  if every non negative admissible clause is derivable.
 \end{enumerate}
\end{definition}
Modulo algebraizability, one obtains the corresponding notions for a quasivariety. In particular, we can express admissibility and derivability of clauses in $\cc L_\vv Q$ using the (quasi)equational logic of $\vv Q$; this is because the Blok-Pigozzi Galois connection transforms (sets of) formulas in $\cc L_\vv Q$ into (sets of) equations in $\vv Q$ in a uniform way. The obtained notions make sense for quasivarieties that do not necessarily correspond to a logic.
\begin{definition}
	Let $\vv Q$ be a quasivariety. A universal sentence $\Sigma \Rightarrow \Delta$ is {\em admissible} in $\vv Q$ if every substitution unifying all the identities in $\Sigma$ also unifies at least one of the identities in $\Delta$. A universal sentence is {\em passive} if there is no substitution unifying its premises, {\em active} otherwise.
$\vv Q$ is {\em (active/passive) universally/structurally complete } if every (active/passive) admissible universal sentence/quasiequation is valid in $\vv Q$. 
\end{definition}

If $P$ is one of those properties, then we say that a logic (or a quasivariety) is {\em hereditarily $P$} if the logic (or the quasivariety) and all its extensions have the property $P$.
Some of these properties are well-known to be distinct: for instance  classical logic is non-negative universally  complete but not universally complete, while intuitionistic logic is not structurally complete (thanks to Harrop's example) but
 it is passively structurally complete (as reported by Wronski in 2005, see \cite{CintulaMetcalfe2009}). The following is a consequence of algebraizability.
\begin{theorem}\label{admissiblederivable} Let $\vv Q$ be a quasivariety of logic, $\Sigma,\Delta$  sets of equations in the language of $\vv Q$ and $\Sigma',\Delta'$ the corresponding sets of formulas in $\cc L_\vv Q$. Then:
\begin{enumerate}
\item $\Sigma' \Rightarrow \Delta'$ is admissible in $\cc L_\vv Q$ if and only if $\Sigma \Rightarrow \Delta$ is admissible in $\vv Q$;
\item $\Sigma' \Rightarrow \Delta'$ is derivable in $\cc L_\vv Q$ if and only if $\vv Q \vDash \Sigma \Rightarrow \Delta$.
\end{enumerate}
\end{theorem}
Moreover, by Corollary \ref{cor:unifiablesigma} we get the following.
\begin{proposition}\label{activepassive} Let $\vv Q$ be a quasivariety of logic, $\Sigma,\Delta$ sets of equations in the language of $\vv Q$ and $\Sigma',\Delta'$ the corresponding sets of formulas in $\cc L_{\vv Q}$. Then:
\begin{enumerate}
\item $\Sigma'\Rightarrow \Delta'$ is active in $\cc L_{\vv Q}$ if and only if $\alg F_\vv Q(X)/\th_{\vv Q}(\Sigma)$ is unifiable in $\vv Q$;
\item $\Sigma'\Rightarrow \Delta'$ is passive in $\cc L_{\vv Q}$ if and only if $\alg F_\vv Q(X)/\th_{\vv Q}(\Sigma)$ is not unifiable in $\vv Q$.
\end{enumerate}
\end{proposition}
The next lemma (also derivable from \cite[Theorem 2]{CabrerMetcalfe2015a}) characterizes admissibility of universal sentences.
\begin{lemma}\label{admissible} Let $\vv Q$ be any quasivariety, let $\Sigma \Rightarrow \Delta$ be a clause in the language of $\vv Q$ and let $\vv U_{\Sigma\Rightarrow\Delta} = \{\alg A \in \vv Q: \alg A \vDash \Sigma\Rightarrow\Delta\}$. Then the following are equivalent:
 \begin{enumerate}
  \item $\Sigma \Rightarrow \Delta$ is admissible in $\vv Q$;
 \item $\alg F_\vv Q(\o) \vDash \Sigma \Rightarrow \Delta$;
 \item $\HH(\vv Q) = \HH(\vv U_{\Sigma \Rightarrow \Delta})$.
 \end{enumerate}
 \end{lemma}
\begin{proof}
The equivalence between (1) and (2) follows directly from the definition of admissibility.
Assume now $\alg F_\vv Q(\o) \vDash \Sigma \Rightarrow \Delta$, then $\alg F_\vv Q(\o) \in \vv U_{\Sigma \Rightarrow \Delta}$. Clearly
$\HH\SU\PP_u(\alg F_\vv Q(\o)) \sse \HH(\vv U_{\Sigma \Rightarrow \Delta}) \sse \HH(\vv Q)$.  Now every algebra is embeddable in an ultraproduct of its finitely generated subalgebras and every finitely generated algebra is a homomorphic image of $\alg F_\vv Q(\o)$. Therefore if $\alg A \in \vv Q$, then
$\alg A \in \SU\PP_u\HH(\alg F_\vv Q(\o)) \sse \HH\SU\PP_u(\alg F_\vv Q(\o))$. So $\HH\SU\PP_u(\alg F_\vv Q (\o)) = \HH(\vv Q)$ and thus  (3) holds.

Conversely assume (3).
Since $\alg F_{\vv Q}(\o) \in \HH(\vv Q) = \HH(\vv U_{\Sigma \Rightarrow \Delta})$, there is $\alg A \in \vv U_{\Sigma \Rightarrow \Delta}$ such that $\alg F_{\vv Q}(\o) \in \HH(\alg A)$. Since $\alg F_{\vv Q}(\o)$ is projective in $\vv Q$, it follows that $\alg F_{\vv Q}(\o) \in \SU(\alg A) \sse \SU(\vv U_{\Sigma \Rightarrow \Delta}) \sse \vv U_{\Sigma \Rightarrow \Delta}$. Therefore, $\alg F_{\vv Q}(\o) \models \Sigma \Rightarrow \Delta$ and (2) holds.
\end{proof}

To conclude the preliminaries, we present the following lemma which will be particularly useful in our proofs.
\begin{lemma}\label{lemma:FThetaSigma}
	Let $\vv Q$ be a quasivariety, and $\Sigma, \Delta$ be finite sets of equations over variables in a finite set $X$. The following are equivalent:
	\begin{enumerate}
		\item $\vv Q \models \Sigma \Rightarrow \Delta$;
		\item $\alg F_{\vv Q}(X)/\theta_{\vv Q}(\Sigma) \models \Sigma \Rightarrow \Delta$;
		\item there is $p \approx q \in \Delta$ such that $p / \theta_{\vv Q}(\Sigma) = q / \theta_{\vv Q}(\Sigma)$ in $\alg F_{\vv Q}(X)/\theta_{\vv Q}(\Sigma)$.
	\end{enumerate}
\end{lemma}
\begin{proof}
	It is clear that (1) implies (2) and (2) implies (3). We now show that (3) implies (1). 
	
	Let $\alg A \in \vv Q$. If there is no assignment of the variables in $X$ to $\alg A$ that models $\Sigma$, then $\alg A \models \Sigma \Rightarrow \Delta$. Otherwise, suppose there is an assignment $h$ such that $\alg A, h \models \Sigma$. Then, since $\theta_{\vv Q}(\Sigma)$ is the smallest relative congruence of $\alg F_{\vv Q}(X)$ containing the set of pairs $S = \{(p,q): p \app q \in \Sigma\}$, by the Second Homomorphism Theorem we can close the following diagram:
	\begin{center}
\begin{tikzpicture}[scale=.9]
\node at (2,3) { $\alg F_{\vv Q}(X)$};
\node[left] at (6,3) {$\alg A$};
\node at (2,1) { $\alg F_{\vv Q}(X)/\theta_{\vv Q}(\Sigma)$};
\node[above] at (3.8,3) {\footnotesize $h$};
\node at (4,1.9) {\footnotesize $f$};
\node[below] at (1.6,2.45) {\footnotesize $\pi_{\Sigma}$};
\draw[->] (2.7,3) -- (5.2,3);
\draw[->] (2.2,1.5) -- (5.5,2.7);
\draw[->>] (2,2.7) -- (2,1.5);
\end{tikzpicture}
\end{center}

That is, there is a homomorphism $f: \alg F_{\vv Q}(X)/\theta_{\vv Q}(\Sigma) \to \alg A$ such that $h = f \pi_{\Sigma}$, where $\pi_\Sigma$ is the natural epimorphism from $\alg F_{\vv Q}(X)$ to $\alg F_{\vv Q}(X)/\theta_{\vv Q}(\Sigma)$. 
Now by (3) there is at least an identity $p \approx q \in \Delta$ such that 
$(p,q) \in \ker(\pi_\Sigma)$. Since $h = f\pi_\Sigma$, $(p,q) \in \ker(h)$, which means that $\alg A, h \models p \approx q$ and therefore $\alg A \models \Sigma \Rightarrow \Delta$. Since $\alg A$ is an arbitrary algebra of $\vv Q$ this shows that $\vv Q \models \Sigma \Rightarrow \Delta$.  
\end{proof}

\section{Universal completeness}\label{univquasi}
In this section we study from the algebraic perspective the notion of universal completeness and its variations: active, nonnegative, passive universal completeness, together with their hereditary versions. That is, we shall see which algebraic properties correspond to the notions coming from the logical perspective (detailed in the preliminaries Subsection \ref{subsec:structuraluniversal}). For each notion, we will present a characterization theorem and some examples. While the characterizations of active and passive universal completeness (to the best of our knowledge) are fully original, we build on existing ones for the other notions, presenting some new results and a coherent presentation in our framework.

\subsection{Universal quasivarieties}\label{univquasi1}
We start with universal completeness.
The following expands \cite[Proposition 6]{CabrerMetcalfe2015a}.
\begin{theorem}\label{prop:universal}
 For any quasivariety $\vv Q$ the following are equivalent:
\begin{enumerate}
\item\label{prop:universal2} $\vv Q$ is universally complete;
\item\label{prop:universal1} for every universal class $\vv U \sse \vv Q$,  $\HH(\vv U) = \HH(\vv Q)$ implies $\vv U = \vv Q$.
\item\label{prop:universal3} $\vv Q = \II\SU\PP_u(\alg F_\vv Q(\o))$;
\item\label{prop:universal4} every finitely presented algebra in $\vv Q$ is in $ \II\SU\PP_u(\alg F_\vv Q(\o))$.
\end{enumerate}
\end{theorem}
\begin{proof}  
(\ref{prop:universal1}) implies (\ref{prop:universal2}) via Lemma \ref{admissiblederivable}.
We show that (\ref{prop:universal2}) implies (\ref{prop:universal1}). Let $\vv U \sse \vv Q$ be a universal class such that $\HH(\vv U) = \HH (\vv Q)$  and suppose that $\vv U \vDash \Sigma \Rightarrow \Delta$;  then
 $$
 \HH(\vv Q) = \HH(\vv U) \sse \HH(\vv U_{\Sigma \Rightarrow \Delta}) \sse \HH(\vv Q).
 $$
 So $\HH(\vv U_{\Sigma \Rightarrow \Delta}) = \HH(\vv Q)$ and by Lemma \ref{admissiblederivable}  $\Sigma \Rightarrow \Delta$ is admissible in $\vv Q$. By (\ref{prop:universal2}), $\vv Q \vDash \Sigma \Rightarrow \Delta$; therefore $\vv U$ and $\vv Q$  are two universal classes in which exactly the same clauses are valid, thus they are equal. Hence (\ref{prop:universal1}) holds, and thus (\ref{prop:universal1}) and (\ref{prop:universal2}) are equivalent.
 
(\ref{prop:universal2}) implies (\ref{prop:universal3}) follows by Lemma \ref{admissiblederivable}.
Moreover, (\ref{prop:universal3}) clearly implies (\ref{prop:universal4}). 
We now show that (\ref{prop:universal4}) implies (\ref{prop:universal2}), which completes the proof.  
Consider a universal sentence $ \Sigma \Rightarrow \Delta$ that is admissible in $\vv Q$, or equivalently (by Lemma \ref{admissiblederivable}), such that $\alg F_\vv Q(\o) \models \Sigma \Rightarrow \Delta$. The finitely presented algebra 
$\alg F_\vv Q (X)/\theta_{\vv Q}(\Sigma) \in  \II\SU\PP_u(\alg F_\vv Q(\o))$ by (4), and thus $\alg F_\vv Q (X)/\theta_{\vv Q}(\Sigma) \models \Sigma \Rightarrow \Delta$. By Lemma \ref{lemma:FThetaSigma}, $\vv Q \models \Sigma \Rightarrow \Delta$ and thus $\vv Q$ is universally complete.
\end{proof}

By algebraizability, and since the property of being universal (for the discussion in Subsection \ref{subsec:universal}) is preserved by categorical equivalence, we get at once:

\begin{corollary}\label{universalq} For a quasivariety $\vv Q$ the following are equivalent:
\begin{enumerate}
\item $\vv Q$ is universally complete;
\item $\cc L_\vv Q$ is universally complete.
\end{enumerate}
\end{corollary}
The following theorem and lemma show a sufficient and a necessary condition respectively for a quasivariety to be universally complete.

\begin{theorem}\label{thm: unifiable}  If every finitely presented algebra in $\vv Q$ is exact  then $\vv Q$ is universally complete.
\end{theorem}
\begin{proof}  If every finitely presented algebra in $\vv Q$ is exact, it is in $\II\SU(\alg F_\vv Q(\o))$, and thus also in $\II\SU\PP_u(\alg F_\vv Q(\o))$. The claim then follows from Theorem \ref{prop:universal}.
\end{proof}

\begin{theorem}\label{lemma: FQ trivial} If $\vv Q$ is universally complete, then $\vv Q$ is unifiable.
\end{theorem}
\begin{proof}
Suppose by counterpositive that there is a finite set of identities $\Sigma$ that is not unifiable in $\vv Q$. Then $\Sigma \Rightarrow \emptyset$ is (passively) admissible but not derivable; indeed it does not hold in the trivial algebra. This implies that $\vv Q$ is not universally complete, and the claim is proved.
\end{proof}
Since projectivity implies exactness, we observe the following immediate consequence of Theorem \ref{thm: unifiable}.
\begin{corollary}\label{cor: unifiable}  If every finitely presented algebra in $\vv Q$ is projective then $\vv Q$ is universally complete.
\end{corollary}
For locally finite varieties there is a stronger result, observed in \cite{CabrerMetcalfe2015a}.
\begin{lemma}\label{lemma: techlemma} \cite{CabrerMetcalfe2015a} Let $\vv Q$ be a locally finite quasivariety; then
 $\alg A \in \II\SU\PP_u(\alg F_\vv Q(\o))$ if and only if every finite subalgebra $\alg B$ of $\alg A$ is in $\II\SU(\alg F_\vv Q(\o))$.
\end{lemma}

\begin{theorem}[\cite{CabrerMetcalfe2015a}]\label{thm: lf unifiable} Let $\vv Q$ be a locally finite  variety of finite type. Then $\vv Q$ is universally complete if and only if $\vv Q$ is unifiable and  has exact unifiers.
\end{theorem}
\begin{proof} Suppose that $\vv Q$ is universally complete; then, by Theorem \ref{lemma: FQ trivial},  $\vv Q$ is unifiable.
Since it is universally complete, every finite algebra in $\vv Q$ is in $\II\SU\PP_u(\alg F_\vv Q(\o))$, hence in $\II\SU(\alg F_\vv Q(\o))$ (by  Lemma \ref{lemma: techlemma}). Thus every finite unifiable algebra in $\vv Q$ is exact and $\vv Q$ has exact unifiers.
The converse claim follows from Theorem \ref{thm: unifiable}.
\end{proof}
\begin{remark}
	We observe that Theorem \ref{lemma: FQ trivial} limits greatly the examples of universally complete quasivarieties. In particular, in quasivarieties with finite type the trivial algebra is finitely presented, and thus if $\vv Q$ is universally complete, it must be unifiable. This means that a quasivariety with more than one constant in its finite type cannot be universally complete if there are nontrivial models where the constants are distinct; similarly if there is only one constant, then it must generate the trivial algebra in nontrivial models, or equivalently, in $\alg F_\vv Q$. If there are no constants, then $\alg F_\vv Q = \alg F_\vv Q(x)$ and, in order to be able to embed the trivial algebra, there has to be an idempotent term.
\end{remark}

Let us now discuss some different examples of universally complete (quasi)varieties.
\begin{example}
Let us consider {\em lattice-ordered abelian groups} (or abelian $\ell$-groups for short). These are algebras $\alg G = (G, \land, \lor, \cdot, ^{-1}, 1)$ where $(G, \cdot, ^{-1}, 1)$ is an abelian group, $(G, \land, \lor)$ is a lattice, and the group operation distributes over the lattice operations.
	Every finitely presented abelian $\ell$-groups is projective \cite{Beynon1977}; thus, the variety of abelian $\ell$-groups is universally complete by Corollary \ref{cor: unifiable}.
	
	The same holds for the variety of \emph{negative cones} of abelian $\ell$-groups. Given an $\ell$-group $\alg G$, the set of elements $G^-= \{x \in G: x \leq 1\}$ can be seen as a {\em residuated lattice} (see Section \ref{sec:FL}) $\alg G^-=(G^-, \cdot, \to, \land, \lor, 1)$ where $(\cdot, \land, \lor, 1)$ are inherited by the group and $x \to y = x^{-1} \cdot y \land 1$. The algebraic category of negative cones of abelian $\ell$-groups is equivalent to the one of abelian $\ell$-groups \cite{GalatosTsinakis2005}, thus every finitely presented algebra is projective  and the variety of negative cones of $\ell$-groups $\vv L \vv G^-$ is universally complete.
Observe that in all these cases the unique constant $1$ is absorbing w.r.t. any basic operation, and it generates the trivial algebra.
\end{example}

\begin{example}
	{\em Hoops} are a particular variety of residuated monoids related to logic which were defined in an unpublished manuscript by B\"uchi and Owens, inspired by the work of Bosbach on partially ordered monoids (see \cite{BlokFerr2000} for details on the theory of hoops).  Hoops have a constant which is absorbing w.r.t. any basic operation; hence the least free algebra is trivial in any variety of hoops and any variety of hoops is unifiable. In  \cite{AglianoUgolini2022} it was shown that every finite hoop is projective in the class of finite hoops which via Lemma \ref{tame} entails that every locally finite variety of hoops has projective unifiers. Since any locally finite quasivariety is contained in a locally finite variety, every locally finite quasivariety of hoops is universally complete. The same holds in the variety of $\imp$-subreducts of hoops, usually denoted by $\mathsf{HBCK}$; again locally finite varieties of $\mathsf{HBCK}$-algebras have projective unifiers \cite{AglianoUgolini2022} and hence they are universally complete.
For a non-locally finite example, we say that a hoop is {\em cancellative} if the underlying monoid is cancellative; cancellative hoops form a variety $\vv C$ that is categorically equivalent to the one of abelian $\ell$-groups \cite{BlokFerr2000}. Hence $\vv C$ is a non locally finite variety of hoops which is universally complete.
\end{example}
The classes of algebras in the above examples all have projective unifiers. However:
\begin{example}
In lattices there are no constants but any variety of lattices is idempotent; hence the least free algebra is trivial and every lattice is unifiable.
Every finite distributive lattice is exact \cite{CabrerMetcalfe2015a} and distributive lattices are locally finite, so distributive lattices are universally complete by Theorem \ref{thm: lf unifiable}. Moreover, as we have already observed in Example \ref{ex: distlattice}, distributive lattices do not have projective unifiers.
\end{example}
We now consider the hereditary version of universal completeness.
\begin{definition}
	A quasivariety $\vv Q$ is {\bf primitive universal} if all its subquasivarieties are universally complete.
\end{definition}

All the above  examples of universally complete varieties are primitive universal and this is not entirely coincidental.
Distributive lattices are trivially primitive universal, since they do not have any  trivial subquasivariety.  For all the other examples,  we have a general result.

\begin{theorem} Let $\vv Q$ be a quasivariety with projective unifiers and such that $\alg F_\vv Q$ is trivial; then $\vv Q$ is primitive universal.
\end{theorem}
\begin{proof} Observe that for any subquasivariety $\vv Q'\sse \vv Q$, $\alg F_{\vv Q'}$ is trivial
as well. Hence every algebra in $\vv Q$ is unifiable in any subvariety to which it belongs. Let $\alg B' = \alg F_{\vv Q'}(X)/\th_{\vv Q'}(\Sigma)$  be finitely presented in $\vv Q'$; then $\alg A = \alg F_\vv Q(X)/\th_\vv Q(\Sigma)$ is finitely presented in $\vv Q$ and thus it is projective in $\vv Q$. But then Lemma \ref{lemma:DZlemma} applies and $\alg B$ is projective; thus $\vv Q'$ has projective unifiers and thus it is universally complete by Corollary \ref{cor: unifiable}.
\end{proof}

Is the same conclusion true if we replace ``projective unifiers'' with ``unifiable, locally finite with exact unifiers''?  We do not know, but we know that  we cannot use an improved version of Lemma \ref{lemma:DZlemma} since it cannot be improved to account for exact unifiers (see Example \ref{ex: De Morgan}).

\subsection{Non-negative and active universal quasivarieties}\label{subsec: active universal}
The situation in which universal completeness fails due {\em only} to the trivial algebras has been first investigated in \cite{CabrerMetcalfe2015a};
the following expands \cite[Proposition 8]{CabrerMetcalfe2015a}.
\begin{theorem}\label{thm: non negative universal} For a quasivariety $\vv Q$ the following are equivalent:
\begin{enumerate}
\item\label{thm: non negative universal3} $\vv Q$ is non-negative universally complete;
\item\label{thm: non negative universal2} every admissible universal sentence is valid in $\vv Q^+$;
\item\label{thm: non negative universal1} every nontrivial algebra is in $\II\SU\PP_u(\alg F_\vv Q(\o))$. 
\item\label{thm: non negative universal4} every nontrivial finitely presented algebra is in $\II\SU\PP_u(\alg F_\vv Q(\o))$. 
\end{enumerate}
\end{theorem}
\begin{proof}
	The equivalence of the first three points is in \cite[Proposition 8]{CabrerMetcalfe2015a}, and (\ref{thm: non negative universal1}) clearly implies (\ref{thm: non negative universal4}). Assume now that (\ref{thm: non negative universal4}) holds, we show (\ref{thm: non negative universal3}). Let $\Sigma \Rightarrow \Delta$ be a non-negative admissible universal sentence with variables in a finite set $X$, we show that $\alg F_\vv Q(X)/\theta_\vv Q(\Sigma) \models \Sigma \Rightarrow \Delta$. If $\alg F_\vv Q(X)/\theta_\vv Q(\Sigma)$ is trivial, then it models $\Sigma \Rightarrow \Delta$ (given that $\Delta$ is not $\emptyset$). Suppose now that $\alg F_\vv Q(X)/\theta_\vv Q(\Sigma)$ is nontrivial, then it is in $\II\SU\PP_u(\alg F_\vv Q(\o))$ by hypothesis and then it models $\Sigma \Rightarrow \Delta$ since the latter is admissible, and thus $\alg F_\vv Q(\o) \models \Sigma \Rightarrow \Delta$ by Lemma \ref{admissiblederivable}. By Lemma \ref{lemma:FThetaSigma}, $\vv Q$ models $\Sigma \Rightarrow \Delta$ and (\ref{thm: non negative universal3}) holds.
\end{proof}
Moreover:
\begin{theorem} For a  quasivariety  $\vv Q$   the following are equivalent:
\begin{enumerate}
\item $\vv Q$ is non-negative universally complete;
\item $\cc L_\vv Q$ is non-negative universally complete.
\end{enumerate}
\end{theorem}
\begin{proof}
	In a categorical equivalence between quasivarieties trivial algebras are mapped to trivial algebras, since the latter can be characterized as the algebras that are a homomorphic image of every algebra in a quasivariety. Thus nontrivial finitely presented algebras are mapped to nontrivial finitely presented algebras, and the result follows from the usual arguments.
\end{proof}
We can also obtain an analogue of Theorem \ref{thm: unifiable}.

\begin{theorem}\label{thm:unifiable2} If every nontrivial finitely presented algebra in $\vv Q$ is exact (or projective), then $\vv Q$ is non-negative universally complete.
\end{theorem}
\begin{proof}
If every nontrivial finitely presented algebra is exact (or projective), then it is in $\II\SU(\alg F_\vv Q(\o))$, and therefore in $\II\SU\PP_u(\alg F_\vv Q(\o))$. The claim then follows from Theorem \ref{thm: non negative universal}.
\end{proof}

Analogously to the case of universal completeness, we get a stronger result for locally finite quasivarieties.
\begin{theorem} \label{thm: nonnegative unifiable}  Let $\vv Q$ be a locally finite quasivariety. Then $\vv Q$ is non-negative universally complete if and only if every nontrivial finitely presented algebra is exact.
\end{theorem}
\begin{proof}
Suppose that $\vv Q$ is locally finite and there is a finite nontrivial algebra $\alg A \in \vv Q$ that is not exact. Then $\alg A \notin \II\SU(\alg F_\vv Q(\o))$ and thus, by Lemma \ref{lemma: techlemma}, $\alg A \notin \II\SU\PP_u(\alg F_\vv Q(\o))$.
Therefore $\vv Q$ cannot be non-negative universally complete by Theorem \ref{thm: non negative universal}.
The other direction follows from Theorem \ref{thm:unifiable2}.
\end{proof}

\begin{example}\label{ex: boolean algebras}
	Boolean algebras are an example of a non-negative universally complete variety that is not universally complete. It is easily seen that every nontrivial finite Boolean algebra is exact (indeed, projective), which shows that Boolean algebras are non-negative universally complete by Theorem \ref{thm: nonnegative unifiable}. However, there are negative admissible clauses: e.g., the ones with premises given by the presentation of the trivial algebra, which is finitely presented but not unifiable. Thus Boolean algebras are not universally complete.
\end{example}
\begin{example}
	 Stone algebras are a different example; in \cite{CabrerMetcalfe2015a} the authors proved, using the duality between Stone algebras and particular Priestley spaces,  that every finite nontrivial Stone algebra is exact;
hence Stone algebras  are  non-negative universally complete.
\end{example}
We now move on to describe active universal completeness from the algebraic perspective.

\begin{theorem}\label{thm: activeu main}
	Let $\vv Q$ be a quasivariety. The following are equivalent:
	\begin{enumerate}
	\item\label{thm: activeu main2} $\vv Q$ is  active universally complete;
	\item\label{thm: activeu main5} every unifiable algebra in $\vv Q$ is in $\II\SU\PP_u(\alg F_{\vv Q}(\o))$;
		\item\label{thm: activeu main1} every finitely presented and unifiable algebra in $\vv Q$ is in $\II\SU\PP_u(\alg F_{\vv Q}(\o))$;
		\item\label{thm: activeu main3} every universal sentence admissible in $\vv Q$ is satisfied by all finitely presented unifiable algebras in $\vv Q$;
		\item\label{thm: activeu main4} for every $\alg A \in \vv Q$, $\alg A \times \alg F_{\vv Q}\in \II \SU \PP_u(\alg F_{\vv Q}(\omega))$.
	\end{enumerate}
\end{theorem}
\begin{proof}
	We start by showing that (\ref{thm: activeu main2}) implies (\ref{thm: activeu main5}).
	 Assume (\ref{thm: activeu main2}), and let $\Sigma \Rightarrow \Delta$ be such that $\alg F_\vv Q(\omega) \models \Sigma \Rightarrow \Delta$; equivalently, by Lemma \ref{admissiblederivable}, $\Sigma \Rightarrow \Delta$ is an admissible universal sentence in $\vv Q$. If $\Sigma$ is unifiable, by hypothesis $\Sigma \Rightarrow \Delta$ is valid in $\vv Q$. Suppose now that $\Sigma$ has variables in a finite set $X$ and it is not unifiable, that is, via Corollary \ref{cor:unifiablesigma} there is no homomorphism from $\alg F_{\vv Q}(X)/\theta_{\vv Q}(\Sigma)$ to $\alg F_{\vv Q}$. Let $\alg A$ be a unifiable algebra in $\vv Q$; we argue that there is no assignment of the variables in $\Sigma$ that validates $\Sigma$ in $\alg A$. Indeed otherwise the following diagram would commute and $\alg F_{\vv Q}(X)/\theta_{\vv Q}(\Sigma)$ would be unifiable, yielding a contradiction.
\begin{center}
\begin{tikzpicture}[scale=.9]
\node at (2,3) { $\alg F_{\vv Q}(X)$};
\node[left] at (6,3) {$\alg A$};
\node[left] at (8,3) {$\alg F_{\vv Q}$};
\node at (2,1) { $\alg F_{\vv Q}(X)/\theta_{\vv Q}(\Sigma)$};
\draw[->>] (2.7,3) -- (5.2,3);
\draw[->>] (6.1,3) -- (7.2,3);
\draw[->>] (2,2.7) -- (2,1.5);
\draw[->>] (2.2,1.5) -- (5.5,2.7);
\end{tikzpicture}
\end{center}

Therefore, $\Sigma \Rightarrow \Delta$ is vacuously satisfied in $\alg A$, which is any unifiable algebra in $\vv Q$, thus (\ref{thm: activeu main5}) holds. Now, clearly (\ref{thm: activeu main5}) implies (\ref{thm: activeu main1}), and (\ref{thm: activeu main1}) and (\ref{thm: activeu main3}) are equivalent by the definitions.

Let us show that (\ref{thm: activeu main3}) implies (\ref{thm: activeu main2}). Let $\Sigma \Rightarrow \Delta$ be an active admissible universal sentence in $\vv Q$ with variables in a finite set $X$; we want to show that it is also valid in $\vv Q$. Since by hypothesis $\Sigma \Rightarrow \Delta$ is active admissible, $\Sigma$ is unifiable, and therefore so is $\alg F_{\vv Q}(X)/\theta_{\vv Q}(\Sigma)$ by Corollary \ref{cor:unifiablesigma}. Then by (\ref{thm: activeu main3}), $\alg F_{\vv Q}(X)/\theta_{\vv Q}(\Sigma) \models \Sigma \Rightarrow \Delta$, which implies that $\vv Q \models \Sigma \Rightarrow \Delta$ by Lemma \ref{lemma:FThetaSigma}. Therefore the first four points are equivalent.

Finally, we show that (\ref{thm: activeu main2}) implies (\ref{thm: activeu main4}) and (\ref{thm: activeu main4}) implies (\ref{thm: activeu main5}), which completes the proof. We start with (\ref{thm: activeu main2}) $\Rightarrow$ (\ref{thm: activeu main4}).
Let $\alg A \in \vv Q$, and consider a clause $\Sigma \Rightarrow \Delta$ valid in $\alg F_{\vv Q}(\omega)$. We show that $\alg A \times \alg F_{\vv Q}(\omega) \models \Sigma \Rightarrow \Delta$. Now, if $\vv Q \models \Sigma \Rightarrow \Delta$, in particular $\alg A \times \alg F_{\vv Q}(\omega) \models \Sigma \Rightarrow \Delta$. Suppose that $\vv Q \not\vDash \Sigma \Rightarrow \Delta$. Since $\vv Q$ is active universally complete, $\Sigma \Rightarrow \Delta$ must be a passive rule, thus $\Sigma$ is not unifiable. Equivalently, there is no assignment $h$ of the variables in $\Sigma$ such that $\alg F_{\vv Q},h\models \Sigma$. Thus, there is also no assignment $h'$ of the variables in $\Sigma$ such that $\alg A \times \alg F_{\vv Q},h'\models \Sigma$, thus $\alg A \times \alg F_{\vv Q}(\omega) \models \Sigma \Rightarrow \Delta$.

It is left to prove (\ref{thm: activeu main4}) $\Rightarrow$ (\ref{thm: activeu main5}). 
Let $\alg A$ be a unifiable algebra in $\vv Q$, then there is a homomorphism $h: \alg A \to \alg F_\vv Q$ (Lemma \ref{free}). Consider the map $h': \alg A \to \alg A \times \alg F_{\vv Q}$ be defined as $h'(a) = (a, h(a))$. Clearly, $h'$ is an embedding of $\alg A$ into $\alg A \times \alg F_{\vv Q} \in \II \SU \PP_u(\alg F_{\vv Q}(\omega))$ (by (\ref{thm: activeu main4})). Thus also $\alg A \in \II \SU \PP_u(\alg F_{\vv Q}(\omega))$, which completes the proof.
\end{proof}
We observe that the previous characterization extends to universal sentences some of the results in \cite{DzikStronkowski2016} about active structural completeness.
We also get the usual result.
\begin{theorem} For a  quasivariety  $\vv Q$   the following are equivalent:
\begin{enumerate}
\item $\vv Q$ is active universally complete;
\item $\cc L_\vv Q$ is active universally complete.
\end{enumerate}
\end{theorem}
\begin{proof}
	The result follows from the fact that embeddings, ultraproducts, being finitely presented and unifiable, are all categorical notions and thus preserved by categorical equivalence.
\end{proof}

Moreover, we have the following lemma whose proof is the same as the one of Theorems \ref{thm: unifiable} and \ref{thm: lf unifiable}.

\begin{theorem}
\label{lemma: ex unif auc}
If $\vv Q$ has exact (or projective) unifiers, then $\vv Q$ is active universally complete.  If $\vv Q$ is also locally finite then it is active universally complete if and only if it has exact unifiers.
\end{theorem}

\begin{example}\label{ex:discriminator}
A {\em discriminator} on a set $A$ is a ternary operation $t$ on $A$ defined by
$$
t(a,b,c) =\left\{
            \begin{array}{ll}
              a, & \hbox{if $a\ne b$;} \\
              c, & \hbox{if $a=b$.}
            \end{array}
          \right.
$$
A variety $\vv V$ is a {\em discriminator variety} \cite{Pixley1971} if there is a ternary term that is the discriminator on all the subdirectly irreducible members of $\vv V$. Discriminator varieties have many strong properties: for instance they are congruence permutable and congruence distributive.

 In \cite[Theorem 3.1]{Burris1992} it has been essentially shown that discriminator varieties have  projective unifiers, and therefore they are all active universally complete by Theorem \ref{lemma: ex unif auc}.
\end{example}

\begin{example}\label{ex: BL-algebras au}
	We now see some examples within the algebraic semantics of many-valued logics; in \cite{AglianoUgolini2022} it has been shown that in any locally finite variety of bounded hoops or $\mathsf {BL}$-algebras (the equivalent algebraic semantics of H\'ajek Basic Logic \cite{Hajek98}), the finite unifiable algebras are exactly the finite projective algebras. It follows
that any of such varieties has projective unifiers and hence it is active universally complete. This holds also for any locally finite quasivariety of bounded hoops or $\mathsf {BL}$-algebras, or their reducts, i.e., bounded $\mathsf{HBCK}$-algebras.

In contrast with the case of (unbounded) hoops, not all of them are non-negative universally complete, as we will now discuss. Let us call {\em chain} a totally orderef algebra. Every finite BL-chain is an ordinal sum of finite Wajsberg hoops, the first of which is an MV-algebra \cite{AglianoMontagna2003}. No finite MV-chain different from the 2-element Boolean algebra $\alg 2$ is unifiable (they are all simple and the least free algebra is $\alg 2$), and thus not exact. It follows by basic facts about ordinal sums that if a locally finite quasivariety $\vv Q$ of BL-algebras contains a chain whose first component is different from $\alg 2$, $\vv Q$ is not non-negative universally complete. The same holds, mutatis mutandis, for bounded hoops and bounded $\mathsf{HBCK}$-algebras.
  In Section \ref{sec:FL} we shall see a different class of (discriminator) varieties coming from many-valued logics that are active universally complete.
\end{example}
\begin{definition}
	We call a quasivariety $\vv Q$  {\em active primitive universal} if every subquasivariety of $\vv Q$ is active universally complete.  
\end{definition}
It is evident from the characterization theorem of active universally complete quasivarieties that a variety $\vv Q$ is active primitive universal if and only if $\cc L_\vv Q$ is hereditarily active universally complete. We have the following fact:

\begin{theorem}\label{thm: activeprimitive} Suppose that $\vv Q$ is a quasivariety such that $\alg F_\vv Q = \alg F_{\vv Q'}$  for all $\vv Q'\sse \vv Q$. If $\vv Q$ has projective unifiers then it is active primitive universal.
\end{theorem}
\begin{proof}  The proof follows from Theorem \ref{lemma: ex unif auc} and Lemma \ref{lemma:dzik2}.
\end{proof}

All varieties in Example \ref{ex: BL-algebras au} satisfy the hypotheses of Theorem \ref{thm: activeprimitive} (as the reader can easily check). For discriminator varieties all the examples of lattice-based varieties in  Section \ref{sec:FL} of this paper (but see also \cite{Burris1992} or  \cite{Citkin2018a} for more examples) have the same property; hence they are all active primitive universal.
 
Now, a variety is {\em q-minimal} if it does not have any proper nontrivial subquasivariety; so a q-minimal variety is necessarily equationally complete.  We have this result by Bergman and McKenzie:

\begin{theorem}\label{thm: mckenziebergman}  \cite{BergmanMcKenzie1990} A locally finite equationally complete  variety is q-minimal  if and only if it has exactly one subdirectly irreducible algebra that is embeddable in any nontrivial member of the variety. Moreover, this is always the case if the variety is congruence modular.
\end{theorem}

It follows immediately that every active universally complete q-minimal variety is active primitive universal. 
\begin{example}
Discriminator varieties are active universally complete as seen in example \ref{ex:discriminator}.
	Now, given a finitely generated discriminator variety $\vv V$, it is generated by a finite algebra $\alg A$ having a discriminator term on it, also called a  {\em quasi-primal} algebra. 	By \cite{Werner1970} $\vv V$ is equationally complete and, since it is congruence modular, it is q-minimal; hence $\vv V$ is active primitive universal.
\end{example}

Finally, we observe that Lemma \ref{lemma:dzik2} cannot be improved to ``having exact unifiers'' and the counterexample is given by {\em De Morgan lattices}; we will see below that they form an active universally complete variety that is not active primitive universal.

\begin{example}\label{ex: De Morgan}
	A  De Morgan lattice is a distributive lattice with a unary operation $\neg$ which is involutive and satisfies the De Morgan Laws.  It is well-known that the variety $\mathsf{DM}$ of De Morgan lattices is locally finite and  has exactly two proper non trivial subvarieties, i.e. the variety $\mathsf{BLa}$ of Boolean lattices (axiomatized by $x \le y \join \neg y$) and the variety $\mathsf{KL}$ of Kleene lattices (axiomatized by $x \meet \neg x \le y \join \neg y$). It is easily seen that all these nontrivial varieties have the same one-generated free algebra whose universe is  $\{x, \neg x, x \join \neg x, x \meet \neg x\}$. It follows that all the subquasivarieties of De Morgan lattices have the same least free algebra and $\mathsf{DM}$ satisfies the hypotheses of Theorem \ref{thm: activeprimitive}.  Admissibility in  De Morgan lattices has been  investigated in \cite{MetcalfeRothlisberger2012} and \cite{CabrerMetcalfe2015a}. Now for
 a finite algebra $\alg A \in \mathsf{DM}$ the following are equivalent:
\begin{enumerate}
\item $\alg A$ is unifiable;
\item the universal sentence $\{x \app \neg x\} \Rightarrow \emptyset$ is valid in $\alg A$;
\item $\alg A \in \II\SU(\alg F_\vv{DM}(\o))$.
\end{enumerate}
The equivalence of (2) and (3) has been proved in \cite[Lemma 28]{CabrerMetcalfe2015a}, while (3) implies (1) trivially. If we assume that (2) does not hold for $\alg A$, then there is an $a \in A$ with $\neg a = a$; so if
$f:\alg A \longrightarrow \alg F_\vv{DM} (x)$ is a homomorphism and $f(a) = \f$, then $\f = \neg \f$. But there is no element in $\alg F_\vv{DM} (x)$ with that property, so $\alg A$ cannot be unifiable.
This concludes the proof of the equivalence of the three statements.

Therefore $\vv{DM}$ has exact unifiers and thus it is active universally complete by Theorem \ref{lemma: ex unif auc}. Now consider the subvariety of $\vv{DM}$ of Kleene lattices. In \cite{CabrerMetcalfe2015a} it is shown that the universal sentence
$$
\Phi := \{x \le \neg x, x \meet \neg y \le \neg x \join y\} \Rightarrow \neg y \le y
$$
is admissible in $\mathsf{KL}$. It is also active, as the reader can easily check that the substitution $x \longmapsto z \land \neg z$, $y \longmapsto \neg z$  unifies the premises of $\Phi$.
However it fails in the three element Kleene lattice  $\alg K_3$ in Figure \ref{Kleene}, with the assignment $x =a$, $y=\neg a$;
hence $\mathsf{KL}$ is not active universally complete. So $\vv{DM}$ is a variety that is active universally complete but not active primitive universal.

\begin{figure}[h]
\begin{center}
\begin{tikzpicture}[scale=.8]
\draw (0,0) -- (0,2);
\draw[fill] (0,0) circle [radius=0.05];
\draw[fill] (0,1) circle [radius=0.05];
\draw[fill] (0,2) circle [radius=0.05];
\node[right]  at (0,1) {$\neg b = b$};
\node[right]  at (0,0) {$\neg a$};
\node[right]  at (0,2) {$a$};
\end{tikzpicture}
\caption{The lattice $\alg K_3$ }\label{Kleene}
\end{center}
\end{figure}
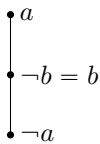
\end{example}

Note that in $\mathsf{KL}$ there must be a finite unifiable algebra that is not exact (since $\mathsf{KL}$ cannot have exact unifiers).  Now a finite Kleene lattice $\alg A$ is exact if and only if both
$\{x \app \neg x\} \Rightarrow \emptyset$ and $\Phi$ are valid in $\alg A$ \cite[Lemma 38]{CabrerMetcalfe2015a}. Let $\alg A = \alg K_3 \times \mathbf 2$; the reader can easily check that $\alg A$ is unifiable in $\mathsf{KL}$
(since it satisfies $\{x \app \neg x\} \Rightarrow \emptyset$  and hence it is unifiable in $\mathsf{DM}$) but does not satisfy $\Phi$. This shows (as promised) that Lemma \ref{lemma:DZlemma} cannot be improved.

\subsection{Passive universal quasivarieties}
We will now see that passive universal completeness in a quasivariety corresponds to an algebraic notion we have already introduced: unifiability. Moreover, we shall see that it corresponds to the apparently weaker notion of negative universal completeness, that is,
every (passive) admissible negative universal sentence is derivable. We recall that a quasivariety $\vv Q$ is unifiable if every finitely presented algebra in $\vv Q$ is unifiable.
\begin{theorem}\label{thm: passive universal}
	For every quasivariety $\vv Q$ the following are equivalent:
	\begin{enumerate}
		\item $\vv Q$ is passive universally complete;
		\item $\vv Q$ is negative universally complete;
		\item $\vv Q$ is unifiable.
	\end{enumerate}
\end{theorem}
\begin{proof}
Assume (1) and let $\Sigma \Rightarrow \emptyset$ be a negative admissible universal sentence; then it is necessarily passive, since there is no substitution that unifies $\emptyset$. Thus, by (1), $\Sigma \Rightarrow \emptyset$ is valid in $\vv Q$.
	
Assume now (2), we prove that it implies (3) by contrapositive. Suppose that $\vv Q$ is not unifiable, that is, there exists a finite set of identities $\Sigma$ that is not unifiable. Then the negative universal sentence $\Sigma \Rightarrow \emptyset$ is (passively) admissible, but it is not derivable (in particular, it fails in the trivial algebra).
	
	Finally, if (3) holds, then (1) trivially holds, since if every set of identities is unifiable there is no passive admissible clause.
\end{proof}
In some cases, we can improve the previous result.
\begin{lemma}\label{lemma:Qunifiable}
Let $\vv Q$ be a quasivariety such that $\II(\alg F_{\vv Q}) = \II\PP_u(\alg F_{\vv Q})$, then the following are equivalent.
\begin{enumerate}
		\item $\vv Q$ is unifiable;
	\item every algebra in $\vv Q$ is unifiable.
\end{enumerate}
\end{lemma}
\begin{proof}
	We prove the nontrivial direction by contraposition. Consider an arbitrary algebra $\alg A \in \vv Q$ and assume that it is not unifiable; without loss of generality we let $\alg A = \alg F_\vv Q(X)/\theta$ for some set $X$ and some relative congruence $\theta$. Since $\alg A$ is not unifiable, there is no assignment $h: \alg F_\vv Q(X) \to \alg F_\vv Q$ such that $\alg F_\vv Q, h \models \Sigma_\theta$, where $\Sigma_\theta = \{t \approx u : (t,u) \in \theta\}$. Equivalently, iff $\alg F_\vv Q \models \Sigma_\theta \Rightarrow \emptyset$. Now, the equational consequence relation relative to a class of algebras $\vv K$ is finitary if and only if $\vv K$ is closed under ultraproducts (see for instance \cite{Raftery2006}); thus by hypothesis the equational consequence relation relative to $\alg F_{\vv Q}$ is finitary, and we obtain that $\alg F_\vv Q \models \Sigma'_\theta \Rightarrow \emptyset$, for $\Sigma'_\theta$ some finite subset of $\Sigma_\theta$. That is, $\Sigma'$ is finite and not unifiable, thus $\vv Q$ is not unifiable and the proof is complete.
\end{proof}
Observe that if a quasivariety $\vv Q$ is such that $\alg F_\vv Q$ is finite, it satisfies the hypothesis of the previous lemma. 
\begin{corollary}
	Let $\vv Q$ be a quasivariety such that $\II(\alg F_{\vv Q}) = \II\PP_u(\alg F_{\vv Q})$, then the following are equivalent.
\begin{enumerate}
		\item $\vv Q$ is passive universally complete;
		\item $\vv Q$ is negative universally complete;
		\item $\vv Q$ is unifiable;
	\item every algebra in $\vv Q$ is unifiable.
\end{enumerate}
\end{corollary}
Since unifiability is preserved by categorical equivalence, we get the following.
\begin{corollary} A quasivariety $\vv Q$ is passive universally complete if and only if $\cc L_\vv Q$ is passive universally complete.
\end{corollary}

\section{Structural completeness}\label{structprim}
In this section we investigate the algebraic counterparts of structural completeness and its variations. The main new results are about the characterization of passive structurally complete quasivarieties; moreover, we also show a characterization of primitive quasivarieties grounding on the results in \cite{Gorbunov1998}.
\subsection{Structural quasivarieties}
The bridge theorems for structural completeness have been first established by Bergman \cite{Bergman1991}. We present the proof for the sake of the reader, expanding with point (6).
\begin{theorem}[\cite{Bergman1991}]\label{structural} For a quasivariety $\vv Q$ the following are equivalent:
\begin{enumerate}
\item\label{structural5} $\vv Q$ is structurally complete;
\item\label{structural4} $\vv Q = \QQ(\alg F_\vv Q(\o))$;
\item\label{structural1} no proper subquasivariety of $\vv Q$ generates a proper subvariety of $\HH(\vv Q)$;
\item\label{structural2} for all $\vv Q'\sse\vv Q$ if $\HH(\vv Q') = \HH(\vv Q)$, then $\vv Q = \vv Q'$;
\item\label{structural3} for all $\alg A \in \vv Q$ if $\VV(\alg A) = \HH(\vv Q)$, then $\QQ(\alg A) = \vv Q$;
\item\label{structural6} every finitely presented algebra in $\vv Q$ is in $\QQ(\alg F_\vv Q(\o))$.
\end{enumerate}
\end{theorem}
\begin{proof}
First, (\ref{structural5}) is equivalent to (\ref{structural4}) via Lemma \ref{admissiblederivable}.
The implications (\ref{structural1}) $\Leftrightarrow$ (\ref{structural2}) $\Rightarrow$ (\ref{structural3}) $\Rightarrow$ (\ref{structural4}) are straightforward.
(\ref{structural4}) implies (\ref{structural2}) since if $\vv Q'\sse \vv Q$ and $\HH(\vv Q') = \HH(\vv Q)$, we get that $\alg F_{\vv Q'}(\omega) = \alg F_{\HH(\vv Q')}(\omega) = \alg F_{\HH(\vv Q)}(\omega) = \alg F_{\vv Q}(\omega)$; thus $\vv Q = \QQ(\alg F_\vv Q(\o)) = \QQ(\alg F_\vv Q'(\o)) \sse \vv Q'$ and then equality holds. Thus the first five points are equivalent; Finally, clearly (\ref{structural4}) implies (\ref{structural6}), and
(\ref{structural6}) implies (\ref{structural4}) since a quasivariety is generated by its finitely presented algebras (\cite[Proposition 2.1.18]{Gorbunov1998}).
\end{proof}

\begin{corollary} A variety $\vv V$ is structurally complete if and only if every proper subquasivariety of $\vv V$ generates a proper subvariety; therefore if $\alg A$ is such that $\VV(\alg A)$ is structurally complete, then $\VV(\alg A) = \QQ(\alg A)$.
\end{corollary}

Since the definition of structural completeness is invariant under categorical equivalence we get also:

\begin{corollary} \label{lemma: SCrelation} Let $\vv Q$ be a quasivariety; then $\vv Q$ is structurally complete if and only if $\mathcal{L}_\vv Q$ is structurally complete.
\end{corollary}
Let us extract some sufficient conditions for structural completeness.

\begin{lemma}\label{lemma: wpstructcomplete}  Let $\vv Q$ be a quasivariety; if
\begin{enumerate}
\item every $\alg A\in \vv K$ is exact in $\vv Q = \QQ(\vv K)$, or
\item every finitely generated algebra in $\vv Q$ is exact, or
\item every finitely presented algebra in $\vv Q$ is exact, or
\item every  finitely generated relative subdirectly irreducible in $\vv Q$ is exact,
\end{enumerate}
then $\vv Q$ is structurally complete. Moreover if every $\alg A\in \vv K$ is  exact in  $\VV(\vv K)$ and every subdirectly irreducible member of $\VV(\vv K)$ is in $\II\SU(\vv K)$, then $\VV(\vv K)$ is structurally complete.
\end{lemma}
\begin{proof} If each algebra in $\vv K$ is exact in $\vv Q = \QQ(\vv K)$, then $\vv K \sse\II\SU(\alg F_\vv Q(\o))$; therefore $\vv Q =\QQ(\vv K) \sse \QQ(\alg F_\vv Q(\o))$ and thus equality holds. Hence $\vv Q$ is structurally complete by the characterization theorem. The other points follow. 

 For the last claim, every subdirectly irreducible member of $\VV(\vv K)$ lies in $\II\SU(\vv K)$ and thus is exact in $\VV(\vv K)$. Since any variety is generated as a quasivariety by its subdirectly irreducible members, $\VV(\vv K)$ is structurally complete.
\end{proof}

We observe that none of the previous conditions is necessary. For locally finite quasivarieties we have a necessary and sufficient condition for structural completeness because of the following:

\begin{lemma}[\cite{CabrerMetcalfe2015a}]\label{lemma: lf ISP} Let $\vv Q$ be a locally finite quasivariety and $\alg A$ a finite algebra in $\vv Q$. Then $\alg A \in \QQ(\alg F_\vv Q(\o))$ if and only if $\alg A \in \II\SU\PP(\alg F_\vv Q(\o))$.
\end{lemma}
The following theorem improves \cite[Corollary 11]{CabrerMetcalfe2015a}.
\begin{theorem}\label{thm: structexact} For a locally finite quasivariety $\vv Q$ of finite type the following are equivalent:
\begin{enumerate}
\item $\vv Q$ is structurally complete;
\item each finite algebra in $\vv Q$ is in $\II\SU\PP(\alg F_\vv Q(\o))$;
\item every finite  relative subdirectly irreducible in $\vv Q$ is exact.
\end{enumerate}
\end{theorem}
\begin{proof} Assume (1); then each finite algebra in $\vv Q$ is in $\QQ(\alg F_\vv Q(\o))$ and thus, by Lemma \ref{lemma: lf ISP}, is in $\II\SU\PP(\alg F_\vv Q(\o))$ and (2) holds. If (2) holds and
 $\alg A$ is finite relative subdirectly irreducible, then it is in $\II\SU(\alg F_\vv Q(\o))$, i.e. it is exact.  Finally if (3) holds, then $\vv Q$ is structurally complete by Lemma \ref{lemma: wpstructcomplete}.
\end{proof}

\subsection{Primitive quasivarieties}\label{subsec: prim quasivarieties}
We now consider the hereditary notion of structural completeness.
\begin{definition}
	A class of algebras $\vv K$ in a quasivariety $\vv Q$ is {\em equational relative to} $\vv Q$ if $\vv K = \VV(\vv K)\cap \vv Q$.
In particular, a subquasivariety $\vv Q'$ of $\vv Q$ is {\em equational relative to} $\vv Q$ if $\vv Q' = \HH(\vv Q')\cap \vv Q$; a quasivariety $\vv Q$ is {\em primitive} if every subquasivariety of $\vv Q$ is equational relative to $\vv Q$. 
\end{definition}
Clearly  primitivity is downward hereditary and a variety $\vv V$ is primitive if and only if every subquasivariety of $\vv V$ is a variety. 
We can  show the following.
\begin{theorem}\label{primitiveQ} For a quasivariety $\vv Q$ the following are equivalent:
\begin{enumerate}
\item $\vv Q$ is primitive;
\item every subquasivariety of $\vv Q$ is structurally complete;
\item for all subdirectly irreducible $\alg A \in \HH(\vv Q)$ and for any $\alg B \in \vv Q$, if $\alg A \in \HH(\alg B)$, then
$\alg A \in \II\SU\PP_u(\alg B)$.
\end{enumerate}
\end{theorem}
\begin{proof} We first show the equivalence between (1) and (2). Suppose that $\vv Q$ is primitive and let $\vv Q'\sse \vv Q$; if $\vv Q'' \sse \vv Q'$ and $\HH(\vv Q'') = \HH(\vv Q')$ then
$$
\vv Q' = \HH(\vv Q') \cap \vv Q = \HH(\vv Q'') \cap \vv Q = \vv Q''
$$
so $\vv Q'$ is structurally complete by Theorem \ref{structural}.

Conversely assume (2), let $\vv Q' \sse \vv Q$ and let $\vv Q'' = \HH(\vv Q') \cap \vv Q$ (it is clearly a quasivariety); then $\HH(\vv Q'') = \HH(\vv Q')$ and thus $\vv Q'' = \vv Q'$, again using the characterization of Theorem \ref{structural}. So
$\vv Q'$ is equational in $\vv Q$ and $\vv Q$ is primitive.

Assume (1) again, and let $\alg A,\alg B \in \vv Q$ with $\alg A$ subdirectly irreducible and $\alg A \in \HH(\alg B)$. Since $\vv Q$ is primitive we have
$$
\QQ(\alg B) = \HH(\QQ(\alg B)) \cap \vv Q
$$
and hence $\alg A \in \QQ(\alg B)$. Since $\alg A$ is subdirectly irreducible, $\alg A \in \II\SU\PP_u(\alg B)$ by Theorem \ref{quasivariety} and (3) holds.

Conversely, assume (3) and let $\vv Q'$ be a subquasivariety of $\vv Q$; if $\alg B \in \HH(\vv Q') \cap \vv Q$,
observe that $\alg B \in \HH(\vv Q)$ and hence $\alg B \le_{sd} \prod \alg A_i$ where the $\alg A_i$ are subdirectly irreducible in $\HH(\vv Q) \cap \HH(\vv Q')$. Then for all $i$ there is $\alg B_i \in \vv Q'$ such that $\alg A_i \in \HH(\alg B_i)$ and hence by hypothesis
$\alg A_i \in \SU\PP_u(\alg B_i)$ and so $\alg A_i \in \vv Q'$ for all $i$. Therefore $\alg B \in \vv Q'$, so $\HH(\vv Q') \cap \vv Q =\vv Q'$ and $\vv Q'$ is equational in $\vv Q$. Therefore $\vv Q$ is primitive and (1) holds.
\end{proof}

As commented in the preliminary section (Subsection \ref{subsec:universal}), primitivity is preserved under categorical equivalence, and therefore:
\begin{corollary}
A quasivariety is primitive if and only if $\cc L_\vv Q$ is hereditarily structurally complete.
\end{corollary}

We will see how Theorem \ref{primitiveQ} can be improved in the locally finite case. Let $\vv Q$ be a quasivariety and let $\alg A \in \vv Q$; we define
$$
[\vv Q:\alg A] =\{\alg B \in \vv Q: \alg A \notin \II\SU(\alg B)\}.
$$

The following lemma describes some properties of $[\vv Q:\alg A]$; the proofs are quite standard with the exception of point (3). As a matter of fact a proof of the forward implication of (3) appears
in \cite[Corollary 2.1.17]{Gorbunov1998}. However the proof is somewhat buried into generality and it is not easy to follow; so we felt that a suitable  translation would make it easier for the readers.

\begin{lemma} \label{universal} Let $\vv Q$ be a quasivariety; then
\begin{enumerate}
\item if $\alg A \in \vv Q$ is finite and $\vv Q$ has finite type, then $[\vv Q:\alg A]$ is a universal class;
\item if $\alg A$ is relative subdirectly irreducible and finitely presented, then  $[\vv Q:\alg A]$ is a quasivariety;
\item $\alg A$ is weakly projective in $\vv Q$ if and only if $[\vv Q:\alg A]$ is closed under $\HH$ if and only if $[\vv Q:\alg A]$ is equational relative to $\vv Q$;
 \item if  $\alg A$ is relative subdirectly irreducible, finitely presented and weakly projective in $\vv Q$, then  $[\vv Q:\alg A]$ is a variety.
 \end{enumerate}
 Moreover if $\vv Q$ is locally finite of finite type, the converse implications in (1),(2) and (4) hold.
\end{lemma}
\begin{proof} For (1), if $\alg A$ is finite, then there is a first order universal sentence $\Psi$ such that, for all $\alg B \in \vv Q$, $\alg B \vDash \Psi$ if and only if
$\alg A \in \II\SU(\alg B)$. 
More in detail, if $|A|=n$, $$\Psi: = \exists x_1 \ldots \exists x_n (\&\{ x_i \neq x_j: i,j \leq n, i \neq j\} \,\,\&\,\, \alg D(\alg A)),$$ where $\alg D(\alg A)$ is the diagram of $\alg A$, that is, a conjunction of universal sentences that describe the operation tables
of $\alg A$ (identifying each element of $\alg A$ with a different $x_i$), and $\&$ is first order logic conjunction.   

Consider $\alg B \in \II\SU\PP_u([\vv Q:\alg A])$, we show that $\alg A \notin \II\SU(\alg B)$; if $\alg A \in \II\SU(\alg B)$, then $\alg A \in \II\SU\PP_u([\vv Q:\alg A])$. Hence there exists a family $(\alg A_i)_{i\in I} \sse [\vv Q: \alg A]$ and
an ultrafilter $U$ on $I$ such that $\alg C = \Pi_{i\in I}\alg A/U$ and $\alg A \in \II\SU(\alg C)$.  So  $\alg C \vDash \Psi$;
 but then by \L{\`o}s Lemma there is a (necessarily nonempty) set of indexes $I' \in U$ such that  $\Psi$ is valid in each $\alg A_i$ with $i \in I'$, which is clearly a contradiction, since each $\alg A_i \in [\vv Q:\alg A]$.
Thus $\alg A \notin \II\SU(\alg B)$ and $\alg B \in [\vv Q:\alg A]$ and  therefore $\II\SU\PP_u([\vv Q:\alg A]) = [\vv Q:\alg A]$ which is a universal class by Lemma \ref{lemma:ISP}

Conversely let $\vv Q$ be locally finite of finite type; every algebra in $\vv Q$ is embeddable in an ultraproduct of  its finitely generated (i.e. finite) subalgebras, say $\alg A \in \II\SU\PP_u(\{\alg B_i: i\in I\})$. If $\alg A$ is not finite, then $\alg A \notin \SU(\alg B_i)$ for all $i$, so $\alg B_i \in [\vv Q:\alg A]$ for all $i$. Since $[\vv Q:\alg A]$ is universal, we would have that $\alg A \in [\vv Q:\alg A]$, a clear contradiction. So $\alg A \in \II\SU(\alg B_i)$ for some $i$ and hence it is finite.

For (2), suppose that $\alg A$ is relative subdirectly irreducible and finitely presented, i.e.  $\alg A \cong \alg F_\vv Q(\mathbf x)/\th(\Sigma)$ where $\mathbf x = (\vuc xn)$ and $\Sigma = \{ p_i(\mathbf x) \app q_i(\mathbf x): i = 1,\dots,m\}$.
We set $a_i = x_i/\th(\Sigma)$; since $\alg A$ is relative subdirectly irreducible, it has a relative monolith $\mu$, i.e. a minimal non trivial relative congruence. Since $\mu$ is minimal, there are $c,d \in A$ such that
$\mu$ is the relative congruence generated by the pair $(c,d)$. Now let $t_c,t_d$ terms in $\alg F_\vv Q(\mathbf x)$ such that $t_c(\vuc an) = c$ and $t_d(\vuc an)=d$ and let $\Phi$ be the quasiequation
\begin{equation*}
	\bigwedge_{i=1}^m p_i(\mathbf x) \app q_i(\mathbf x) \  \longrightarrow \ t_c(\mathbf x) \app  t_d(\mathbf x).
\end{equation*}
Then $\alg A \not\vDash \Phi$; moreover if $\alg C \in \vv Q$ is a homomorphic image of $\alg A$ which is not isomorphic with $\alg A$, then $\alg C \vDash \Phi$.  We claim that $[\vv Q:\alg A] =\{\alg B \in \vv Q: \alg B \vDash \Phi\}$ and since $\Phi$ is a quasiequation this implies that $[\vv Q:\alg A]$ is a quasivariety.  Clearly if $\alg B \vDash \Phi$, then $\alg A \notin \II\SU(\alg B)$; conversely assume that $\alg B \not\vDash \Phi$. Then
there are $\vuc bn \in B$ such that $p_i(\vuc bn) = q_i(\vuc bn)$ but $t_c(\vuc bn) \ne t_d(\vuc bn)$. Let $g$ be the homomorphism extending the assignment $x_i \longmapsto b_i$; then $\th(\Sigma) \sse \op{ker}(g)$ so by the Second Homomorphism Theorem there is a homomorphism $f: \alg A \longrightarrow \alg B$ such that $f(a_i) = b_i$.  Observe that  $f(\alg A) \in \vv Q$ (since it is a subalgebra of $\alg B \in \vv Q$) and $f(\alg A) \not\vDash \Phi$, so by what we said above $f(\alg A) \cong \alg A$; this clearly implies $\alg A \in \II\SU(\alg B)$, so $\alg B \notin [\vv Q: \alg A]$ as wished.

For the converse, let $\vv Q$ be locally finite of finite type; by (1) $\alg A$ is finite. Suppose that $\alg A \le_{sd} \prod_{i \in I} \alg B_i$ where each $\alg B_i$ is relative subdirectly irreducible in $\vv Q$.
Since $\alg A$ is finite, each $\alg B_i$ can be taken to be finite; if $\alg A \notin \II\SU(\alg B_i)$ for all $i$, then $\alg B_i \in [\vv Q:\alg A]$ for all $i$ and hence, being $[\vv Q:\alg A]$ a quasivariety
we have $\alg A \in [\vv Q:\alg A]$ which is impossible. Hence there is an $i$ such that $\alg A \in \II\SU(\alg B_i)$, so that $|A| \le |B_i|$; on the other hand $\alg B \in \HH(\alg A)$, so $|B| \le |A|$. Since everything is finite we have $\alg A \cong \alg B_i$ and then $\alg A$ is relative subdirectly irreducible.

For the first forward direction of (3), suppose that $\alg B \in \HH([\vv Q:\alg A])$. If $\alg A \in \II\SU(\alg B)$, then $\alg A \in \SU\HH([\vv Q:\alg A]) \sse \HH\SU([\vv Q:\alg A])$. Now $[\vv Q: \alg A]\sse \vv Q$ and $\alg A$ is weakly projective in $\vv Q$; so $\alg A \in \SU([\vv Q:\alg A])$ which is impossible.  It follows that $\alg A \notin \II\SU(\alg B)$ and $\alg B \in [\vv Q:\alg A]$; thus $[\vv Q:\alg A]$ is closed under $\HH$. For the second forward direction, it is easy to see that if $[\vv Q:\alg A]$ is closed under $\HH$ then $[\vv Q:\alg A]$ is equational relative to $\vv Q$.
Assume now that $[\vv Q:\alg A]$ is closed under $\HH$, we show that $\alg A$ is weakly projective in $\vv Q$. Suppose that $\alg A \in \HH(\alg B)$ for some $\alg B \in \vv Q$; if $\alg A \notin \II\SU(\alg B)$, then
$\alg B \in [\vv Q:\alg A]$  and, since $[\vv Q:\alg A]$ is closed under $\HH$, $\alg A \in [\vv Q:\alg A]$, again a contradiction. Hence $\alg A \in \II\SU(\alg B)$ and $\alg A$ is weakly projective in $\vv Q$. A completely analogous proof shows that if $[\vv Q:\alg A]$ is equational relative to $\vv Q$ then $\alg A$ is weakly projective, which completes the proof of (3).

(4) follows directly from (1), (2) and (3).
\end{proof}

Thus if $\alg A$ is relative subdirectly irreducible and finitely presented, then $[\vv Q:\alg A]$ is a quasivariety; this is the key to prove the following result, appearing in \cite[Proposition 5.1.24]{Gorbunov1998}. We present a self-contained proof for the sake of the reader.

\begin{theorem}[\cite{Gorbunov1998}]\label{mainstructural} If $\vv Q$ is a locally finite quasivariety of finite type, then the following are equivalent.
\begin{enumerate}
\item\label{mainstructural1} $\vv Q$ is primitive;
\item\label{mainstructural2} for all finite relative subdirectly irreducible  $\alg A \in \vv Q$,  $[\vv Q:\alg A]$ is equational relative to $\vv Q$;
\item\label{mainstructural3} every finite  relative subdirectly irreducible  $\alg A \in \vv Q$ is weakly projective in $\vv Q$;
\item\label{mainstructural4} every finite relative subdirectly irreducible  $\alg A \in \vv Q$ is weakly projective in the class of finite algebras in $\vv Q$.
\end{enumerate}
\end{theorem}
\begin{proof}
	(\ref{mainstructural2}) and (\ref{mainstructural3}) are equivalent by Lemma \ref{universal}, and (\ref{mainstructural3}) and (\ref{mainstructural4}) are equivalent in locally finite quasivarieties. 
	
	Now, (\ref{mainstructural1}) implies (\ref{mainstructural2}) by Lemma \ref{universal}, since if $\alg A$ is a finite relative subdirectly irreducible algebra then $[\vv Q:\alg A]$ is a quasivariety, and if $\vv Q$ is primitive every subquasivariety is equational relative to $\vv Q$ by definition. 
	
Finally, assume (\ref{mainstructural3}) and let $\vv Q'$ be a subquasivariety of $\vv Q$; consider a finite algebra $\alg B \in \HH(\vv Q') \cap \vv Q$, then 
	$\alg B$ is a subdirect product of finite relative subdirectly irreducible algebras in $\vv Q$, that is, $\alg B \le_{sd} \prod_{i \in I} \alg A_i$ where each $\alg A_i$ is finite relative subdirectly irreducible in $\vv Q$, and thus it is also weakly projective in $\vv Q$ by hypothesis.
	Since $\alg B \in \HH(\vv Q')$, there is $\alg A \in \vv Q'$ such that $\alg B \in \HH(\alg A)$. But then for each $i \in I$, $\alg A_i \in \HH(\alg A)$; since each $\alg A_i$ is weakly projective in $\vv Q$, it is also isomorphic to a subalgebra of $\alg A$. Thus, $\alg B \in \II\SU\PP(\alg A) \sse \vv Q'$, and therefore $\vv Q' =  \HH(\vv Q') \cap \vv Q$, which means that $\vv Q$ is primitive and (\ref{mainstructural1}) holds.
	\end{proof}

Most results in the literature are about structurally complete and primitive {\em varieties} of algebras and the reason is quite obvious; first the two concepts are easier to formulate for varieties. Secondly being subdirectly irreducible is an absolute concept (every subdirectly irreducible algebra is relative subdirectly irreducible in any quasivariety to which it belongs) while being relative subdirectly irreducible depends on the subquasivariety we are considering. Of course when a quasivariety is generated by a ``simple'' class (e.g. by finitely many finite algebras), then Theorem \ref{quasivariety}(2) gives a simple solution, but in general describing the relative subdirectly irreducible algebras in a quasivariety is not an easy task.

Now, it is clear that if $\vv Q$ is non-negative universally complete, then it is structurally complete. Finding examples of (quasi)varieties that are structurally complete but not primitive is not easy; one idea is to find a finite algebra $\alg A$ such that $\alg A$ satisfies the hypotheses of Lemma \ref{lemma: wpstructcomplete}, but $\VV(\alg A)$ contains some strict (i.e. not a variety) subquasivariety. We will see an example of this in Section \ref{sec:lattices}. Let us now show some different kinds of examples of primitive (quasi)varieties.

\begin{example}
The variety of bounded distributive lattices is primitive  (as we will discuss in Section \ref{sec:lattices}), since it is equationally complete  and congruence modular and so is q-minimal by Theorem \ref{thm: mckenziebergman}.
\end{example}

It is well-known (and easy to check) that the variety of distributive lattices is a {\em dual discriminator variety}; a {\em dual discriminator} on a set $A$ is a ternary operation $d$ on $A$ defined by
$$
d(a,b,c) =\left\{
            \begin{array}{ll}
              c, & \hbox{if $a\ne b$;} \\
              a, & \hbox{if $a=b$.}
            \end{array}
          \right.
$$
A variety $\vv V$ is a dual discriminator variety \cite{FriedPixley1979} if there is a ternary term that is the dual discriminator on all the subdirectly irreducible members of $\vv V$.
Dual discriminator varieties, as opposed to discriminator varieties, do not necessarily have projective unifiers.
However, recently in \cite{Caicedoetal2021} the authors have extended the results in \cite{BergmanMcKenzie1990} (such as Theorem \ref{thm: mckenziebergman}) in two directions: every minimal dual discriminator variety is q-minimal, hence primitive and, if the variety is also idempotent, then minimality can be dropped and the variety is primitive. This last fact gives raise to different examples of  primitive varieties.
\begin{example}
 A {\em weakly associative lattice} is an algebra $\la A,\join,\meet \ra$ where $\join$ and $\meet$ are idempotent, commutative and satisfy the absorption laws but (as the name reveals) only a weak form of associativity. In \cite{FriedPixley1979} the authors proved that there is a largest dual discriminator variety $\vv U$ of weakly associative lattices; since weakly associative lattices are idempotent, $\vv U$ is the largest primitive variety of weakly associative lattices.
\end{example}
\begin{example}
The {\em pure dual discriminator variety} $\vv D$ (see \cite[Theorem 3.2]{FriedPixley1979})  is a variety with a single ternary operation $d(x,y,z)$ satisfying
\begin{align*}
&d(x,y,y) \app y\\
&d(x,y,x) \app x \\
&d(x,x,y) \app x\\
&d(x,y,d(x,y,z)) \app d(x,y,z)\\
&d(u,v,d(x,y,z)) \app d(d(u,v,x),d(u,v,y),d(u,v,z))
\end{align*}
which is enough to prove that $\vv D$ is a dual discriminator variety. Since $d$ is idempotent  $\vv D$ is an idempotent dual discriminator variety and so it is primitive.
\end{example}

A different example is given by the following.

\begin{example}
	 A {\em modal algebra} is a Boolean algebra with a modal operator $\Box$, that we take as a basic unary operation, satisfying $\Box 1 \app 1$ and $\Box (x \meet y) \app \Box x \meet \Box y$; there is an extensive literature on modal algebras (see for instance \cite{Wolter1997} and the bibliography therein).
  A modal algebra is a {\em K4-algebra} if it satisfies
 $\Box x \le \Box\Box x$;
 in \cite{Rybakov1995} V.V. Rybakov classified all the primitive varieties of K4-algebras. However very recently \cite{Carr2022} Carr discovered a mistake in Rybakov's proof; namely Rybakov in his description missed some varieties that all have the properties of containing a unifiable weakly projective algebra that is not projective.  So any of such varieties, though primitive,  does not have projective unifiers.
\end{example}

We now present some examples from (quasi)varieties that are the equivalent algebraic semantics of (fragments) of many-valued logics; in particular, of infinite-valued \L ukasiewicz logic. 

\begin{example}
Wajsberg algebras are the equivalent algebraic semantics of infinite-valued \L ukasiewicz logic in the signature of bounded commutative residuated lattices $(\cdot, \to, \land, \lor, 0, 1)$ and they are term-equivalent to the better known MV-algebras \cite{CDOM}; Wajsberg hoops are their $0$-free subreducts. About these algebras there are some recent results \cite{Agliano2022a}. In summary:
\begin{enumerate}
\ib the only primitive variety of Wajsberg algebras is the variety of Boolean algebras, that is also non-negative universally complete, and it is the only non-negative universally complete variety of Wajsberg algebras;
\ib there are nontrivial primitive quasivarieties of Wajsberg algebras;
\ib a proper variety of Wajsberg hoops is structurally complete if and only if it is primitive if and only if every subdirectly irreducible is either finite or perfect.
\end{enumerate}

The third point above clearly implies that the variety of Wajsberg hoops is not primitive. Considering varieties of $\imp$-subreducts, the $\imp$-subreducts of Wajsberg hoops is a subvariety of $\mathsf{BCK}$-algebras usually denoted by $\mathsf{LBCK}$;
 every locally finite subvariety of $\mathsf{LBCK}$-algebras is a variety of $\mathsf{HBCK}$-algebras, so it is universally complete. However:
\begin{enumerate}
\ib the only non locally finite subvariety is the entire variety $\mathsf{LBCK}$ \cite{Komori1978};
\ib $\mathsf{LBCK}$  it is generated as a quasivariety by its finite chains \cite{AFM};
\ib every infinite chain contains all the finite chains as subalgebras \cite{Komori1978};
\ib so if $\vv Q$ is a quasivariety which contains only finitely many chains, then $\VV(\vv Q)$ is locally finite, hence universally complete and so
 $\vv Q = \VV(\vv Q)$;
\ib otherwise $\vv Q$ contains infinitely many chains and so $\VV(\vv Q) =  \vv Q = \mathsf{LBCK}$.
\end{enumerate}
Hence every subquasivariety of $\mathsf{LBCK}$ is a variety and $\mathsf{LBCK}$ is primitive.
The status of non locally finite varieties of basic hoops and basic algebras is still unclear (except for the cases we mentioned) and it is under investigation.
\end{example}

\subsection{Active structurally complete quasivarieties}
The problem of active structural completeness has been tackled in \cite{DzikStronkowski2016}; it is an extensive and profound paper touching many aspects and there is no need to reproduce it here. We will only state the definition, the main result, and we will display an example.

\begin{theorem}[\cite{DzikStronkowski2016}] For a quasivariety $\vv Q$ the following are equivalent:
\begin{enumerate}
\item $\vv Q$ is active structurally complete;
 \item every unifiable algebra of $\vv Q$ is in $\QQ(\alg F_{\vv Q}(\omega))$;
\item every finitely presented unifiable algebra in $\vv Q$ is in $\QQ(\alg F_\vv Q(\o))$;
\item every admissible quasiequation in $\vv Q$ is valid in all
the finitely presented unifiable algebras in $\vv Q$;
		\item for every $\alg A \in \vv Q$, $\alg A \times \alg F_{\vv Q}\in \QQ(\alg F_{\vv Q}(\omega))$.
		\item for every $\alg A \in \vv Q_{rsi}$, $\alg A \times \alg F_{\vv Q}\in \II \SU \PP_u(\alg F_{\vv Q}(\omega))$.
\end{enumerate}
\end{theorem}
Given that, we have as usual:

\begin{corollary} A quasivariety $\vv Q$ is active structurally completeif and only if $\cc L_\vv Q$ is actively structurally complete.
\end{corollary}

\begin{example}\label{ex: as not au}
	An $S4$-algebra is a $K4$-algebra satisfying $\Box x \le x$; if we define $\Diamond x := \neg \Box \neg x$, then a {\em monadic} algebra is an $S4$-algebras satisfying $\Diamond x \le \Box\Diamond x$. Now let
$\alg A$, $\alg B$ be the monadic algebra and the $S4$-algebra in Figure \ref{monadic} and let $\vv V = \VV(\alg A)$ and $\vv W= \VV(\alg B)$.

\begin{figure}[htbp]
\begin{center}
\begin{tikzpicture}[scale=.9]
\draw (0,0) -- (-1,1) -- (0,2) --(1,1) -- (0,0);
\draw[fill] (0,0) circle [radius=0.05];
\draw[fill] (1,1) circle [radius=0.05];
\draw[fill] (0,2) circle [radius=0.05];
\draw[fill] (-1,1) circle [radius=0.05];
\node[right] at (0,0) {\footnotesize $0=\Box 0 = \Box a = \Box b$};
\node[right] at (0,2) {\footnotesize $1=\Box 1$};
\node[right] at (1,1) {\footnotesize $b$};
\node[left] at (-1,1) {\footnotesize $a$};
\node[below] at (0,-.5) {\footnotesize  $\alg A$};
\draw (6,0) -- (5,1) -- (6,2) --(7,1) -- (6,0);
\draw (6,1) -- (5,2) -- (6,3) --(7,2) -- (6,1);
\draw (5,2) -- (5,1) --  (6,0) -- (6,1) -- (7,2) -- (7,1) -- (6,2) --(6,3);
\draw[fill] (6,0) circle [radius=0.05];
\draw[fill] (6,1) circle [radius=0.05];
\draw[fill] (6,2) circle [radius=0.05];
\draw[fill] (5,1) circle [radius=0.05];
\draw[fill] (5,2) circle [radius=0.05];
\draw[fill] (6,3) circle [radius=0.05];
\draw[fill] (7,1) circle [radius=0.05];
\draw[fill] (7,2) circle [radius=0.05];
\node[right] at (6,0) {\footnotesize $0=\Box 0 = \Box c$};
\node[right] at (7,1) {\footnotesize $ b=\Box b = \Box a'$};
\node[left] at (5,1) {\footnotesize $ a=\Box a = \Box b'$};
\node[right] at (6,1) {\footnotesize $c$};
\node[left] at (5,2) {\footnotesize $b'$};
\node[right] at (7,2) {\footnotesize $a'$};
\node[right] at (6,2) {\footnotesize $\Box c'$};
\node[left] at (6,2) {\footnotesize $c'$};
\node[below] at (6,-.5) {\footnotesize  $\alg B$};
\end{tikzpicture}
\caption{$\alg A$ and $\alg B$}\label{monadic}
\end{center}
\end{figure}

Let $\vv U = \vv V \join \vv W$ (the varietal join);
from \cite[Section 8]{DzikStronkowski2016} one can deduce that:
\begin{enumerate}
\ib every finitely generated algebra in $\vv U$ is isomorphic to the direct product of an algebra in $\vv V$ and one in $\vv W$, hence
$\vv U$ is locally finite;
\ib $\vv U$ is active structurally complete but not structurally complete;
\ib $\vv U$ does not have exact unifiers.
\end{enumerate}
Since $\vv U$ is locally finite, by Theorem \ref{lemma: ex unif auc}, it cannot be active universally complete; so $\vv U$ is an example of a variety that is active structurally complete but not active universally complete.
\end{example}

\subsection{Passive quasivarieties}
Passive structurally complete quasivarieties have been studied in \cite{MoraschiniWannenburg2020} in relation to the joint embedding property, while here we take a different path.
We start with the following observation.
\begin{proposition}\label{prop: nnu iff ps and au}
	A quasivariety $\vv Q$ is passive structurally complete if and only if  every non-negative  passive admissible universal sentence is derivable in $\vv Q$.
\end{proposition}
\begin{proof}
	For the non-trivial direction, suppose $\vv Q$ is passive structurally complete, and let $\Sigma \Rightarrow \Delta$ be a non-negative ($\Delta \neq \emptyset$) passive admissible universal sentence. This means that $\Sigma$ is not unifiable, and thus, each quasiequation $\Sigma \Rightarrow \delta$, for any $\delta \in \Delta$, is passive admissible. By hypothesis, each such  $\Sigma \Rightarrow \delta$ is valid in $\vv Q$, thus so is  $\Sigma \Rightarrow \Delta$ and the conclusion holds.
\end{proof}
It is clear that a key concept to study passive clauses is understanding the unifiability of the premises. In order to do so, we introduce the following notion.

\begin{definition}
We say that a finite set of identities $\Sigma$ is {\em trivializing} in a class of algebras $\vv K$ if the quasiequation $\Sigma \Rightarrow (x \approx y)$ is valid in $\vv K$, where the variables $x, y$ do not appear in $\Sigma$.
\end{definition}
Notice that such a quasiequation $\Sigma \Rightarrow (x \app y)$ is valid in an algebra $\alg A$ if and only if either $\alg A$ is trivial, or there is no assignment of the variables of $\Sigma$ in $\alg A$ that makes $\Sigma$ valid in $\alg A$.
\begin{lemma}\label{prop:trivial}
	Let $\vv Q$ be a quasivariety, and let $\Sigma$ be a finite set of equations in its language. The following are equivalent:
	\begin{enumerate}
		\item $\Sigma$ is not unifiable in $\vv Q$;
		\item $\alg F_{\vv Q}$ is nontrivial and $\Sigma$ is trivializing in $\QQ(\alg F_{\vv Q})$;
		\item $\alg F_{\vv Q} \models \Sigma \Rightarrow \emptyset$.
	\end{enumerate}
\end{lemma}
\begin{proof}
It is easy to see that (2) and (3) are equivalent, modulo the fact that a set of identities is trivializing in $\QQ(\alg F_{\vv Q})$ if and only if it is trivializing in $\alg F_{\vv Q}$.

Let us now assume that the identities in $\Sigma$ are on a (finite) set of variables $X$.
	Then, given Lemma \ref{free}, $\Sigma$ is not unifiable in $\vv Q$ if and only if there is no homomorphism $h: \alg F_{\vv Q}(X)/\theta_{\vv Q}(\Sigma) \to \alg F_{\vv Q}$. We show that the latter holds if and only if there is no homomorphism $k: \alg F_{\vv Q}(X) \to \alg F_{\vv Q}$ such that $k(t) = k(u)$ for each $t \approx u \in \Sigma$. Indeed, for the non-trivial direction, suppose that there is a homomorphism $k: \alg F_{\vv Q}(X) \to \alg F_{\vv Q}$ with the above property. Then the following diagram commutes, i.e., there is a homomorphism $h: \alg F_{\vv Q}(X)/\theta_{\vv Q}(\Sigma) \to \alg F_\vv Q$:

\begin{center}
\begin{tikzpicture}[scale=.9]
\node at (2,3) { $\alg F_{\vv Q}(X)$};
\node[left] at (4.8,3) {$\alg F_{\vv Q}$};
\node at (2,1) { $\alg F_{\vv Q}(X)/\theta_{\vv Q}(\Sigma)$};
\draw[->] (2.7,3) -- (4,3);
\draw[->] (2.2,1.5) -- (4,2.75);
\draw[->] (2,2.7) -- (2,1.5);
\node[above] at (3.3,3) {\footnotesize $k$};
\node at (3.5,2.1) {\footnotesize $h$};
\node[below] at (1.8,2.25) {\footnotesize $\pi$};
\end{tikzpicture}
\end{center}

Notice that there is no homomorphism $k: \alg F_{\vv Q}(X) \to \alg F_{\vv Q}$ such that $k(t) = k(u)$ for each $t \approx u \in \Sigma$ if and only if there is no assignment of variables in $X$ validating $\Sigma$ in $\alg F_{\vv Q}$. The latter is equivalent to $\alg F_{\vv Q} \models \Sigma \Rightarrow \emptyset$.
\end{proof}
We are now ready to prove the characterization theorem.
\begin{theorem}\label{thm:PUC}
	Let $\vv Q$ be a quasivariety, then the following are equivalent.
	
	\begin{enumerate}
	\item\label{thm:PUC1} $\vv Q$ is passive structurally complete;
	\item\label{thm:PUC4} $\alg F_{\vv Q} \models \Sigma \Rightarrow \emptyset$ implies $\Sigma$ is trivializing in $\vv Q$;
	\item\label{thm:PUC5} either $\alg F_{\vv Q}$ is trivial, or $\Sigma$ is trivializing in $\QQ(\alg F_{\vv Q})$ implies $\Sigma$ is trivializing in $\vv Q$;
		\item\label{thm:PUC2} every nontrivial finitely presented algebra is unifiable.
	\end{enumerate}
\end{theorem}
	\begin{proof}
		We first show that (\ref{thm:PUC1}) and (\ref{thm:PUC4}) are equivalent. 
	By definition, $\vv Q$ is passive structurally complete if and only if 
	each quasiidentity $\Sigma \Rightarrow \delta$ where $\Sigma$ is not unifiable in $\vv Q$ is valid in $\vv Q$. That is, $\Sigma$ not unifiable in $\vv Q$ implies $\vv Q \models \Sigma \Rightarrow \delta$, for all identities $\delta$.
		By Proposition \ref{prop:trivial}, the latter is equivalent to:
		$\alg F_{\vv Q} \models \Sigma \Rightarrow \emptyset$ implies $\vv Q \models \Sigma \Rightarrow \delta$, for all identities $\delta$.
		 From this it follows the particular case where $\delta = \{x \approx y\}$, with $x, y$ not appearing in $\Sigma$. In turn, if $\alg F_{\vv Q} \models \Sigma \Rightarrow \emptyset$ implies $\vv Q \models \Sigma \Rightarrow (x \approx y)$, then clearly $\vv Q \models \Sigma \Rightarrow \delta$ for any $\delta$, and thus (\ref{thm:PUC1}) $\Leftrightarrow$ (\ref{thm:PUC4}). 		 
		
		Now, (\ref{thm:PUC4}) and (\ref{thm:PUC5}) are equivalent by Lemma \ref{prop:trivial}, thus the first three points are equivalent. 
		Let us now assume (\ref{thm:PUC4}) and prove (\ref{thm:PUC2}). We consider a nontrivial finitely presented algebra in $\vv Q$, $\alg F_{\vv Q}(X)/\theta_{\vv Q}(\Sigma)$. If it is not unifiable, by Lemma \ref{prop:trivial} $\alg F_{\vv Q} \models \Sigma \Rightarrow \emptyset$. By (\ref{thm:PUC4}) this implies that $\Sigma$ is trivializing in $\vv Q$, that is, $\vv Q \models \Sigma \Rightarrow (x \approx y)$ (with $x,y$ new variables). This clearly implies that $\alg F_{\vv Q}(X)/\theta_{\vv Q}(\Sigma)$ is trivial, a contradiction. Thus $\alg F_{\vv Q}(X)/\theta_{\vv Q}(\Sigma)$ is unifiable and (\ref{thm:PUC2}) holds.
	
	Finally, we prove that (\ref{thm:PUC2}) implies (\ref{thm:PUC1}). Suppose $\Sigma \Rightarrow \delta$ is a passive quasiequation over variables in $X$, that is, $\Sigma$ is not unifiable in $\vv Q$. By Lemma \ref{prop:trivial} $\alg F_{\vv Q} \models \Sigma \Rightarrow \emptyset$. Let $x,y$ be variables not in $X$, and consider the finitely presented algebra $\alg F_{\vv Q}(X')/\theta_{\vv Q}(\Sigma)$, where $X' = X \cup \{x, y\}$ and suppose by way of contradiction that it is not trivial. By (\ref{thm:PUC2}) it is unifiable, that is, there is a homomorphism $h: \alg F_{\vv Q}(X')/\theta_{\vv Q}(\Sigma) \to \alg F_{\vv Q}$. Then, considering the natural epimorphism $\pi_{\Sigma}: \alg F_{\vv Q}(X') \to \alg F_{\vv Q}(X')/\theta_{\vv Q}(\Sigma)$, the composition $h \pi_\Sigma$ is an assignment  from $X'$ to $\alg F_\vv Q$ satisfying $\Sigma$; but $\alg F_{\vv Q} \models \Sigma \Rightarrow \emptyset$, a contradiction. Thus $\alg F_{\vv Q}(X')/\theta_{\vv Q}(\Sigma)$ is trivial, and therefore $x / \theta_{\vv Q}(\Sigma) = y / \theta_{\vv Q}(\Sigma)$. By Lemma \ref{lemma:FThetaSigma} $\vv Q \models \Sigma \Rightarrow (x \approx y)$, and thus $\vv Q \models \Sigma \Rightarrow \delta$ and (\ref{thm:PUC1}) holds.
	\end{proof}
Analogously to the case of passive universal completeness, if the smallest free algebra is isomorphic to all its ultraproducts we can improve the previous result.
\begin{lemma}
Let $\vv Q$ be a quasivariety such that $\II(\alg F_{\vv Q}) = \II\PP_u(\alg F_{\vv Q})$, then the following are equivalent.
\begin{enumerate}
		\item every nontrivial finitely presented algebra in $\vv Q$ is unifiable;
	\item every nontrivial algebra in $\vv Q$ is unifiable.
\end{enumerate}
\end{lemma}
\begin{proof}
The proof is analogous to the one of Lemma \ref{lemma:Qunifiable};
	we prove the nontrivial direction by contraposition. Consider an arbitrary algebra $\alg A = \alg F_\vv Q(X)/\theta \in \vv Q$ and assume that it is not unifiable. Then there is no assignment $h: \alg F_\vv Q(X) \to \alg F_\vv Q$ such that $\alg F_\vv Q, h \models \Sigma_\theta$, where $\Sigma_\theta = \{t \approx u : (t,u) \in \theta\}$. Equivalently, iff $\alg F_\vv Q \models \Sigma_\theta \Rightarrow \emptyset$. Now, the equational consequence relation relative to $\alg F_{\vv Q}$ is finitary (since all ultraproducts of $\alg F_{\vv Q}$ are isomorphic to $\alg F_{\vv Q}$); thus we obtain that $\alg F_\vv Q \models \Sigma'_\theta \Rightarrow \emptyset$, for $\Sigma'_\theta$ some finite subset of $\Sigma_\theta$. But $\alg F_{\vv Q}(X)/\theta \not\vDash \Sigma'_\theta \Rightarrow (x \approx y)$ (with $x,y \notin X$), since it is nontrivial, which contradicts (\ref{thm:PUC4}) of Theorem \ref{thm:PUC}; equivalently it contradicts (1) and thus the proof is complete.
\end{proof}

\begin{corollary}\label{cor:PUCPu}
Let $\vv Q$ be a quasivariety such that $\II(\alg F_{\vv Q}) = \II\PP_u(\alg F_{\vv Q})$, then the following are equivalent.
\begin{enumerate}
\item $\vv Q$ is passively structurally complete;
	\item $\alg F_{\vv Q} \models \Sigma \Rightarrow \emptyset$ implies $\Sigma$ is trivializing in $\vv Q$;
	\item either $\alg F_{\vv Q}$ is trivial, or $\Sigma$ is trivializing in $\QQ(\alg F_{\vv Q})$ implies $\Sigma$ is trivializing in $\vv Q$;
		\item every nontrivial finitely presented algebra is unifiable;
		\item every nontrivial algebra in $\vv Q$ is unifiable.
\end{enumerate}
\end{corollary}
\begin{remark}
	The previous corollary can be applied whenever $\alg F_\vv Q$ is finite, therefore to all locally finite quasivarieties, but also to more complex classes of algebras, e.g., all subquasivarieties of $\mathsf{FL}_{w}$ (see Subsection \ref{sec:FL}).
\end{remark}
We will see an interesting application of Theorem \ref{thm:PUC} (or Corollary \ref{cor:PUCPu}) in substructural logics in Subsection \ref{sec:FL}; let us now show some other consequences.
Given a quasivariety $\vv Q$ whose smallest free algebra $\alg F_\vv Q$ is nontrivial, let us consider the following set:
$$
\mathcal{P}_{\vv Q} = \{\Sigma \Rightarrow \delta: \vv Q(\alg F_{\vv Q})\vDash \Sigma \Rightarrow \{x \approx y\}, \delta \mbox{ any identity}\}.
$$	
$\mathcal{P}_{\vv Q}$ axiomatizes a subquasivariety of $\vv Q$, that we denote with $\vv P_{\vv Q}$. From Theorem \ref{thm:PUC} we get the following.

\begin{corollary}
Let $\vv Q$ be a quasivariety such that $\alg F_\vv Q$ is nontrivial.
Every passive structurally complete subquasivariety of $\vv Q$ is contained in
	$\vv P_{\vv Q}$, which is the largest subquasivariety of $\vv Q$ that is passive structurally complete.
\end{corollary}

Moreover, for locally finite quasivarieties the characterization theorem reads as follows.

\begin{corollary}
Let $\vv Q$ be a locally finite quasivariety, then the following are equivalent.
\begin{enumerate}
\item  $\vv Q$ is passive structurally complete;
\item every algebra in  $\vv Q$ is unifiable;
\item  every finite algebra in $\vv Q$ is unifiable;
\item  every finite subdirectly irreducible in $\mathsf Q$ is unifiable.
\end{enumerate}
\end{corollary}

A nontrivial algebra $\alg A$ is \emph{Koll\'ar} if it has no trivial subalgebras, and a quasivariety $\vv Q$ is a {\em Koll\'ar quasivariety} if all nontrivial algebras in $\vv Q$ are Koll\'ar.
By  \cite{Kollar1979} if $\alg A$ belongs to a Koll\'ar quasivariety, $1_\alg A$, the largest congruence of $\alg A$, is compact in $\op{Con}_\vv Q(\alg A)$; from there a straightforward application of Zorn's Lemma yields that if $\alg A$ is nontrivial there is at least one maximal congruence $\th \in \op{Con}_\vv Q (\alg A)$ below $1_\alg A$ (i.e.  $\alg A/\th$ is relative simple).

\begin{theorem}\label{semideg} If $\vv Q$ is a Koll\'ar quasivariety and $\alg F_\vv Q$ is the only finitely generated relative simple algebra in $\vv Q$, then
$\vv Q$ is passive structurally complete.
\end{theorem}
\begin{proof} Let $\alg A$ be a  nontrivial finitely presented algebra in $\vv Q$; since $\vv Q$ is a Koll\'ar quasivariety, $\alg A$ has a relative simple homomorphic image, that must be finitely generated. Hence it  must be equal to $\alg F_\vv Q$, so $\alg A$ is unifiable; by Theorem \ref{thm:PUC}  $\vv Q$ is passive structurally complete.
\end{proof}

\begin{corollary}\label{semideglf} For a locally finite Koll\'ar quasivariety $\vv Q$ such that $\alg F_{\vv Q}$ has no proper subalgebra the following are equivalent:
\begin{enumerate}
\item $\alg F_\vv Q$ is the only  finite  relative simple algebra in $\vv Q$;
\item $\vv Q$ is passive structurally complete.
\end{enumerate}
\end{corollary}
\begin{proof}  If (1) holds, than (2) holds by Theorem \ref{semideg}.  Conversely assume (2); then every nontrivial finitely presented algebra in $\vv Q$ is unifiable.  Since $\vv Q$ is locally finite $\alg F_\vv Q$ is finite and nontrivial since $\vv Q$ is Koll\'ar; now  since $\alg F_{\vv Q}$ has no proper subalgebra no finite relative simple algebra different from $\alg F_\vv Q$ can be unifiable, but $\vv Q$ must contain at least a relative simple algebra \cite[Theorem 3.1.8]{Gorbunov1998}. Hence $\alg F_\vv Q$ must be relative simple and (1) holds.
\end{proof}

The next results will allow us to find interesting applications in varieties of bounded lattices, which we will explore in Section \ref{sec:lattices}.  We say that an algebra $\alg A$ in a variety $\vv V$ is {\em flat} if $\HH\SU(\alg A)$ does not contain any simple algebra different from
$\alg F_\vv V$.

\begin{theorem}  Let $\vv V$ be a Koll\'ar variety; if every finitely generated
algebra in $\vv V$ is flat then $\vv V$ is passive structurally complete. If $\vv V$ is locally finite,
then the converse holds as well.
\end{theorem}
\begin{proof} First, if $\alg F_\vv V$ is trivial then $\vv V$ is vacuously passive structurally complete. If $\alg F_\vv V$ is nontrivial and every finitely generated algebra is flat, then the only finitely
generated simple lattice in $\vv V$ must be $\alg F_{\vv V}$; since $\vv V$ is Koll\'ar, $\vv V$ is passive structurally complete by Theorem \ref{semideg}.

If $\vv V$ is locally finite and passive structurally complete, then $\alg F_\vv V$ is the only
finite simple algebra in $\vv V$ by Corollary \ref{semideglf}. It follows that no finite simple
algebra different from  $\alg F_\vv V$ can appear in $\HH\SU(\alg A)$ for any finite $\alg A\in\vv V$.
So every finite algebra in $\vv V$ must be flat.
\end{proof}

\begin{theorem}\label{fgflat} Let $\vv V$ be a congruence distributive Koll\'ar variety;  a finitely generated variety $\vv W \sse \vv V$  is passive structurally complete  if and only if each generating algebra is flat.
\end{theorem}
\begin{proof} Suppose that $\vv W = \VV(K)$ where $K$ is a finite set of finite algebras; by
J\'onsson Lemma any simple algebra in $\vv V$ is in $\HH\SU(K)$. If $K$ consists entirely of flat algebras, then there cannot be any simple algebra in $\vv V$ different from $\alg F_\vv W$, so $\vv W$ is passive structurally complete. On the other hand if $\alg A \in K$ is not flat, then there is an algebra
$\alg B \in \HH\SU(K)$ which is simple and different from $\alg F_\vv V$. Clearly $\alg B \in \vv W$, which is not passive structurally complete.
\end{proof}

\section{Applications to algebra and logic}
In this last section we will see some relevant examples and applications of our results in the realm of algebra and (algebraic) logic that deserve a deeper exploration than the examples already presented in the previous sections. We will start with focusing on varieties of lattices and bounded lattices, and then move to their expansions that are the equivalent algebraic semantics of subtructural logics: residuated lattices. 

As a main result, in the last subsection we present the logical counterpart of the characterization of passive structural completeness in substructural logics with weakening, that is, such a logic is passively structurally complete if and only if every classical contradiction is explosive in it; building on this, from the algebraic perspective, we are able to axiomatize the largest variety of representable bounded commutative integral residuated lattices that is passively structurally complete (and such that all of its quasivarieties have this property). Notice that this characterization establishes negative results as well: if a logic (or a quasivariety) is not passively structurally complete, a fortiori it is not structurally complete either.

\subsection{(Bounded) lattices}\label{sec:lattices}
In this subsection we start with some results about primitive (quasi)varieties of lattices, and then move to bounded lattices, where in particular we obtain some new results about passive structurally complete varieties.

\subsubsection{Primitivity in lattices}\label{sec: primlattices}
Many examples of  quasivarieties that are  primitive can be found in  lattices satisfying {\em Whitman's condition} (W);  Whitman's condition is a universal sentence that holds in free lattices:
\begin{equation}
\{x \meet y \le u \join v\} \Rightarrow \{x \le u \join v, y \le u \join v, x \meet y \le u, x \meet y \le v\} \tag{W}.
\end{equation}
Now a finite lattice is finitely projective in the variety of all lattices if and only if it satisfies (W) \cite{DaveySands1977}, which implies:

\begin{lemma}\label{whitman} Let $\vv K$ be a finite set of finite lattices.  If every lattice in $\vv K$ satisfies (W) then $\QQ(\vv K)$ is primitive.
\end{lemma}
\begin{proof} $\QQ(\vv K)$ is locally finite and by Theorem \ref{quasivariety}(2) every relative subdirectly irreducible lies  in $\II\SU(\vv K)$; as (W) is a universal sentence it is preserved under subalgebras, thus they all satisfy (W)
and hence they are all finitely projective in the variety of lattices and then also in $\QQ(\vv K)$. By Theorem \ref{mainstructural}(4), $\QQ(\vv K)$ is primitive.
\end{proof}

Luckily finite lattices satisfying  (W) abound, so there is no shortage of primitive quasivarieties of lattices. For varieties of lattices the situation is slightly different; in particular, because of Lemma \ref{lemma: wpstructcomplete} it is not enough that all lattices in $\vv K$ are weakly projective in $\VV(\vv K)$ to guarantee that $\VV(\vv K)$ is structurally complete.

First we introduce some lattices: $\alg M_n$ for $3\le n \le \o$ are  the modular lattices consisting of a top, a bottom, and $n$ atoms while
the lattices  $\alg M_{3,3}$ and $\alg M_{3,3}^+$ are displayed in Figure \ref{m3lattices}.

\begin{figure}[htbp]
\begin{center}
\begin{tikzpicture}[scale=.7]
\draw (0,0) -- (0,2) -- (1,3) --(1,1) -- (2,2) --(1,3) -- (0,2) -- (-1,1)-- (0,0) -- (1,1) -- (0,2);
\draw[fill] (0,0) circle [radius=0.05];
\draw[fill] (0,2) circle [radius=0.05];
\draw[fill] (1,1) circle [radius=0.05];
\draw[fill] (-1,1) circle [radius=0.05];
\draw[fill] (1,3) circle [radius=0.05];
\draw[fill] (2,2) circle [radius=0.05];
\draw[fill] (0,1) circle [radius=0.05];
\draw[fill] (1,2) circle [radius=0.05];
\node[below] at (1,0) {\footnotesize  $\alg M_{3,3}$};
\draw (6,0)  --(8.5,2.5) -- (7.5,3.5) -- (6,2);
\draw (6,2) -- (5,1)-- (6,0) -- (7,1) -- (6,2) -- (6,0);
\draw (6.5,2.5) -- (7.5,1.5) -- (7.5,3.5);
\draw[fill] (6,0) circle [radius=0.05];
\draw[fill] (5,1) circle [radius=0.05];
\draw[fill] (6,2) circle [radius=0.05];
\draw[fill] (7,1) circle [radius=0.05];
\draw[fill] (6,1) circle [radius=0.05];
\draw[fill] (6.5,2.5) circle [radius=0.05];
\draw[fill] (7.5,1.5) circle [radius=0.05];
\draw[fill] (8.5,2.5) circle [radius=0.05];
\draw[fill] (7.5,3.5) circle [radius=0.05];
\draw[fill] (7.5,2.5) circle [radius=0.05];
\node[below] at (7,0) {\footnotesize  $\alg M_{3,3}^+$};
\end{tikzpicture}
\caption{$\alg M_{3,3}$ and $\alg M_{3,3}^+$}\label{m3lattices}
\end{center}
\end{figure}
Observe that all the above lattices, with the exception of $\alg M_{3,3}$, satisfy (W). Now Gorbunov (\cite{Gorbunov1998}, Theorem 5.1.29) showed that $\alg M^+_{3,3}$ is {\em splitting} in the lattice of subquasivarieties of modular lattices. More in detail for any quasivariety $\vv Q$ of modular lattices, either $\alg M^+_{3,3} \in \vv Q$ or else $\vv Q= \QQ(\alg M_n)$ for some $n \le \o$.  Observe that, for $n <\o$, $\vv Q(\alg M_n)$ is primitive by Lemma \ref{whitman} and $\VV(\alg M_n) = \QQ(\alg M_n)$ by Lemma \ref{lemma: Q(A) variety}; then the only thing left to show is that  $\VV(\alg M_\o)$ is a primitive variety and Gorbunov did exactly that.
On the other hand no variety $\vv V$ of lattices containing $\alg M_{3,3}^+$ can be primitive; in fact $\alg M_{3,3}$ is a simple homomorphic image of $\alg M^+_{3,3}$ that  cannot be embedded in  $\alg M^+_{3,3}$.
By Lemma \ref{lemma: Q(A) variety}, $\QQ (\alg M_{3,3}^+) \subsetneq \VV(\alg M_{3,3}^+)$,  so $\vv V$ contains a strict (i.e. not a variety) subquasivariety and cannot be primitive.
Thus Gorbunov's result can be formulated as: {\em a variety of modular lattices is primitive if and only if it does not contain $\alg M^+_{3,3}$}.
Note that it cannot be improved to quasivarieties: since $\alg M_{3,3}^+$ satisfies (W), $\QQ(\alg M^+_{3,3})$ is primitive by Lemma \ref{whitman}. However we observe:

\begin{lemma}\label{lemmaM33} If $\vv Q$ is a quasivariety of modular lattices and $\alg M_{3,3} \in \vv Q$, then $\vv Q$ is not primitive.
\end{lemma}
\begin{proof} Clearly the two element lattice $\mathbf 2 \in \vv Q$ and it is easy to check that $\alg M_{3,3}^+ \le_{sd} \mathbf 2 \times \alg M_{3,3}$ so $\alg M_{3,3}^+ \in \vv Q$ and  $\alg M_{3,3} \in \HH(\alg M_{3,3}^+)$.
Since $\alg M_{3,3}$ cannot be embedded in $\alg M_{3,3}^+$, in $\vv Q$ there is a simple finite (so finitely presented, since lattices have finite type) algebra that is not weakly projective. By Theorem \ref{mainstructural}, $\vv Q$ is not primitive.
\end{proof}

Therefore to find a variety of modular lattices that is structurally complete but not primitive it is enough to find a finite lattice $\alg F$ such that $\alg M_{3,3}^+ \in
\VV(\alg F)$ but $\vv K=\{\alg F\}$ satisfies the hypotheses of Lemma \ref{lemma: wpstructcomplete}.  Bergman in \cite{Bergman1991} observed that the {\em Fano lattice} $\alg F$ has exactly those characteristics; the Fano lattice is the (modular) lattice of subspaces of $(\mathbb Z_2)^3$ seen as a vector space on $\mathbb Z_2$ and it is displayed in Figure \ref{fano}.

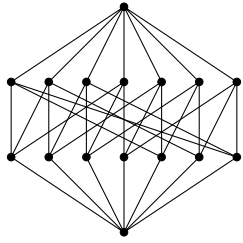
\begin{figure}[htbp]
\begin{center}
\begin{tikzpicture}[scale=1]
\draw (0,0) -- (-1.5,1) -- (-1.5,2) -- (0,3) -- (-1,2) -- (-1,1) -- (0,0) -- (-.5,1) -- (-.5,2) -- (0,3) -- (0,2) -- (0,1) -- (0,0) ;
\draw (0,0) -- (1.5,1) -- (1.5,2) -- (0,3) -- (1,2) -- (1,1) -- (0,0) -- (.5,1) -- (.5,2) -- (0,3)  ;
\draw (-1.5,2) -- (.5,1) -- (1,2);
\draw (-1.5,2) -- (1.5,1);
\draw (-1.5,1) -- (-1,2) -- (1,1) -- (1.5,2);
\draw (-1.5,1) -- (0,2) -- (-.5,1);
\draw (-1,1) -- (-.5,2) -- (1.5,1);
\draw (-1,1) -- (.5,2) -- (0,1);
\draw (-.5,1) -- (1,2);
\draw (0,1) -- (1.5,2);
\draw[fill] (0,0) circle [radius=0.05];
\draw[fill] (0,1) circle [radius=0.05];
\draw[fill] (-.5,1) circle [radius=0.05];
\draw[fill] (-1,1) circle [radius=0.05];
\draw[fill] (-1.5,1) circle [radius=0.05];
\draw[fill] (0.5,1) circle [radius=0.05];
\draw[fill] (1,1) circle [radius=0.05];
\draw[fill] (1.5,1) circle [radius=0.05];
\draw[fill] (0,2) circle [radius=0.05];
\draw[fill] (-.5,2) circle [radius=0.05];
\draw[fill] (-1,2) circle [radius=0.05];
\draw[fill] (-1.5,2) circle [radius=0.05];
\draw[fill] (0.5,2) circle [radius=0.05];
\draw[fill] (1,2) circle [radius=0.05];
\draw[fill] (1.5,2) circle [radius=0.05];
\draw[fill] (0,3) circle [radius=0.05];
\end{tikzpicture}
\caption{The Fano lattice}\label{fano}
\end{center}
\end{figure}

Now:
\begin{enumerate}
\ib $\alg F$ is projective in $\VV(\alg F)$ \cite{HermannHuhn1976};
\ib the subdirectly irreducible members of $\VV (\alg F)$ are exactly  $\mathbf 2, \alg M_3, \alg M_{3,3}, \alg F$ and they are all subalgebras of $\alg F$.
\end{enumerate}
It follows that $\alg F$ does not satisfies (W) (since $\alg M_{3,3}$ does not), $\VV(\alg F)$ is structurally complete and (since  $\alg M_{3,3} \in \VV(\alg F)$)  not primitive by Lemma \ref{lemmaM33}; also $\QQ(\alg F)$ is structurally complete but, since $\alg M_{3,3} \in \QQ(\alg F)$, it cannot be primitive as well.

Primitive varieties of lattices have been studied in depth in \cite{JipsenNation2022}; there the authors proved the following theorem that explains the behavior we have seen above.

\begin{theorem}[\cite{JipsenNation2022}] If  $\alg A$ is a lattice satisfying (W), then $\VV(\alg A)$ is primitive if and only if  every subdirectly irreducible lattice in $\HH\SU(\alg A)$ satisfies (W).
\end{theorem}

We believe that many of the techniques in \cite{JipsenNation2022} could be adapted to gain more understanding of primitive quasivarieties of lattices, but proceeding along this path would make this part  too close to being a paper in lattice theory, and we have chosen a different focus. We only borrow an example from \cite{JipsenNation2022} that shows that Lemma \ref{whitman} cannot be inverted for quasivarieties.
Let $\alg H^+, \alg H$ be the lattices in Figure \ref{H}.

It is easily seen that the pair $\alg H^+, \alg H$ behaves almost like the pair $\alg M_{3,3}^+, \alg M_{3,3}$: $\alg H^+$ satisfies (W) (so $\QQ(\alg H^+)$ is primitive), $\alg H$ does not satisfy (W) and $\alg H^+ \le_{sd} \mathbf 2 \times \alg H$. As above we can conclude that $\VV(\alg H^+)$ is not primitive.  However $\VV(\alg H)$ is primitive \cite{JipsenNation2022} so $\QQ(\alg H)$ is a primitive quasivariety generated by a finite lattice not satisfying (W).

\begin{figure}[htbp]
\begin{center}
\begin{tikzpicture}[scale=.7]
\draw (0,0) -- (-.5,1) -- (0,2) --(.5,1) -- (0,0);
\draw (0,2) -- (-.5,3) -- (0,4) -- (.5,3) --(0,2);
\draw (0,0) -- (2,2) -- (0,4);
\draw[fill] (0,0) circle [radius=0.05];
\draw[fill] (-.5,1) circle [radius=0.05];
\draw[fill] (0,2) circle [radius=0.05];
\draw[fill] (.5,1) circle [radius=0.05];
\draw[fill] (-.5,3) circle [radius=0.05];
\draw[fill] (0,4) circle [radius=0.05];
\draw[fill] (.5,3) circle [radius=0.05];
\draw[fill] (2,2) circle [radius=0.05];
\node[below] at (.5,0) {\footnotesize  $\alg H$};
\draw (6,0) -- (5.5,1) -- (6,2) --(6.5,1) -- (6,0);
\draw (6,2) -- (6,2.5);
\draw (6,2.5) -- (5.5,3.5) -- (6,4.5) --(6.5,3.5) -- (6,2.5);
\draw (6,0) -- (8,2.25) -- (6,4.5);
\draw[fill] (6,0) circle [radius=0.05];
\draw[fill] (5.5,1) circle [radius=0.05];
\draw[fill] (6,2) circle [radius=0.05];
\draw[fill] (6.5,1) circle [radius=0.05];
\draw[fill] (6,2.5) circle [radius=0.05];
\draw[fill] (5.5,3.5) circle [radius=0.05];
\draw[fill] (6,4.5) circle [radius=0.05];
\draw[fill] (6.5,3.5) circle [radius=0.05];
\draw[fill] (8,2.25) circle [radius=0.05];
\node[below] at (6.5,0) {\footnotesize  $\alg H^+$};
\end{tikzpicture}
\caption{$\alg H$ and $\alg H^+$}\label{H}
\end{center}
\end{figure}

\subsubsection{Bounded lattices}
We now focus on applications of our results in varieties of bounded lattices.
A {\em bounded} lattice is a lattice with two constants, $0$ and $1$, that represent the top and the bottom of the lattice. Bounded lattices form a variety $\vv L^b$ that shares many features with variety of lattices. In particular, let $\mathbf 2^b$ be the two element bounded lattice, then the variety  of bounded distributive lattices is $\vv D^b = \II\SU\PP(\mathbf 2^b)$. Therefore
$$
\QQ(\alg F_{\vv D_b}(\o)) \sse \vv D^b = \II\SU\PP(\mathbf 2^b) \sse \QQ(\alg F_{\vv D_b}(\o))
$$
and by Theorem \ref{structural}, the variety of bounded distributive lattices $\vv D^b$ is structurally complete, as shown in \cite{DzikStronkowski2016}.
In \cite{BergmanMcKenzie1990} it is shown that locally finite, congruence modular, minimal varieties are q-minimal; since these hypotheses apply to $\vv D^b$, the latter is also primitive. However, it is not non-negative universally complete; it is a nice exercise in general algebra to show that  for any variety $\vv V$ of bounded lattices, $1$ is join irreducible in $\alg F_{\vv V}(\o)$. It follows that
$$
\{x \join y \app 1\} \Rightarrow \{x \app 1, y \app 1\}
$$
is an active universal sentence that is admissible in $\vv V$. But it is clearly not derivable, since any nontrivial variety of bounded lattices contains
$\mathbf 2^b \times \mathbf 2^b$ which does not satisfy the universal sentence. 

\begin{proposition}\label{prop: bdl not nnu}
	No nontrivial variety of bounded lattices is  active universally complete.
\end{proposition}
 Actually something more is true; if $\vv V$ is a variety of bounded lattices that is structurally complete, then by Theorem \ref{thm: structexact}, each finite subdirectly irreducible algebra  $\alg A \in \vv V$ must satisfy the above universal sentence, i.e. $1$ must be join irreducible in $\alg A$. But the bounded lattices $\alg N_5^b$ and $\alg M_3^b$ do not satisfy that, so any structurally complete variety of bounded lattice must omit them both. As in the unbounded case, this means that the variety must be the variety of bounded distributive lattices. Thus:
\begin{proposition}[\cite{DzikStronkowski2016}]\label{proo: bdl structural} 
	The variety of bounded distributive lattices is the only (active) structurally complete variety of bounded lattices.
\end{proposition}

We have seen that active structural completeness does not have much meaning in bounded lattices.
Passive structural completeness has more content, as we are now going to show. Notice that any variety of bounded lattices is Koll\'ar and $\alg F_\vv V= \mathbf 2^b$ for any variety $\vv V$ of bounded lattices.
Since $\mathbf 2^b$ is simple and has no proper subalgebras, any simple bounded lattice not isomorphic with $\mathbf 2^b$ is not unifiable; in particular  if a variety $\vv V$ contains a finite simple lattice $\alg L$ different
from $\mathbf 2^b$, then $\VV(\alg L)$  cannot be passive structurally completeby Corollary \ref{semideglf}, and hence neither can $\vv V$.

We will use this fact to show that the only variety of bounded modular lattices that is passive structurally complete is the one we already know to possess that property, i.e.  the variety $\vv D^b$ of bounded distributive lattices.
A key step is to show that $\alg M^b_3$ is splitting in the variety of bounded modular lattices; in the unbounded case, this follows from the fact that $\alg M_3$ is projective and subdirectly irreducible. However, $\alg M_3^b$ is not projective in the variety of bounded modular lattices.
Indeed, the lattice in Figure \ref{notproj} is a bounded modular lattice having $\alg M^b_3$ as homomorphic image, but it has no subalgebra isomorphic with $\alg M^b_3$, which hence cannot be a retract.

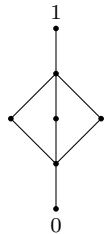
\begin{figure}[htbp]
\begin{center}
\begin{tikzpicture}[scale=.6]
\draw (0,0)-- (0,1) --(-1,2)-- (0,3) -- (0,1) -- (1,2) -- (0,3) -- (0,4);
\draw[fill] (0,0) circle [radius=0.05];
\draw[fill] (0,1) circle [radius=0.05];
\draw[fill] (0,2) circle [radius=0.05];
\draw[fill] (0,3) circle [radius=0.05];
\draw[fill] (0,4) circle [radius=0.05];
\draw[fill] (1,2) circle [radius=0.05];
\draw[fill] (-1,2) circle [radius=0.05];
\node[below] at (0,0) {\footnotesize  $0$};
\node[above] at (0,4) {\footnotesize  $1$};
\end{tikzpicture}
\end{center}
\end{figure}\label{notproj}

However we can use A. Day idea in \cite{Day1975}; a finite algebra $\alg A$ is {\em finitely projected} in a variety $\vv V$ if for any $\alg B \in \vv V$
if $f: \alg B \longrightarrow \alg A$ is surjective, then there is a finite subalgebra $\alg C$ of $\alg B$ with $f(\alg C) \cong \alg A$. Clearly any finite projective lattice is finitely projected.
A finite algebra $\alg A$  {\em splitting} in a variety  $\vv V$ if $\alg A \in \vv V$ and there is a subvariety $\vv W_\alg A \sse \vv V$ such that for any variety $\vv U \sse \vv V$ either $\alg A \in \vv U$
or $\vv U \sse \vv W_\alg A$. This simply means  that the lattice of subvarieties of $\vv V$ is the disjoint union of the filter generated by $\VV(\alg A)$ and the ideal generated by $\vv W_\alg A$.
The key result is:

\begin{theorem}\label{day}(\cite{Day1975}, Theorem 3.7) If $\vv V$ is a congruence distributive variety, then any finitely projected subdirectly irreducible algebra in $\vv V$ is splitting in $\vv V$.
\end{theorem}

\begin{lemma}\label{finpro} Let $\vv V^b$ be a variety of bounded lattices and let $\vv V$ be the variety  of lattice subreducts of $\vv V^b$. If $\alg L$ is  finitely projected in $\vv V$, then
$\alg L^b$ is finitely projected in $\vv V^b$.
\end{lemma}
\begin{proof} 
The fact that $\vv V$ is indeed a variety is easy to check.
Let now $\alg A^b \in \vv V^b$ and suppose that there is an onto homomorphism $f: \alg A^b \longrightarrow \alg L^b$;  then $f$ is onto from $\alg A$ to $\alg L$ and since
$\alg L$ is finitely projected in $\vv V$ there is a subalgebra $\alg B$ of $\alg A$ with $f(\alg B) \cong \alg L$. But   $B\cup\{0,1\}$ is the universe of a finite subalgebra $\alg C$
of $\alg A^b$. Extend $f$ to $\hat f$ by setting $\hat f(0)=0$ and $\hat f(1)=1$; then $\hat f(\alg C) \cong \alg L^b$ and so $\alg L^b$ is finitely projected in $\vv V^b$.
\end{proof}

\begin{theorem}
 A variety of modular bounded lattices is passive structurally complete if and only if it is the variety of bounded distributive lattices.
\end{theorem}
\begin{proof} $\vv D^b$ is structurally complete, hence passive structurally complete. Conversely observe that $\alg M_3$ is projective in the variety of modular lattices, so $\alg M_3^b$ is finitely
projected in the variety  of bounded modular lattices. Hence, by Theorem \ref{day},  $\alg M^b_3$ is splitting in the variety, which means that for any variety $\vv V$ of bounded modular lattices, either $\alg M_3^b \in \vv V$ or
$\vv V$ is $\vv D^b$. But if $\alg M^b_3 \in \vv V$ then $\vv V$ cannot be passive universally complete, since $\alg M_3^b$ is simple. The conclusion follows.
\end{proof}

In order to find other relevant varieties of bounded lattices that are passive structurally complete, we are going to take a closer look at flat lattices.
Finding flat bounded lattices is not hard since
the lattice of subvarieties of lattices has been studied thoroughly and a lot is known about it (an excellent survey is \cite{JipsenRose1992}). Clearly $\alg N_5$ is flat and hence so is $\alg N_5^b$;
however we know exactly all the covers of the minimal nondistributive varieties of lattices (which is of course $\VV(\alg N_5)$). There are 15
finite subdirectly irreducible nonsimple lattices, commonly called $\alg L_1,\dots,\alg L_{15}$ (some of them are in Figure \ref{lattices}) that generate all the
 join irreducible (in the lattice of subvarieties) covers of $\VV(\alg N_5)$. It is easy to see their bounded versions  all are join irreducible covers of $\VV(\alg N^b_5)$ in the lattice
 of subvarieties of bounded lattices. We suspect that they are also the only join irreducible covers; one needs only to check that the (rather long) proof for lattices \cite{JonssonRival1979} goes through
  for bounded lattices but  we leave this simple but tedious task to the reader. In any case for $i=1,\dots,15$ the subdirectly irreducible algebras in $\VV(\alg L_i^b)$ are exactly $\mathbf 2^b, \alg N_5^b$ and $\alg L_i^b$
  (via a straightforward application of J\'onsson Lemma); so each $\alg L_i^b$ is flat and each $\VV(\alg L_i^b)$ is passively structurally complete (by Theorem \ref{fgflat}).

\begin{figure}[htbp]
\begin{center}
\begin{tikzpicture}[scale=.7]
\draw[fill] (0,0) circle [radius=0.05];
\draw[fill] (-1,1) circle [radius=0.05];
\draw[fill] (0,1) circle [radius=0.05];
\draw[fill] (1,1) circle [radius=0.05];
\draw[fill] (-1,2) circle [radius=0.05];
\draw[fill] (0,2) circle [radius=0.05];
\draw[fill] (1,2) circle [radius=0.05];
\draw[fill] (0,3) circle [radius=0.05];
\draw[fill] (-.5,0.5) circle [radius=0.05];
\draw (0,0) -- (-1,1) -- (-1,2) -- (0,3) -- (1,2) -- (1,1) -- (0,0);
\draw (-1,2) -- (0,1) -- (0,0) -- (0,1) -- (1,2);
\draw (-1,1) -- (0,2) -- (0,3) -- (0,2) -- (1,1);
\node at (1,0) {\footnotesize $\alg L_{13}$};
\draw[fill] (5,0) circle [radius=0.05];
\draw[fill] (4,1) circle [radius=0.05];
\draw[fill] (5,1) circle [radius=0.05];
\draw[fill] (6,1) circle [radius=0.05];
\draw[fill] (4,2) circle [radius=0.05];
\draw[fill] (5,2) circle [radius=0.05];
\draw[fill] (6,2) circle [radius=0.05];
\draw[fill] (5,3) circle [radius=0.05];
\draw[fill] (5.5,2.5) circle [radius=0.05];
\draw (5,0) -- (4,1) -- (4,2) -- (5,3) -- (6,2) -- (6,1) -- (5,0);
\draw (4,2) -- (5,1) -- (5,0) -- (5,1) -- (6,2);
\draw (4,1) -- (5,2) -- (5,3) -- (5,2) -- (6,1);
\node at (6,0) {\footnotesize $\alg L_{14}$};
\draw[fill] (10,0) circle [radius=0.05];
\draw[fill] (8.5,1.5) circle [radius=0.05];
\draw[fill] (11.5,1.5) circle [radius=0.05];
\draw[fill] (10,3) circle [radius=0.05];
\draw[fill] (9.5,0.5) circle [radius=0.05];
\draw[fill] (10.5,0.5) circle [radius=0.05];
\draw[fill] (10,1) circle [radius=0.05];
\draw[fill] (9.5,2.5) circle [radius=0.05];
\draw[fill] (10.5,2.5) circle [radius=0.05];
\draw[fill] (10,2) circle [radius=0.05];
\draw (10,0) -- (8.5,1.5) -- (10,3) -- (11.5,1.5) -- (10,0);
\draw (9.5,2.5) -- (10,2) -- (10,1) -- (10,2) -- (10.5,2.5);
\draw (9.5,0.5)-- (10,1) -- (10.5,.5);
\node at (11,0) {\footnotesize $\alg L_{15}$};
\draw[fill] (2.5,5) circle [radius=0.05];
\draw[fill] (1,6.5) circle [radius=0.05];
\draw[fill] (4,6.5) circle [radius=0.05];
\draw[fill] (2.5,8) circle [radius=0.05];
\draw[fill] (2.5,6.5) circle [radius=0.05];
\draw[fill] (1.75,7.25) circle [radius=0.05];
\draw[fill] (3.25,7.25) circle [radius=0.05];
\draw (2.5,5) -- (1,6.5) -- (2.5,8) -- (4,6.5) -- (2.5,5);
\draw (1.75,7.25) -- (2.5,6.5) -- (2.5,5) -- (2.5,6.5) -- (3.25,7.25);
\node at (3.5,5) {\footnotesize $\alg L_{1}$};
\draw[fill] (7.5,5) circle [radius=0.05];
\draw[fill] (6,6.5) circle [radius=0.05];
\draw[fill] (9,6.5) circle [radius=0.05];
\draw[fill] (7.5,8) circle [radius=0.05];
\draw[fill] (7.5,6.5) circle [radius=0.05];
\draw[fill] (6.75,5.75) circle [radius=0.05];
\draw[fill] (8.25,5.75) circle [radius=0.05];
\draw (7.5,5) -- (6,6.5) -- (7.5,8) -- (9,6.5) -- (7.5,5);
\draw (6.75,5.75) -- (7.5,6.5)-- (7.5,8) -- (7.5,6.5) -- (8.25,5.75);
\node at (8.5,5) {\footnotesize $\alg L_{2}$};
\end{tikzpicture}
\caption{}\label{lattices}
\end{center}
\end{figure}

Let's make more progress: consider the rules
\begin{align*}
&x \meet y \app x \meet z \quad\Rightarrow\quad x \meet y \app x \meet (y \join z)\tag{$SD_\meet$}\\
&x \join y \app x \join z \quad\Rightarrow\quad x\join y \app x \join (y \meet z)\tag{$SD_\join$}.
\end{align*}
A lattice is {\em meet semidistributive} if it satisfies $SD_\meet$, {\em join semidistributive} if it satisfies
$SD_\meet$ and {\em semidistributive} if it satisfies both.
Clearly (meet/join) semidistributive lattices form  quasivarieties called $\mathsf{SD}_\meet$, $\mathsf{SD}_\join$ and $\mathsf{SD}$ respectively, and so
do their bounded versions.
It is a standard exercise to show that homomorphic
images of a finite (meet/join) semidistributive lattices are (meet/join) semidistributive.  It is also possible to show none of the three quasivariety (and their bounded versions) is a variety (see \cite{JipsenRose1992} p. 82 for an easy argument); they  are also not locally finite since for instance  $\alg F =\alg F_\mathsf{SD}(x,y,z)$ is infinite; hence $\alg F^b$ is a bounded infinite three-generated lattice and thus $\mathsf{SD}^b$ is not locally finite as well. A variety $\vv V$ of (bounded) lattices is (meet/join) semidistributive if $\vv V \sse \mathsf{SD}$ ($\vv V \sse \mathsf{SD}/\meet$\, /\,$\vv V \sse \mathsf{SD}/\join$).

We need a little bit of lattice theory.  A filter of $\alg L$ is an upset $F$ of $\alg L$ that is closed under meet; a filter is {\em prime} if
$ a \join b \in F$ implies $a \in F$ or $b \in F$. An {\em ideal} $I$ of $\alg L$ is the dual concept, i.e. a downset that is closed under join; an ideal is {\em prime} if $a \meet b \in I$ implies $a \in I$ or $b \in I$. The following lemma is straightforward.

\begin{lemma} If $F$ is a prime filter of $\alg L$ ($I$ is a prime ideal of $\alg L$), then $L \setminus F$ is a prime ideal of $\alg L$ ($\alg L\setminus I$ is a prime filter of $\alg L$).
\end{lemma}

\begin{lemma}  Any bounded (meet/join) semidistributive lattice is unifiable in the variety of bounded lattices.
\end{lemma}
\begin{proof} Let $\alg L$ be bounded and meet semidistributive. Since $\alg L$ is lower bounded by $0$ a standard application of Zorn Lemma yields a maximal proper filter $F$ of $\alg L$;  we claim that $F$ is also prime. Let $a,b \notin F$; then the filter generated by $F \cup \{a\}$ must be the entire lattice.
Hence there must be a $c \in F$ with $c \meet a =0$; similarly there must be a $d \in F$ with $d \meet b = 0$. Let $e = c \meet d$; then $e \in F$ and $e\meet a = e \meet b = 0$ and by meet semidistributivity
$e \meet (a \join b) = 0$. But if $a \join b \in F$, then $0 \in F$, a clear contradiction. Hence $a \join b \notin F$ and $F$ is prime.

Let now $\f: \alg L\Longrightarrow \mathbf 2^b$ defined by
$$
\f(x) = \left\{
          \begin{array}{ll}
            1, & \hbox{if $x \in F$;} \\
            0, & \hbox{if $x \notin F$.}
          \end{array}
        \right.
$$
Using the fact that $F$ is prime and $L \setminus F$ is prime it is straightforward to check that $\f$ is a homomorphism. Therefore $\alg L$ is unifiable.

A dual proof shows that the conclusion holds for join semidistributivity  and a fortiori for semidistributivity.
\end{proof}

\begin{proposition}
	Any bounded finite (meet/join) semidistributive lattice  is flat.
\end{proposition}
\begin{proof}  If $\alg L$ is finite and (meet/join) semidistributive, every lattice in $\HH\SU(\alg L)$ is finite and (meet/join) semidistributive. So it is unifiable and, if simple, it must be equal to $\alg 2^b$; therefore
$\alg L$ is flat.
\end{proof}

\begin{corollary}\label{cor: sdbl is ps} Every locally finite (meet/join) semidistributive variety of bounded lattices is passive structurally complete.
\end{corollary}

In \cite{Lee1985} several (complex) sets of equations implying semidistributivity are studied; one of them is useful to us, since it describes a class of locally finite varieties.
 The description is interesting in that involves some of the $\alg L_i's$ we have introduced before.

\begin{theorem} \cite{Lee1985} There exists a finite set $\Gamma$ of lattices equations such that, if $\vv V$ is any variety of lattices such that $\vv V \vDash \Gamma$, then the following hold:
\begin{enumerate}
\item $\vv V$ is semidistributive;
\item $\vv V$ is locally finite;
\item only $\alg L_{13}, \alg L_{14}, \alg L_{15} \in \vv V$.
\end{enumerate}
\end{theorem}

A variety satisfying $\Gamma$ is called {\em almost distributive} and it is straightforward to check that a similar result holds for varieties of bounded lattices.
Therefore:
\begin{proposition}
	Every almost distributive variety of bounded lattices is passive structurally complete.
\end{proposition}

We close this subsection with a couple of observations; first $\VV(\alg L_1^b,\alg L_2^b)$ is a variety of bounded lattices that
is passive structurally complete (by Theorem \ref{fgflat}) but neither meet nor join semidistributive. Next, what about infinite flat (bounded) lattices?
 We stress that in \cite{McKenzie1994} there are several examples of this kind and we believe that a careful analysis of the proofs therein could give some insight
  on how to construct a non locally finite variety of bounded lattices that it is passive structurally complete. But again, this is not a paper in lattice theory; therefore we defer this investigation.

\subsection{Substructural logics and residuated lattices}\label{sec:FL}

Originally, {\em substructural logics} were introduced as logics which, when formulated as Gentzen-style systems, lack some (including ``none'' as a special case) of the three basic {\em structural rules} (i.e.  exchange, weakening and contraction) of classical logic.  Nowadays, substructural logics are usually intended as those logics whose equivalent algebraic semantics are residuated structures, and they encompass most of the interesting non-classical logics: intuitionistic logic, basic logic, fuzzy logics, relevance logics and many other systems. Precisely, by substructural logics we mean here the axiomatic extensions of the Full Lambek calculus $\mathcal{FL}$ (see \cite{GJKO} for details and a survey on substructural logics).
All these logics are {\em strongly algebraizable}: their equivalent algebraic semantics are all {\em varieties} of $\mathsf{FL}$-algebras, particular residuated lattices that we shall now define.

A {\em residuated lattice} is an algebra  $\alg A = \la A,\join,\meet,\cdot,\lr,\rr, 1\ra$ where
\begin{enumerate}
\item $\la A, \join, \meet\ra $ is a lattice;
\item $\la A, \cdot,1\ra$ is a monoid;
\item $\lr$ and $\rr$ are the right and left divisions w.r.t. $\cdot$, i.e., $x \cdot y \leq z$ iff $y \leq x \rr z$ iff $x \leq z \lr y$, where $\leq$ is given by the lattice ordering.
\end{enumerate}
Residuated lattices form a variety $\mathsf{RL}$ and an equational axiomatization, together with many equations holding in these very rich structures, can be found in \cite{BlountTsinakis2003}.

A residuated lattice $\alg A$ is {\em integral} if it satisfies the equation $x \le 1$; it  is {\em commutative} if $\cdot$ is commutative, and in this case the divisions coincide: $x \backslash y = y / x$, and they are usually denoted with $x \to y$.
The classes of residuated lattices that satisfy any combination of integrality and commutativity are subvarieties of $\mathsf{RL}$. We shall call the variety of integral residuated lattices $\mathsf{IRL}$, commutative residuated lattices $\mathsf{CRL}$, and their intersection $\mathsf{CIRL}$.

Residuated lattices with an extra constant $0$ in the language are called $\mathsf{FL}$-algebras, since they are the equivalent algebraic semantics of the Full Lambek calculus $\mathcal{FL}$. Residuated lattices are then the equivalent algebraic semantics of $0$-free fragment of $\mathcal{FL}$, $\mathcal{FL}^{+}$. An $\mathsf{FL}$-algebra is \emph{$0$-bounded} if it satisfies the inequality $0 \leq x$ and the variety of zero-bounded $\mathsf{FL}$-algebras is denoted by $\mathsf{FL}_o$;
integral and $0$-bounded $\mathcal{FL}$-algebras are called $\mathcal {FL}_w$ algebras (since they are the equivalent algebraic semantics of the Full Lambek Calculus with weakening), and we call its commutative subvariety $\mathcal{FL}_{ew}$.

Restricting ourselves to the commutative case there is another interesting equation:
$$
(x \imp y) \join (y \imp x) \app 1.
$$
It can be shown (see \cite{BlountTsinakis2003} and \cite{JipsenTsinakis2002}) that a subvariety of $\mathsf{FL_{ew}}$ or $\mathsf{CIRL}$ satisfies the above equation if and only if any algebra therein is a subdirect product of totally ordered algebras, and this implies that all the subdirectly irreducible algebras are totally ordered. Such varieties are called {\em representable}  and the subvariety axiomatized by that equation  is the largest subvariety of $\mathsf{FL_{ew}}$ or $\mathsf{CIRL}$ that is representable. The representable subvariety of $\mathsf{FL}_{ew}$  is usually denoted by $\MTL$, since it is the equivalent algebraic semantics of Esteva-Godo's {\em Monoidal t-norm  based logic} \cite{EstevaGodo2001}.

\subsubsection{Active universal completeness}
We have already seen examples of subvarieties of $\mathsf{FL}_{ew}$-algebras that are active universally complete, but those were all locally finite subvarieties of $\mathsf{BL}$-algebras, that is, $\mathsf{MTL}$-algebras satisfying the divisibility equation: $x \land y = x (x \to y)$. In this section we will display a different class of examples.  If $\alg A$ is any algebra a congruence $\th \in \Con A$ is a {\bf factor congruence} if there is a $\th' \in \Con A$ such that $\th \join \th' = 1_\alg A$, $\th \meet \th' = 0_\alg A$ and $\th,\th'$ permute.
It is an easy exercise in general algebra to show that in this case $\alg A \cong \alg A/\th \times \alg A/\th'$; note that $1_\alg A$ and $0_\alg A$ are factor congruences that gives a trivial decomposition. A less known fact (that appears in \cite{Citkin2018a}) is:

\begin{lemma}\label{factor} Let $\alg A$ be any algebra and $\th$ a factor congruence; then $\alg A/\th$ is a retract of $\alg A$ if and only if there is a homomorphism $h: \alg A/\th \longrightarrow \alg A/\th'$.
\end{lemma}
\begin{proof} Suppose first that there is a homomorphism $h: \alg A/\th \longrightarrow \alg A/\th'$. 
Since $\alg A \cong \alg A/\th \times \alg A/\th'$ for $u \in A$, $u = (a/\th,b/\th')$, we set $f(u) = a/\th$; then $f:\alg A \longrightarrow \alg A/\th$ is clearly an epimorphism, since $(a/\th,a/\th') \in A$ for all $a \in A$.  Let
$$
g(a/\th) = (a/\th,h(a/\th)).
$$
One can check that $g$ is a homomorphism with standard calculations and clearly $fg= id_{\alg A/\th}$. Hence $\alg A/\th$ is a retract of $\alg A$.

Conversely suppose that $f,g$ witness a retraction from $\alg A/\th$ in $\alg A$; then if $g(a/\th)= (u/\th,v/\th')$, set $h(a/\th) = v/\th'$.  It is then easy to see that $h$ is a homomorphism and the thesis holds.
\end{proof}

Observe that in any $\mathsf{FL}$-algebra, every compact (i.e., finitely generated) congruence is principal; as a matter of fact if $\alg A$ is in $\mathsf{FL}$, $X = \{(a_1,b_1),\dots,(a_n,b_n)\}$ is a finite set of pairs from $A$  and
$p = \bigwedge_{i=1}^n [ (a_i \backslash b_i) \meet (b_i \backslash a_i) \meet 1]$ then $\cg_\alg A(X) = \cg_\alg A(p,1)$.

\begin{theorem}\label{thm: discr almost universal} Let $\vv Q$ be a quasivariety of $\mathsf{FL}_w$-algebras in which every principal congruence is a factor congruence; then $\vv Q$ has projective unifiers.
\end{theorem}
\begin{proof} Let $\alg F_\vv Q(X)/\th$ be a finitely presented unifiable algebra in $\vv Q$; then there is an onto homomorphism from $\alg F_\vv Q(X)/\th(\Sigma)$ to $\alg F_\vv Q = \alg 2$.
Now $\th = \th(\Sigma)$ is a principal congruence, hence it is a factor congruence with witness $\th'$, i.e. $\alg F_\vv Q(X) \cong \alg F_\vv Q(X)/\th \times \alg F_\vv Q(X)/\th'$. If $\th' = 1_\alg A$, then $\alg F_\vv Q(X) = \alg F_\vv Q(X)/\th$
and so it is projective. Otherwise $\alg F_\vv Q = \alg 2$ is embeddable in $\alg F_\vv Q(X)/\th'$; hence there is a homomorphism from from $\alg F_\vv Q(X)/\th$ to $\alg F_\vv Q(X)/\th'$. By Lemma \ref{factor} $\alg F_\vv Q(X)/\th$ is a retract of $\alg F_\vv Q(X)$, i.e. it is projective.
\end{proof}
So any quasivariety of $\mathsf{FL}_{w}$-algebras with the property that every principal congruence is a factor congruence is active universally complete (Theorem \ref{lemma: ex unif auc}); really it is active primitive universally complete, because
$\alg F_\vv Q$ is the two-element algebra for any quasivariety $\vv Q$ of $\mathsf{FL}_{w}$-algebras (Theorem \ref{thm: activeprimitive}).
We observe in passing that for any $\mathsf{FL}_{w}$ algebra  every factor congruence is principal; this is because every variety of $\mathsf{FL}_{w}$-algebras is Koll\'ar and
congruence distributive.
Discriminator varieties of $\mathsf{FL}_{ew}$-algebras have been completely described in \cite{Kowalski2005}; as a consequence we have:

\begin{theorem} For a variety $\vv V$ of $\mathsf{FL}_{ew}$-algebras the following are equivalent:
\begin{enumerate}
\item $\vv V$ is a discriminator variety;
\item $\vv V$ is semisimple, i.e. all the subdirectly irreducible members of $\vv V$ are simple;
\item there is an $n \in \mathbb N$ such that $\vv V \vDash x \join \neg x^n \approx 1$;
\item for any $\alg A \in \vv V$ every compact (i.e. principal) congruence is a factor congruence.
\end{enumerate}
\end{theorem}
\begin{proof} The equivalence of (1), (2) and (3) has been proved in \cite{Kowalski2005}. Assume then (1);  it is well-known that in every discriminator variety every principal congruence is a factor congruence. In fact if $\vv V$ is a discriminator variety with discriminator term $t(x,y,z)$ let for any $\alg A \in \vv V$ and $a,b \in A$
\begin{align*}
&\th_\alg A(a,b) = \{(u,v): t(a,b,u)=t(a,b,v)\}  \\
&\gamma_\alg A(a,b)= \{(u,v): t(a,t(a,b,u),u) = t(a,t(a,b,v),v)\}.
\end{align*}
Using the properties of the discriminator term it is easy to verify that they are congruences and the complement of each other; since discriminator varieties are congruence permutable they are factor congruences and (4) holds.

Conversely assume (4) and let $\alg A$ be a subdirectly irreducible member of $\vv V$. Let $\mu_\alg A$ be the minimal nontrivial congruence of $\alg A$; then $\mu_\alg A$ is principal, so it must be a factor congruence. This is
possible if and only if $\mu_\alg A=1_\alg A$; therefore $\alg A$ is simple, and $\vv V$ is semisimple.
\end{proof}
\begin{corollary}
	Every discriminator (or, equivalently, semisimple) variety $\vv V$ of $\mathsf{FL}_{ew}$-algebras is active primitive universal.
\end{corollary}
We observe that Theorem \ref{thm: discr almost universal} does not add anything as far as $\mathsf{BL}$-algebras are concerned; in fact any discriminator variety of $\mathsf{FL}_{ew}$-algebras must satisfy $x^n \app x^{n+1}$ for some $n$ \cite{Kowalski2005} and the varieties of $\mathsf{BL}$-algebras with that property are exactly the locally finite varieties, which we already pointed out are active universally complete.

\subsubsection{Passive  structural completeness}
A particularly interesting application of our characterization of passive structurally complete varieties is in the subvariety of integral and $0$-bounded FL-algebras.
Let us rephrase Theorem \ref{thm:PUC} in this setting. First, using residuation it is easy to see that every finite set of identities in $\mathsf{FL}$ is equivalent to a single identity. Moreover, in every subquasivariety $\vv Q$ of $\mathsf{FL}_w$, the smallest free algebra $\alg F_{\vv Q}$ is the 2-element Boolean algebra $\alg 2$, and its generated quasivariety is the variety of Boolean algebras.
\begin{corollary}\label{cor:psctriv}
	Let $\vv Q$ be a quasivariety of $\mathsf{FL}_w$-algebras, then the following are equivalent:
	\begin{enumerate}
		\item $\vv Q$ is passive structurally complete;
	\item every trivializing identity in the variety of Boolean algebras is trivializing in $\vv Q$;
	\item every nontrivial finitely presented algebra is unifiable.
	\item every nontrivial algebra is unifiable.
	\end{enumerate}
\end{corollary}

The previous corollary has a possibly more transparent shape from the point of view of the logics.
Let us call a formula $\varphi$ in the language of $\mathsf{FL}$-algebras \emph{explosive in a logic $\cc L$}, with consequence relation $\vdash_{\cc L}$, if $\varphi \vdash_{\cc L} \delta$ for all formulas $\delta$ in the language of $\cc L$. Moreover, we call $\varphi$ a \emph{contradiction in $\cc L$} if $\varphi \vdash_{\cc L} 0$. Since $\mathsf{FL}_w$-algebras are $0$-bounded, it is clear that contradictions coincide with explosive formulas in all axiomatic extensions of $\cc F\cc L_w$.

\begin{corollary}\label{cor:PUClogic}
	Let $\mathcal{L}$ be an axiomatic extension of $\mathcal{FL}_w$, then the following are equivalent:
	\begin{enumerate}
		\item $\mathcal{L}$ is passively structurally complete.
	\item Every contradiction of classical logic is explosive in $\mathcal{L}$.
	\item Every passive rule of $\mathcal{L}$ has explosive premises.
	\end{enumerate}
\end{corollary}

Let us first explore the consequences of the equivalence between (1) and (2) in Corollary \ref{cor:PUClogic}. It is well known that intuitionistic logic is passively structurally complete (reported by Wronski at the 51st Conference on the History of Logic, Krakow, 2005). This is easily seen by Corollary \ref{cor:PUClogic}, indeed, observe that any contradiction of classical logic $\varphi$ is such
that its negation $\neg \varphi$ is a theorem of classical logic. Using the Glivenko translation and the deduction theorem, we obtain that $\varphi$ is explosive in intuitionistic logic as well, which is then passively structurally complete.
We will now show how this argument can be extended to a wide class of logics.

Let us write the negations corresponding to the two divisions as $\neg x = x \backslash 0$ and $\sim x = 0 / x$.
Following \cite{GalatosOno2006,GalatosOno2009}, we say that two logics $\mathcal{L}_1$ and $\mathcal{L}_2$ are \emph{Glivenko equivalent} if for all formuals $\varphi$: $$\vdash_{\mathcal{L}_1} \neg \varphi  \;\;\mbox{ iff }\;\; \vdash_{\mathcal{L}_2} \neg \varphi$$(equivalently, $\vdash_{\mathcal{L}_1} \sim \varphi  \mbox{ iff } \vdash_{\mathcal{L}_2} \sim \varphi$).
Given a logic $\cc L$, we call \emph{Glivenko logic of $\cc L$} the smallest substructural logic that is Glivenko equivalent to $\cc L$ . Moreover, we call \emph{Glivenko logic of $\cc L$ with respect to $\cc L'$}, and denote it with $\cc G_{\cc L'}(\cc L)$ the smallest extension of $\cc L'$ that is Glivenko equivalent to $\cc L$ (all these notions make sense by the results in \cite{GalatosOno2006,GalatosOno2009}). $\cc G_{\cc L'}(\cc L)$ is axiomatized relatively to $\cc L'$ by the set of axioms $\{ \neg \sim \varphi : \vdash_{\cc L}\varphi\}$, or equivalently by the set $\{  \sim \neg \varphi : \vdash_{\cc L}\varphi\}$. 

Here we are interested in the Glivenko equivalent of classical logic with respect to $\mathcal{FL}_w$. From the algebraic perspective, this corresponds to the largest subvariety of $\mathsf{FL}_w$ that is Glivenko equivalent to Boolean algebras, $\mathsf{G}_{\mathsf{FL}_w}(\vv B)$. The latter is axiomatized in \cite[Corollary 8.33]{GJKO} as the subvariety of $\mathsf{FL}_w$ satisfying:
\begin{enumerate}
	\item $\sim (x \land y) = \sim (xy)$
	\item $\sim (x \backslash y) = \sim (\neg x \lor y)$
	\item $\neg (x \backslash y) = \neg (\sim x \lor y)$
	\item $\sim(x \backslash y) = \sim (\neg \sim x \backslash \neg \sim y)$
	\item $\sim(x / y) = \sim (\neg \sim x / \neg \sim y)$.
\end{enumerate}

\begin{theorem}
	Every axiomatic extension $\cc L$ of the Glivenko logic of classical logic with respect to $\mathcal{FL}_w$ is passively structurally complete.
\end{theorem}
\begin{proof}
Consider a contradiction of classical logic $\varphi$, by the deduction theorem $\vdash_{\cc C\cc L} \neg \varphi$ (where $\vdash_{\cc C\cc L}$ is the consequence relation of classical logic). Since $\cc L$ is Glivenko equivalent to classical logic, $\vdash_{\mathcal{L}} \neg \varphi$.
It can be easily checked that this implies that $\varphi \vdash_{\mathcal{L}} 0$ (it is a consequence of the parametrized local deduction theorem which holds in every extension of $\mathcal{FL}$ \cite{GJKO}, but it is also straightforward to see in models). Thus $\varphi$ is a contradiction of $\cc L$, or equivalently it is explosive in $\cc L$, which is then passively structurally complete by Corollary \ref{cor:PUClogic}.
\end{proof}
Thus, every subvariety of $\mathsf{G}_{\mathsf{FL}_w}(\vv B)$ is passive structurally complete. In particular, the commutative subvariety $\mathsf{G}_{\mathsf{FL}_{ew}}(\vv B)$ is the variety of pseudocomplemented $\mathsf{FL}_{ew}$-algebras (\cite{FazioStJohn2022}), axiomatized by $$x \land \neg x \approx 0.$$
Examples of passive structurally complete varieties then include Heyting algebras, Stonean MTL-algebras and as a consequence, e.g., product algebras and G\"odel algebras.

We observe that these are not all of the passive structural complete varieties of $\mathsf{FL}_{w}$ (nor of $\mathsf{FL}_{ew}$). Let us indeed obtain a different kind of examples.
\begin{definition}
	We say that a variety $\vv V$ has a \emph{Boolean retraction term} if there exists a term $t$ in the language of residuated lattices (i.e., $0$-free) such that, for every $\alg A \in \vv V$, $t$ defines an idempotent endomorphism on $\alg A$ whose image is the Boolean skeleton of $\alg A$, that is, the set of complemented elements of $\alg A$. 
\end{definition}
Varieties with a Boolean retraction term have been studied at length by Cignoli and Torrens in a series of papers, see in particular \cite{CignoliTorrens2012}.  
These are all varieties in which all nontrivial algebras retract onto a nontrivial Boolean algebra, thus they satisfy the hypotheses of Corollary \ref{cor:psctriv} and they are passive structurally complete. 
Some of these varieties have been shown in \cite{AglianoUgolini2022a} to have projective unifiers, thus they satisfy Theorem \ref{thm:unifiable2} and they are non-negative universally complete.
Among those we cite some varieties of interest in the realm of many-valued logics: the variety of product algebras, the variety generated by perfect MV-algebras, the variety $\mathsf{NM^-}$ of nilpotent minimum algebras without negation fixpoint and some varieties that have been called {\em nilpotent product} in \cite{AglianoUgolini2019a} or \cite{AglianoUgolini2019b}.

We will see that in the representable variety of $\mathsf{FL}_{ew}$, $\mathsf{MTL}$, we can fully characterize passive structurally complete varieties as those with a Boolean retraction term.
By \cite{CignoliTorrens2012}, the largest subvariety of $\MTL$ with a Boolean retraction term is axiomatized relatively to $\MTL$ by the Di Nola-Lettieri equation:
\begin{equation}
	(x + x)^2 = x^2 + x^2 \tag{DL}
\end{equation}
where $x + y = \neg(\neg x \cdot \neg y)$. The latter identity has been introduced by Di Nola and Lettieri to axiomatize within MV-algebras the variety generated by the Chang algebra.
This variety is called $\mathsf{sDL}$ in \cite{Ugolini2018} ($\mathsf{BP}_0$ in \cite{Noguera2007,AFU2017}), and it includes, for instance: 
pseudocomplemented MTL-algebras (also called SMTL-algebras), and thus G\"odel algebras and product algebras; involutive $\mathsf{BP}_0$-algebras and thus the variety generated by perfect MV-algebras and nilpotent minimum algebras without negation fixpoint.

 Let us say that an element of an $\mathsf{FL}_{ew}$-algebra $\alg A$ has {\em finite order $n$} if $x^n = 0$, and {\em infinite order} if there is no such $n$. We call \emph{perfect} an algebra $\alg A \in \mathsf{FL}_{ew}$ such that, for all $a \in A$, $a$ has finite order if and only if $\neg a$ has infinite order. Now, $\mathsf{sDL}$ turns out to be the variety generated by the perfect chains (see \cite{Ugolini2018,AFU2017}).
\begin{lemma}\label{lemma:perfectchain}
	A chain $\alg A \in \mathsf{FL}_{ew}$ is perfect if and only if there is no element with finite order $ a \in A$ such that $a \geq \neg a$.
\end{lemma}
\begin{proof}
	By order preservation, if there is an element $ a \in A, a \geq \neg a, a^n = 0$, then both $a$ and its negation have finite order, thus the chain is not perfect. Suppose now a chain $\alg A$ is not perfect. Observing that for every element $x \in A$ it cannot be that both $x$ and $\neg x$ gave infinite order, we get that there is an element $a \in A$ such that both $a$ and its negation $\neg a$ have finite order. If $a \not\geq \neg a$, since $\alg A$ is a chain, $a < \neg a$. Then $\neg\neg a \leq \neg a$, and they both have finite order.
\end{proof}
\begin{theorem}\label{thm: MTL collapse}
	For a  subvariety $\vv V$ of $\mathsf{MTL}$ the following are equivalent:
\begin{enumerate}
\item $\vv V$ is passive structurally complete;
\item $\vv V$ is a subvariety of $\mathsf{sDL}$.
\end{enumerate}
\end{theorem}
\begin{proof} 
Since subvarieties of $\mathsf{sDL}$ have a Boolean retraction term (2) implies (1) by Corollary \ref{cor:psctriv}.
Suppose now that $\vv V \nsubseteq \mathsf{sDL}$. Then there is a chain $\alg A$ in $\vv V$ that is not perfect. By Lemma \ref{lemma:perfectchain}, there exists $a \in A, a \geq \neg a, a^n = 0 \mbox{ for some } n \in \mathbb{N}.$
	 Thus, $\neg(a \lor \neg a)^n = 1$.
	 But the identity $\neg(x \lor \neg x)^n = 0$ holds in Boolean algebras. Thus $\neg(x \lor \neg x)^n \approx 1$ is trivializing in Boolean algebras but not in $\vv V$. By Corollary \ref{cor:psctriv}, $\vv V$ is not passive structurally complete and thus (1) implies (2). 
\end{proof} 

Notice that the previous theorem also implies that a variety of MTL-algebras that is not a subvariety of $\mathsf{sDL}$ cannot be structurally complete.

We mention that structural completeness in subvarieties of $\mathsf{MTL}$ (or their logical counterparts) has been studied by several authors: e.g., \cite{Wojtylak1976} and \cite{Gispert2016} for  \L ukasiewicz logics,  \cite{DzikWronski1973} G\"odel logic, and \cite{CintulaMetcalfe2009} for fuzzy logics in the MTL framework; in the latter the authors show for instance that all subvarieties of pseudocomplemented MTL-algebras ($\mathsf{SMTL}$) are passive structurally complete. This result is here obtained as a consequence of Theorem \ref{thm: MTL collapse}, since $\mathsf{SMTL}$ is a subvariety of $\mathsf{sDL}$. From the results mentioned above and the characterzation theorem, it also follows that the only varieties of MV-algebras (the equivalent algebraic semantics of infinite-valued \L ukasiewicz logic) that are structurally complete are Boolean algebras and the variety generated by perfect MV-algebras (this result has been obtained following a different path in \cite{Gispert2016}).

We also remark that a variety of $\mathsf{FL}_{ew}$-algebras can be at most non-negative universally complete since trivial algebras are finitely presented and not unifiable (unifiability is a necessary condition for universal completeness by Theorem \ref{lemma: FQ trivial}); by Proposition \ref{prop: nnu iff ps and au} this happens if and only if the variety is active universally complete and passive structurally complete.
Thus, for instance, a semisimple variety of $\mathsf{FL}_{ew}$-algebras satisfying the conditions in Corollary \ref{cor:psctriv} would be non-negative universally complete.
We stress that this observation  is not of particular interest in MTL-algebras, since the only discriminator variety in $\mathsf{sDL}$ is the variety of Boolean algebras. Indeed, consider a chain $\alg A$ in a discriminator variety $\vv V$ in $\mathsf{sDL}$. Then there is some $n \in \mathbb{N}$ such that $\vv V \models x \lor \neg x^n \approx 1$. Let now $a \in A$; either $a$ has finite order, and then from $a \lor \neg a^n$ we obtain that $a = 1$, or $a$ has infinite order, and then $\neg a$ has finite order. So by the analogous reasoning $\neg a = 1$. Therefore $\alg A$ is the two-element chain, and $\vv V$ is the variety of Boolean algebras.

\section{Conclusions}\label{sec: conclusions}

 In  Figure \ref{classes} we display several classes of varieties that we have considered in this paper (and the labels should be self explanatory); we are dropping the hereditary subclasses to avoid clutter. Observe that this is really a meet semilattice under inclusion.

\begin{figure}[htbp]
\begin{center}
\begin{tikzpicture}[scale=1.3]
\node at (0,0) {\footnotesize $U$};
\node at (-1,1) {\footnotesize  $NNU$};
\node at (-2,2) {\footnotesize  $AU$};
\node at (0,2) { \footnotesize $S$};
\node at (1,1) {\footnotesize $S+PU$};
\node at (2,2) {\footnotesize $PU$};
\node at (1,3) {\footnotesize  $PS$};
\node at (-1,3) {\footnotesize $AS$};
\draw  (0.2,0.2)-- (0.8,0.8);
\draw  (-0.2,0.2)-- (-0.8,0.8);
\draw  (1.2,1.2)-- (1.8,1.8);
\draw  (-1.2,1.2)-- (-1.8,1.8);
\draw  (0.8,1.2)-- (0.2,1.8);
\draw  (1.8,2.2)-- (1.2,2.8);
\draw  (-.2,2.2)-- (-0.8,2.8);
\draw  (-1.8,2.2)-- (-1.2,2.8);
\draw  (-.8,1.2)-- (-.2,1.8);
\draw  (.2,2.2)-- (.8,2.8);
\end{tikzpicture}
\caption{The universal and structural completeness classes.}\label{classes}
\end{center}
\end{figure}
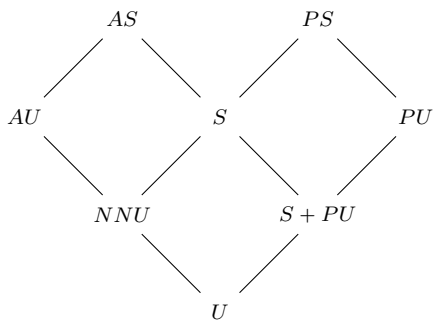

Almost all the classes are provably distinct.
\begin{enumerate}
\item The variety of bounded distributive lattice is structurally complete (Proposition \ref{proo: bdl structural}) but it is neither passive universally complete, since it is Koll\'ar and the least free algebra is not trivial,
nor  non-negative universally complete (Proposition \ref{prop: bdl not nnu}). Hence $S \ne NNU,S+PU$.
\item The variety of Boolean algebras is non-negative universally complete but not universally complete (Example \ref{ex: boolean algebras}) so $NNU\ne U$.
\item Any locally finite variety of $\mathsf{BL}$-algebras is active universally complete and some of them are not non-negative universally complete (Example \ref{ex: BL-algebras au}), so $AU \ne NNU$.
\item The variety  in Example \ref{ex: as not au} is active structurally complete but not active universally complete, hence $AS \ne AU$.
\item Any locally finite variety of bounded semidistributive lattices different from the distributive variety is passive structurally complete (Corollary \ref{cor: sdbl is ps}) but not structurally complete, since the only
structurally complete variety of bounded distributive lattices is the distributive variety (Proposition \ref{proo: bdl structural}); as above it is also not passive universally complete.
Hence $PS \ne S, PU$.
\item The variety $\VV(\alg M_{3,3}^+)$ (Section \ref{sec:lattices}) is passive universally complete, as any variety of lattices, but it is not structurally complete since $\QQ(\alg M_{3,3}^+) \nsubseteq \VV(\alg M_{3,3}^+)$; hence
$PU \ne S+PU$.
\item  Example 7.11 in \cite{DzikStronkowski2016} shows that $AS \ne S$.
\end{enumerate}

Moreover for the primitive counterparts:
\begin{enumerate}
\item the variety $\VV(\alg F)$ generated by the Fano lattice is structurally complete and passive universally complete but not primitive (Section \ref{sec:lattices}).
\item the variety of De Morgan lattices (Example \ref{ex: De Morgan}) is active universally complete but not active primitive universal.
\item the variety of injective monounary algebras is active structurally complete but not active primitive structural (Example 7.2 in \cite{DzikStronkowski2016}).
\end{enumerate}

There are three examples that we were not able to find,  which would guarantee total separation of all the classes we have considered:
\begin{enumerate}
\item A (quasi)variety that is structurally complete and passive universally complete, but not universally complete.
\item A non-negative universally complete (quasi)variety such that not all subquasivarieties are non-negative universally complete.
\item A universally complete variety which is not primitive universal.
\end{enumerate}
The natural example for (3) would be a locally finite variety with exact  unifiers having a subvariety without exact unifiers.  However we are stuck because of lack of examples: we have only one unifiable locally finite variety with exact (non projective) unifiers, i.e.  the variety of distributive lattices, which is trivially primitive universal.  A similar situation happens for (2); all the examples of non-negative universally complete varieties we have are either equationally complete and congruence distributive (so they do not have nontrivial subquasivarieties), or else are active universally complete just by consequence of their characterization (such as the subvarieties of $\mathsf{FL}_{ew}$ in Section \ref{sec:FL}). Then we have Stone algebras that are not equationally complete but the  only nontrivial subvariety is the variety of Boolean algebras, that is non-negative universally complete. Now from Corollary \ref{cor: VQ pu Q pu} it is immediate that every subquasivariety of $\mathsf{ST}$ is non-negative universally complete.  In conclusion a deeper investigation of universally complete and non-negative universally complete varieties is needed.

For (1) the situation is (slightly) easier to tackle: any primitive variety of lattices that is not universally complete gives a counterexample. While it seems impossible that all the primitive varieties in Section \ref{sec: primlattices} are universally complete, actually proving that one it is not does not seem easy. This is due basically to the lack of information on free algebras in specific varieties of lattices, such as for instance $\VV(\alg M_3)$; note that this variety is locally finite and hence all the finitely generated free algebras are finite. But we are not aware of any characterization.

\providecommand{\bysame}{\leavevmode\hbox to3em{\hrulefill}\thinspace}
\providecommand{\MR}{\relax\ifhmode\unskip\space\fi MR }
\providecommand{\MRhref}[2]{%
  \href{http://www.ams.org/mathscinet-getitem?mr=#1}{#2}
}
\providecommand{\href}[2]{#2}

\end{document}